\documentclass[final]{siamart220329}
\usepackage{t1enc}
\usepackage{amssymb}
\usepackage{bm}
\usepackage{booktabs}
\usepackage[final]{listings}
\usepackage{mathtools}
\usepackage{multirow}
\usepackage{paralist}
\usepackage{textcomp}
\usepackage{wasysym}
\usepackage{algorithmic}

\newsiamremark{remark}{Remark}
\title{Vectorization of a thread-parallel\\Jacobi singular value decomposition method%
\thanks{This work has been supported in part by Croatian Science
  Foundation under the project IP--2014--09--3670
  ``Matrix Factorizations and Block Diagonalization Algorithms''
  (\href{https://web.math.pmf.unizg.hr/mfbda}{MFBDA}).  The prototype
  implementation is available in a GitHub repository
  \url{https://github.com/venovako/VecJac}.}}
\author{Vedran Novakovi\'{c}%
\thanks{ORCID:~\url{https://orcid.org/0000-0003-2964-9674};
  completed a part of this research as an independent collaborator on the MFBDA project,
  10000 Zagreb, Croatia (\href{mailto:venovako@venovako.eu}{venovako@venovako.eu})}}
\headers{Vectorization of the Jacobi SVD}{V.~Novakovi\'{c}}
\begin{document}
\lstloadlanguages{C}
\lstset{language=C,extendedchars=false,numbers=left,numberstyle=\tiny,frame=lines,basicstyle=\small\ttfamily,columns=fullflexible,texcl=false,mathescape=true,lineskip=.25\baselineskip}
\maketitle
\begin{abstract}\looseness=-1
  The eigenvalue decomposition (EVD) of (a batch of) Hermitian
  matrices of order two has a role in many numerical algorithms, of
  which the one-sided Jacobi method for the singular value
  decomposition (SVD) is the prime example.  In this paper the batched
  EVD is vectorized, with a vector-friendly data layout and the
  AVX-512 SIMD instructions of Intel CPUs, alongside other key
  components of a real and a complex OpenMP-parallel Jacobi-type SVD
  method, inspired by the sequential \texttt{xGESVJ} routines from
  LAPACK\@.  These vectorized building blocks should be portable to
  other platforms that support similar vector operations.
  Unconditional numerical reproducibility is guaranteed for the
  batched EVD, sequential or threaded, and for the column
  transformations, that are, like the scaled dot-products, presently
  sequential but can be threaded if nested parallelism is desired.  No
  avoidable overflow of the results can occur with the proposed EVD or
  the whole SVD\@.  The measured accuracy of the proposed EVD often
  surpasses that of the \texttt{xLAEV2} routines from LAPACK\@.  While
  the batched EVD outperforms the matching sequence of \texttt{xLAEV2}
  calls, speedup of the parallel SVD is modest but can be improved and
  is already beneficial with enough threads.  Regardless of their
  number, the proposed SVD method gives identical results, but of
  somewhat lower accuracy than \texttt{xGESVJ}.
\end{abstract}
\section{Introduction}\label{s:1}
\looseness=-1
The eigenvalue decomposition (EVD) of a Hermitian or a symmetric
matrix of order two~\cite{Jacobi-1846} is a part of many numerical
algorithms, some of which are implemented in the
LAPACK~\cite{Anderson-et-al-99} library, like the one-sided
Jacobi-type
algorithm~\cite{Drmac-97,Drmac-Veselic-08a,Drmac-Veselic-08b} for the
singular value decomposition (SVD) of general matrices.  It also
handles the $2\times 2$ terminal cases in the
QR-based~\cite{Francis-61,Francis-62} and the
MRRR-based~\cite{Dhillon-Parlett-03,Dhillon-Parlett-04} algorithms for
the eigendecomposition of Hermitian/symmetric matrices.  Its direct
application, the two-sided Jacobi-type EVD method for
symmetric~\cite{Jacobi-1846} and Hermitian~\cite{Hari-BegovicKovac-21}
matrices, has not been included in LAPACK but is widely known.

The first part of the paper aims to show that reliability of the
$2\times 2$ symmetric/Hermitian EVD, produced by the LAPACK routines
\texttt{xLAEV2}, can be improved by scaling the input matrix $A$ by an
easily computable power of two.  This preprocessing not only prevents
the scaled eigenvalues from overflowing (and underflowing if possible),
but also preserves the eigenvectors of a complex $A$ from becoming
inaccurate and non-orthogonal in certain cases of subnormal components
of $a_{12}=\bar{a}_{21}$.  The proposed EVD formulas are branch-free,
implying no \texttt{if-else} statements, but relying instead on the
standard-conformant~\cite{IEEE-754-2019} handling of the special
floating-point values by the $\min\!/\!\max$ functions.  If the
formulas are implemented in the SIMD fashion, they can compute several
independent EVDs in one go, instead of one by one in a sequence.

An algorithm that performs the same operation on a collection (a
``batch'') of inputs of the same or similar (usually small) dimensions
is known as batched (see, e.g.,~\cite[Sect.~10]{Abdelfattah-et-al-16}
and~\cite{Abdelfattah-et-al-21}).  A vectorized, batched EVD of
Hermitian matrices of order two is thus proposed, in a similar vein as
the batched $2\times 2$ SVD from~\cite{Novakovic-20}.  Recall that a
SIMD vector instruction performs an operation on groups (called
vectors) of scalars, packed into the lanes of vector registers, at
roughly the cost of one (or few) scalar operation.  For example, two
vectors, each with $\mathtt{s}$ lanes, can be multiplied, each lane of
the first vector by the corresponding lane of the second one, at a
fraction of the time of $\mathtt{s}$ scalar multiplications.  If a
batch of input and output matrices is represented as a set of separate
data streams, each containing the same-indexed elements of a
particular sequence of matrices (e.g., all $a_{11}^{(\ell)}$), then a
stream can be handled as one or more vectors, depending on the
hardware's vector width $\mathtt{s}$, such that, e.g.,
$a_{11}^{(\ell)}$ is held in the $(\ell\bmod\mathtt{s})$-th lane of
the $\lfloor\ell/\mathtt{s}\rfloor$-th vector.  Every major CPU
architecture, like those from Intel, AMD, IBM, and ARM, offers vector
operations with varying, but ever expanding\footnote{A reader
interested in porting the proposed method to another platform is
advised to consult the up-to-date architecture manuals for the
vectorization support of a particular generation of CPUs.} vector
widths and instruction subsets.  Also, there are specialized vector
engines, like NEC SX-Aurora TSUBASA with 2048-bit-wide registers.
Thus, it is not a question should an algorithm be vectorized, but how
to do that, if possible.

The second part of the paper focuses on vectorization of the one-sided
Jacobi-type SVD for general matrices (but with at least as many rows
as there are columns), in its basic form~\cite{Drmac-97}, as
implemented in the \texttt{xGESVJ} LAPACK routines, without the
rank-revealing QR factorization and other
preprocessing~\cite{Drmac-Veselic-08a,Drmac-Veselic-08b} of the more
advanced \texttt{xGEJSV} routines.  Both sets of routines are
sequential by design, even though the used BLAS/LAPACK subroutines
might be parallelized.  The Jacobi-type method is not the fastest SVD
algorithm available, but it provides the superior accuracy of the
singular vectors and high relative accuracy of the singular
values~\cite{Demmel-Veselic-92}.  It is shown how the whole method
(but with a parallel pivot strategy), i.e., all of its components, can
be vectorized with the Intel AVX-512 instruction subset.  Those
BLAS-like components that are strategy-agnostic can be retrofitted to
the \texttt{xGESVJ} routines.

\looseness=-1
Certain time-consuming components of the Jacobi-type SVD method, like
the column pair updates, have for long been considered for
vectorization~\cite{deRijk-89}.  In fact, any component realized with
calls of optimized BLAS or LAPACK routines is automatically vectorized
if the routines themselves are, such as dot-products and
postmultiplications of a column pair by a rotation-like
transformation.  In the sequential method, however, only one
transformation is generated per a method's step, in essence by the EVD
of a pivot Grammian matrix of order two, while a parallel method
generates up to $\left\lfloor n/2\right\rfloor$ transformations per
step, with $n$ being the number of columns of the iteration matrix.
This sequence of independent EVDs is a natural candidate for
vectorization that computes several EVDs at once, and for its further
parallelization, applicable for larger $n$.  Also, some vector
platforms provide no direct support for complex arithmetic with
vectors of complex values in the typical, interleaved
$z=(\Re{z},\Im{z})$ representation, and thus the split
one~\cite{VanZee-Smith-17}, with an accompanying set of BLAS-like
operations, is convenient.  In the split representation, as explained
in \cref{sss:2.4.1}, a complex array is kept as two real arrays, such
that the $i$th complex element $z_i$ is represented by its real
($\Re{z_i}$) and imaginary ($\Im{z_i}$) parts, stored separately in
the corresponding real arrays.

\looseness=-1
Batched matrix factorizations (like the tall-and-skinny QR) and
decompositions (like a small-order Jacobi-type SVD) in the context of
the blocked Jacobi-type SVD on GPUs have been developed
in~\cite{Novakovic-15,Boukaram-et-al-18}, while the batched
bidiagonalization on GPUs has been considered in~\cite{Dong-et-al-18}.
Efficient batched kernels are ideal for acceleration of the blocked
Jacobi-type SVD methods, but the non-blocked (pointwise) ones induce
only batches of the smallest, $2\times 2$ EVD problems for the
one-sided, and the SVD ones for the
two-sided~\cite{Kogbetliantz-55,Novakovic-20} methods, what motivated
developing of the proposed batched EVD\@.

The two mentioned parts of the paper further subdivide as follows.  In
\cref{s:2} the vectorized batched EVD of Hermitian matrices of order
two is developed and its numerical properties are assessed.  In
\cref{s:3} the robust principles of (re-)scaling of the iteration
matrix in the SVD method are laid out, such that the matrix can never
overflow under inexact transformations, and a majority of the
remaining components of the method is vectorized.  In \cref{s:4} the
developed building blocks are put together to form a vectorized
one-sided Jacobi-type SVD method (with a parallel strategy), that can
be executed single- or multi-threaded, with identical outputs
guaranteed in either case.  Numerical testing is presented in
\cref{s:5}.  The paper concludes with a summary and some directions
for further research in \cref{s:6}.  Appendix contains most proofs and
additional algorithms, code, and numerical results.
\section{Vectorization of eigendecompositions of order two}\label{s:2}
Let $A$ be a symmetric (Hermitian) matrix of order two, $U$ an
orthogonal (unitary) matrix of its eigenvectors, i.e., its
diagonalizing Jacobi rotation, real~\cite{Jacobi-1846} or
complex~\cite{Hari-BegovicKovac-21}, and $\Lambda$ a real diagonal
matrix of its eigenvalues.  In the complex case,
\begin{equation}
  A=U\Lambda U^{\ast},\quad
  U=\begin{bmatrix}
  \cos\varphi & -\mathrm{e}^{-\mathrm{i}\alpha}\sin\varphi\\
  \mathrm{e}^{\mathrm{i}\alpha}\sin\varphi & \cos\varphi
  \end{bmatrix},\quad
  \Lambda=\begin{bmatrix}
  \lambda_1 & 0\\
  0 & \lambda_2
  \end{bmatrix},
  \label{e:1}
\end{equation}
with $\varphi\in\left[-\pi/4,\pi/4\right]$ and
$\alpha\in\left\langle-\pi,\pi\right]$, while $\alpha=0$ in the real
case.  The angles $\varphi$ and $\alpha$ are defined in terms of the
elements of $A$, as detailed in \cref{ss:2.1}.  From~\cref{e:1} it
follows by two matrix multiplications that
\begin{equation}
  A=\cos^2\varphi
  \begin{bmatrix}
    \lambda_1+\lambda_2\tan^2\varphi & \mathrm{e}^{-\mathrm{i}\alpha}\tan\varphi(\lambda_1-\lambda_2)\\
    \mathrm{e}^{\mathrm{i}\alpha}\tan\varphi(\lambda_1-\lambda_2) & \lambda_1\tan^2\varphi+\lambda_2
  \end{bmatrix}.
  \label{e:2}
\end{equation}

Should $A$ be constructed from~\cref{e:2}, e.g., for testing purposes,
with its eigenvalues prescribed, then it suffices to ensure that
$|\lambda_1|+|\lambda_2|\le\nu/8$, where $\nu$ is the largest
finite floating-point value, to get
$\max|a_{ij}|\le\nu/(4\sqrt{2})$.  In \cref{p:s2} it is shown that
such a bound on the magnitudes of the elements of $A$ guarantees that
its eigenvalues will not overflow.  If the bound does not hold, $A$
has to be downscaled to compute $\Lambda$.

One applicable power-of-two scaling algorithm was proposed
in~\cite[subsection~2.1]{Novakovic-20} in the context of the SVD of a
general real or complex $2\times 2$ matrix, and is adapted for the EVD
of a symmetric or Hermitian matrix $A$ of order two in \cref{ss:2.3}.

Contrary to the standalone EVD, for orthogonalization of a pivot
column pair in the Jacobi SVD algorithm it is sufficient to find the
eigenvectors $U$ of their pivot Grammian matrix $A$, while its
eigenvalues in $\Lambda$ are of no importance.  If the orthogonalized
pivot columns are also to be ordered in the iteration matrix
non-increasingly with respect to their Frobenius norms, a permutation
$P$ (as in, e.g.,~\cite{Novakovic-15}) has to be found such that
$\lambda_1'\ge\lambda_2'\ge 0$, where $\Lambda'=P^TU^{\ast}AUP$, i.e.,
the pivot columns have to be transformed by multiplying them from the
right by $UP$ instead of by $U$.  For a comparison of the eigenvalues
of $A$ to be made, it is only required that the eigenvalue with a
smaller magnitude is finite, while the other one may be allowed to
overflow.
\subsection{Branch-free computation of the Jacobi rotations}\label{ss:2.1}
Assume that $A$, represented by its lower triangle elements $a_{11}$,
$a_{21}$, and $a_{22}$, has already been scaled.  From the
annihilation condition $U^{\ast}AU=\Lambda$, as shown in
\cref{ss:SM1.1} similarly to~\cite{Hari-BegovicKovac-21}, but using
the fused multiply-add operation with a single rounding of its
result~\cite{IEEE-754-2019} (i.e.,
$\mathop{\mathrm{fma}}(a,b,c)=\mathop{\circ}(a\cdot b+c)$, where
$\circ$ denotes the chosen rounding method), and assuming $\arg{0}=0$
for determinacy, it follows
\begin{equation}
  \begin{aligned}
    \lambda_1/\cos^2\varphi&=(a_{22}\cdot\tan\varphi+2|a_{21}|)\cdot\tan\varphi+a_{11},\\
    \lambda_2/\cos^2\varphi&=(a_{11}\cdot\tan\varphi-2|a_{21}|)\cdot\tan\varphi+a_{22},
  \end{aligned}
  \label{e:jaclam}
\end{equation}
and for $\alpha$ from~\cref{e:1}, in the complex case computed as
in~\cref{e:zpolar},
\begin{equation}
  \alpha=\arg{a_{21}},\quad
  \mathrm{e}^{\mathrm{i}\alpha}=a_{21}/|a_{21}|.
  \label{e:alpha}
\end{equation}
In the real case~\cref{e:alpha} leads to
$\mathrm{e}^{\mathrm{i}\alpha}=\mathop{\mathrm{sign}}{a_{21}}=\pm 1$
(note, $\mathop{\mathrm{sign}}{\pm 0}=\pm 1$).  Also,
\begin{equation}
  \tan\varphi=\frac{\tan(2\varphi)}{1+\sqrt{\tan(2\varphi)\cdot\tan(2\varphi)+1}},\qquad
  \cos\varphi=\frac{1}{\sqrt{\tan\varphi\cdot\tan\varphi+1}}.
  \label{e:jactan}
\end{equation}
With the methods from~\cite{Novakovic-20}, and denoting by
$\mathop{\mathrm{fl}}(x)=\mathop{\circ}(x)$ the rounded,
floating-point representation of the value of the expression $x$,
$\tan(2\varphi)$ is computed branch-free, as
\begin{equation}
  \tan(2\varphi)=\min\left(\mathop{\mathrm{fmax}}\left(\frac{2|a_{21}|}{|a|},0\right),\mathop{\mathrm{fl}}(\sqrt{\nu})\right)\cdot\mathop{\mathrm{sign}}{a},\qquad
  a=a_{11}-a_{22}.
  \label{e:tan2}
\end{equation}
The $\mathrm{fmax}$ call returns its second argument if the first one
is $\mathtt{NaN}$ (when $|a_{21}|=|a|=0$).

The upper bound of $\mathop{\mathrm{fl}}(\sqrt{\nu})$ on
$|\tan(2\varphi)|$ enables simplifying~\cref{e:jactan}, since, when
rounding to nearest, tie to even,
$\mathop{\mathrm{fma}}(\mathop{\mathrm{fl}}(\sqrt{\nu}),\mathop{\mathrm{fl}}(\sqrt{\nu}),1)$
cannot overflow in single or double precision, so its square root
in~\cref{e:jactan} can be computed faster than resorting to the
$\mathop{\mathrm{hypot}}(x,y)=\sqrt{x^2+y^2}$ function.  It can be
verified (by a C program, e.g.) that the same bound ensures
that~\cref{e:jactan} gives the correct answer for $\tan\varphi$ (i.e.,
$\pm 1$) instead of $\mathtt{NaN}$ when the unbounded $\tan(2\varphi)$
would be $\pm\infty$ due to $|a|=0$ and $|a_{21}|>0$.

\looseness=-1
By factoring out $\cos\varphi$ from $U$ in~\cref{e:1},
$\sin\varphi=\tan\varphi\cdot\cos\varphi$ does not have to be
computed, unless the eigenvectors are not required to be explicitly
represented, since they can postmultiply a pivot column pair in the
Jacobi SVD algorithm in this factored form, with $\cos\varphi$ and
$\tan\varphi$, as in~\cref{e:transf}.  In the Jacobi SVD algorithm,
the ordering of the (always positive) eigenvalues of a pivot matrix is
important for ordering the transformed pivot columns of the iteration
matrix, i.e., for swapping them if $\lambda_1<\lambda_2$, while the
actual eigenvalues have to be computed for the standalone, generic EVD
only.
\subsubsection{On $\mathtt{invsqrt}$}\label{sss:2.1.1}
The
$\mathop{\mathtt{invsqrt}}(\mathsf{x})=\mathop{\mathrm{fl}}(\mathsf{1}/\sqrt{\mathsf{x}})$
intrinsic function (a vectorized, presently not correctly rounded
implementation of the standard's recommended $\mathrm{rSqrt}$
function~\cite{IEEE-754-2019}) is appropriate for computation of the
cosines directly from the squares of the secants ($\tan^2\varphi+1$)
in~\cref{e:jactan}, but it may not be available on another platform.
A slower but unconditionally reproducible
alternative, as explained in~\cite{Novakovic-Singer-22}, is to
compute
\begin{equation}
  \begin{gathered}
    \mathop{\mathrm{fl}}(\sec^2\varphi)=\mathop{\mathrm{fma}}(\mathop{\mathrm{fl}}(\tan\varphi),\mathop{\mathrm{fl}}(\tan\varphi),1),\quad
    \mathop{\mathrm{fl}}(\sec\varphi)=\mathop{\mathrm{fl}}(\sqrt{\mathop{\mathrm{fl}}(\sec^2\varphi)}),\\
    \mathop{\mathrm{fl}}(\cos\varphi)=\mathop{\mathrm{fl}}(1/\mathop{\mathrm{fl}}(\sec\varphi)),\quad
    \mathop{\mathrm{fl}}(\sin\varphi)=\mathop{\mathrm{fl}}(\mathop{\mathrm{fl}}(\tan\varphi)/\mathop{\mathrm{fl}}(\sec\varphi)),\quad
  \end{gathered}
  \label{e:jacsec}
\end{equation}
and optionally replace~\cref{e:jaclam} by the potentially more
accurately evaluated expressions
\begin{equation}
  \begin{aligned}
    \lambda_1&=((a_{22}\cdot\tan\varphi+2|a_{21}|)\cdot\tan\varphi+a_{11})/\sec^2\varphi,\\
    \lambda_2&=((a_{11}\cdot\tan\varphi-2|a_{21}|)\cdot\tan\varphi+a_{22})/\sec^2\varphi.
  \end{aligned}
  \label{e:lamsec}
\end{equation}

\looseness=-1
Let $\delta_1$ be the relative rounding error accumulated in the
process of computing $\mathop{\mathrm{fl}}(\cos\varphi)$ from
$x=\mathop{\mathrm{fl}}(\sec^2\varphi)$.  From~\cref{e:jacsec},
$\mathop{\mathrm{fl}}(\sec\varphi)=\sqrt{x}(1+\varepsilon_0)=y$, and
\begin{equation}
  \mathop{\mathrm{fl}}(\cos\varphi)=\frac{1+\varepsilon_1}{y}=\frac{1+\varepsilon_1}{1+\varepsilon_0}\cdot\frac{1}{\sqrt{x}}=\delta_1\cdot\frac{1}{\sqrt{x}},\quad
  0\le\max\{|\varepsilon_0|,|\varepsilon_1|\}\le\varepsilon,
  \label{e:dcos}
\end{equation}
where $\varepsilon=2^{-(p+1)}=0.5\mathop{\mathrm{ulp}}(1)$ is the
machine precision when rounding to nearest, $p$ is the number of bits
of significand ($23$, $52$, or $112$ for single, double, or quadruple
precision), and
$\mathop{\mathrm{ulp}}(x)=2^{\lfloor\lg|x|\rfloor-p}$.  By minimizing
and maximizing $\delta_1$ in~\cref{e:dcos} it follows
\begin{equation}
  \delta_1^-=\frac{1-\varepsilon}{1+\varepsilon}\le\delta_1^{}\le\frac{1+\varepsilon}{1-\varepsilon}=\delta_1^+.
  \label{e:d1}
\end{equation}
\subsection{Implementation of complex arithmetic}\label{ss:2.2}
\looseness=-1
Let $a$, $b$, $c$, and $d$ be complex values such that $d=a\cdot b+c$.
By analogy with the real fused multiply-add operation, let such a
ternary complex function and its value in floating-point be denoted as
$\tilde{d}=\mathop{\mathrm{fl}}(d)=\mathop{\mathrm{fma}}(a,b,c)$.
Then, $\tilde{d}$ can be computed with only two roundings for
each of its components (as suggested in the \texttt{cuComplex.h}
header in the CUDA toolkit\footnote{The CUDA toolkit is available at
\url{https://developer.nvidia.com/cuda-toolkit}.}), as
\begin{equation}
  \Re\tilde{d}=\mathop{\mathrm{fma}}(\Re{a},\Re{b},\mathop{\mathrm{fma}}(-\Im{a},\Im{b},\Re{c})),\quad
  \Im\tilde{d}=\mathop{\mathrm{fma}}(\Re{a},\Im{b},\mathop{\mathrm{fma}}(\Im{a},\Re{b},\Im{c})).
  \label{e:zfma}
\end{equation}

A branch-free way of computing the polar form of a complex value
$z$ as $|z|\mathrm{e}^{\mathrm{i}\beta}$, with $\beta=\arg{z}$ and
$\mathop{\mathrm{fl}}(|z|)=\mathop{\mathrm{hypot}}(\Re{z},\Im{z})\le\nu$,
was given in~\cite[Eq.~(1)]{Novakovic-20} as
\begin{equation}
  \mathop{\mathrm{fl}}(\cos\beta)=\mathop{\mathrm{fmin}}(|\Re{z}|/\mathop{\mathrm{fl}}(|z|),1)\cdot\mathop{\mathrm{sign}}{\Re{z}},\quad
  \mathop{\mathrm{fl}}(\sin\beta)=\Im{z}/\max\{\mathop{\mathrm{fl}}(|z|),\check{\mu}\},
  \label{e:zpolar}
\end{equation}
where $\check{\mu}$ is the smallest subnormal positive non-zero real
value.  Floating-point operations must not trap on exceptions and
subnormal numbers have to be supported as both inputs and outputs.
The $\mathop{\mathrm{fmin}}$ and $\mathop{\mathrm{fmax}}$ functions
have to return their second argument if the first one is
$\mathtt{NaN}$, but the full C language~\cite{C-18} semantics is not
required.
\subsubsection{On $\mathtt{hypot}$}\label{sss:2.2.1}
\looseness=-1
The $\mathtt{hypot}$ vector intrinsic function may not be available in
another environment.  It can be substituted~\cite{Novakovic-20} by a
na\"{\i}ve yet fully vectorized and unconditionally reproducible
\cref{a:hypot} (using the notation from \cref{ss:2.4}), based on the
well-known relations $\mathop{\mathrm{hypot}}(0,0)=0$ and (abusing the
symbols $m$ and $M$)
\begin{equation}
  \begin{gathered}
    \sqrt{x^2+y^2}=M\sqrt{\left(\frac{x}{M}\right)^2+\left(\frac{y}{M}\right)^2}=M\sqrt{\left(\frac{m}{M}\right)^2+1}=M\sqrt{q^2+1},\\
    q=m/M,\quad m=\min\{|x|,|y|\},\quad\nu\ge M=\max\{|x|,|y|\}>0.
  \end{gathered}
  \label{e:hypot}
\end{equation}

\begin{algorithm}[hbtp]
  \caption{A na\"{\i}ve vectorized $\mathtt{hypot}$.}
  \label{a:hypot}
  \begin{algorithmic}[1]
    \REQUIRE{vectors $\mathsf{x}$ and $\mathsf{y}$ with finite values in each lane}
    \ENSURE{$\mskip-5mu\approx\mskip-4mu\sqrt{\mathsf{x}^{\mathsf{2}}+\mathsf{y}^{\mathsf{2}}}$, not correctly rounded but without undue overflow in each lane}
    \STATE{$\mathsf{0}=\mathop{\mathtt{setzero}}();\quad\mathsf{1}=\mathop{\mathtt{set1}}(1.0);$}
    \COMMENT{all vector lanes set to a constant}
    \STATE{$-\mathsf{0}=\mathop{\mathtt{set1}}(-0.0);\quad\mathsf{x}=\mathop{\mathtt{andnot}}(-\mathsf{0},\mathsf{x});\quad\mathsf{y}=\mathop{\mathtt{andnot}}(-\mathsf{0},\mathsf{y});$}
    \COMMENT{$\mathsf{x}=|\mathsf{x}|$, $\mathsf{y}=|\mathsf{y}|$}
    \STATE{$\mathsf{m}=\mathop{\mathtt{min}}(\mathsf{x},\mathsf{y});\quad\mathsf{M}=\mathop{\mathtt{max}}(\mathsf{x},\mathsf{y});$}
    \STATE{$\mathsf{q}=\mathop{\mathtt{max}}(\mathop{\mathtt{div}}(\mathsf{m},\mathsf{M}),\mathsf{0});$}
    \COMMENT{$\mathtt{max}$ replaces $0/0\to\mathtt{NaN}$ with $0$}
    \RETURN{$\mathop{\mathtt{mul}}(\mathop{\mathtt{sqrt}}(\mathop{\mathtt{fmadd}}(\mathsf{q},\mathsf{q},\mathsf{1})),\mathsf{M});$}
  \end{algorithmic}
\end{algorithm}

\Cref{l:hypot} gives the relative error bounds for
$\mathop{\mathrm{hypot}}(x,y)$ if computed in the
standard~\cite{IEEE-754-2019} floating-point arithmetic as in
\cref{a:hypot}, with (exactly representable) scalars $x$ and $y$
instead of vectors.  In~\cref{e:hypot} and in the scalar
\cref{a:hypot}, $M>0$ implies
$\mathop{\mathrm{fl}}(\mathop{\mathrm{hypot}}(x,y))>0$, while $M=0$
implies $q=0=\mathop{\mathrm{fl}}(\mathop{\mathrm{hypot}}(x,y))$ in
the latter.

\begin{lemma}\label{l:hypot}
  Let $\mathop{\mathrm{hypot}}(x,y)$ be defined by~\cref{e:hypot}, and
  assume that neither overflow nor underflow occurs in the
  final multiplication of $M$ by
  $\mathop{\mathrm{fma}}(\mathop{\mathrm{fl}}(q),\mathop{\mathrm{fl}}(q),1)$.
  Then,
  \begin{equation}
    \begin{gathered}
      \mathop{\mathrm{fl}}(\mathop{\mathrm{hypot}}(x,y))=\delta_2\mathop{\mathrm{hypot}}(x,y),\\
      \delta_2^-=(1-\varepsilon)^{5/2}\sqrt{1-\varepsilon(2-\varepsilon)/2}<\delta_2^{}<(1+\varepsilon)^{5/2}\sqrt{1+\varepsilon(2+\varepsilon)/2}=\delta_2^+.
    \end{gathered}
    \label{e:rehyp}
  \end{equation}
\end{lemma}
\begin{proof}
  See \cref{sss:SM1.1.1}.
\end{proof}

\begin{remark}\label{r:ufl}\looseness=-1
  For all complex, representable values $z\ne 0$,
  $|\mathop{\mathrm{fl}}(z/\mathop{\mathrm{fl}}(|z|))|\le\sqrt{2}$ in
  the standard rounding modes, since, from~\cref{e:hypot},
  $\mathop{\mathrm{fl}}(|z|)\ge\max\{|\Re{z}|,|\Im{z}|\}$, what,
  together with $z=|z|\mathrm{e}^{\mathrm{i}\beta}$
  and~\cref{e:zpolar}, gives
  $1\ge|\mathop{\mathrm{fl}}(\cos\beta)|\ge|\mathop{\mathrm{fl}}(\sin\beta)|$
  if $|\Re{z}|\ge|\Im{z}|$, or
  $1\ge|\mathop{\mathrm{fl}}(\sin\beta)|\ge|\mathop{\mathrm{fl}}(\cos\beta)|$
  otherwise.  In certain cases that do not satisfy the assumption of
  \cref{l:hypot},
  $\mathop{\mathrm{fl}}(\mathrm{e}^{\mathrm{i}\alpha})$
  from~\cref{e:alpha}, computed by~\cref{e:hypot} or by another
  $\mathrm{hypot}$, might be quite inaccurate, with
  $|\mathop{\mathrm{fl}}(\mathrm{e}^{\mathrm{i}\alpha})|\le\sqrt{2}$.
  \Cref{r:hypot} shows that this bound is sharp.
\end{remark}

\begin{remark}\label{r:hypot}
  For all representable $x=\mathop{\mathrm{fl}}(x)$,
  $\mathop{\mathrm{fl}}(\mathop{\mathrm{hypot}}(x,x))=\mathop{\mathrm{fl}}(\mathop{\mathrm{fl}}(\sqrt{2})|x|)$
  if computed as in~\cref{e:hypot}.  Let $\Psi_{\mathtt{T}}^{\circ}$
  be a set of subnormal non-zero $x$, that depends on the
  implementation of $\mathrm{hypot}$, the datatype $\mathtt{T}$, and
  the rounding mode $\circ$ in effect, on which
  $\mathop{\mathrm{fl}}(\mathop{\mathrm{hypot}}(x,x))=|x|$.  This set
  is non-empty with the default rounding (e.g., it contains
  $x=\check{\mu}$ for \cref{a:hypot} and for the tested math library's
  $\mathrm{hypot}$), but it should be empty when rounding to
  $+\infty$.  For $y=x\in\Psi_{\mathtt{T}}^{\circ}$,
  $\delta_2^{}=1/\sqrt{2}$ in~\cref{e:rehyp}.  If for a complex $z$
  holds $|\Re{z}|=|\Im{z}|=|x|\in\Psi_{\mathtt{T}}^{\circ}$, then
  $\mathop{\mathrm{fl}}(|z|)=|x|$, and so
  $|\mathop{\mathrm{fl}}(\cos\beta)|=|\mathop{\mathrm{fl}}(\sin\beta)|=1$
  in~\cref{e:zpolar}.  This has serious consequences for accuracy of
  the eigenvectors computed by the complex LAPACK
  routines\footnote{\url{https://github.com/Reference-LAPACK/lapack/blob/master/INSTALL/test_zcomplexabs.f}
  and its history contain further comments on accuracy of the absolute
  value of a complex number.}, as shown in \cref{ss:5.2}.  If
  $\mathop{\mathrm{fl}}(U)$ comes from \cref{a:z8jac2} or its single
  precision version, an inaccurate
  $\mathop{\mathrm{fl}}(\mathrm{e}^{\mathrm{i}\alpha})$
  from~\cref{e:alpha}, for which
  $|\mathop{\mathrm{fl}}(\mathrm{e}^{\mathrm{i}\alpha})|=\sqrt{2}$, is
  avoided in many cases if~\cref{e:z} implies a large enough upscaling.
\end{remark}
\subsection{Almost exact scaling of a Hermitian matrix of order two}\label{ss:2.3}
A scaling of $A$ sufficient for finiteness of $\Lambda$ in
floating-point is given in \cref{p:s2}.

\begin{proposition}\label{p:s2}
  If for a Hermitian matrix $A$ of order two holds
  \begin{equation}
    \hat{a}=\max_{1\le j\le i\le 2}|a_{ij}|\le\nu/(4\sqrt{2})=\tilde{\nu},
    \label{e:scla}
  \end{equation}
  then no output from any computation
  in~\cref{e:alpha,e:jactan,e:tan2,e:jacsec,e:lamsec}, including the
  resulting eigenvalues, can overflow, assuming the
  standard~\cite{IEEE-754-2019} \emph{non-stop} floating-point
  arithmetic in at least single precision ($p\ge 23$) and rounding to
  the nearest, tie to even.

  Moreover, \emph{barring any underflow} of the results of those
  computations, the following relative error bounds hold for the
  quantities computed as in~\cref{e:alpha,e:jactan,e:tan2,e:jacsec}:
  \begin{equation}
    \begin{gathered}
      \begin{aligned}
        \mathop{\mathrm{fl}}(\cos\alpha)&=\delta_{\alpha}'\cos\alpha\\
        \mathop{\mathrm{fl}}(\sin\alpha)&=\delta_{\alpha}''\sin\alpha
      \end{aligned}\,,\quad
      1-4.000000\,\varepsilon<\left\{\begin{gathered}
      \delta_{\alpha}',\delta_{\alpha}''\\
      |\mathop{\mathrm{fl}}(\mathrm{e}^{\pm\mathrm{i}\alpha})|\end{gathered}\right\}
      <1+4.000001\,\varepsilon,\\
      \mathop{\mathrm{fl}}(\tan\varphi)=\delta_{\varphi}^{\mathbb{F}}\tan\varphi,\quad
      \begin{rcases}
        1-\hphantom{0}5.500000\,\varepsilon\\
        1-11.500000\,\varepsilon
      \end{rcases}<\delta_{\varphi}^{\mathbb{F}}<
      \begin{cases}
        1+\hphantom{0}5.500001\,\varepsilon,&\mathbb{F}=\mathbb{R},\\
        1+11.500004\,\varepsilon,&\mathbb{F}=\mathbb{C},
      \end{cases}\\
      \mathop{\mathrm{fl}}(\cos\varphi)=\delta_c^{\mathbb{F}}\cos\varphi,\quad
      \begin{rcases}
        1-\hphantom{0}8.000000\,\varepsilon\\
        1-14.000000\,\varepsilon
      \end{rcases}<\delta_c^{\mathbb{F}}<
      \begin{cases}
        1+\hphantom{0}8.000002\,\varepsilon,&\mathbb{F}=\mathbb{R},\\
        1+14.000006\,\varepsilon,&\mathbb{F}=\mathbb{C},
      \end{cases}
    \end{gathered}
    \label{e:prop}
  \end{equation}
  where $\mathbb{F}$ indicates whether $A$ is real or complex.  For
  different ranges of $\delta_1$ or $\delta_2$ the bounds can be
  recalculated as described in \cref{ss:SM1.3}, as well as for a
  different rounding mode (while lowering the upper bound of
  $\mathop{\mathrm{fl}}(\sqrt{\nu})$ for $\mathop{\mathrm{fl}}(\tan 2\varphi)$
  if required).
\end{proposition}
\begin{proof}
  The proof is presented in \cref{ss:SM1.2}.
\end{proof}

\Cref{c:s2} gives a practical scaling bound in the terms of the
magnitudes of the components of the elements of $A$ and an exactly
representable value based on $\nu$.

\begin{corollary}\label{c:s2}
  If a Hermitian matrix $A$ of order two is scaled such that
  \begin{equation}
    \max_{1\le j\le i\le 2}\max\{|\Re{a_{ij}}|,|\Im{a_{ij}}|\}\le\tilde{\nu}/\sqrt{2}=\nu/8,
    \label{e:sclac}
  \end{equation}
  then the assumption~\cref{e:scla} of \cref{p:s2} holds for $A$.
\end{corollary}
\begin{proof}
  From $\max\{|\Re{a_{ij}}|,|\Im{a_{ij}}|\}\le\nu/8$ it follows
  $|a_{ij}|\le\nu\sqrt{2}/8=\nu/(4\sqrt{2})=\tilde{\nu}$.
\end{proof}

With $\eta=\left\lfloor\lg(\nu/8)\right\rfloor$,~\cref{e:sclac}
implies the scaling factor $2^{\zeta}$ of a $A$ ($A'=2^{\zeta}A$),
where
\begin{equation}
  \begin{gathered}
    \zeta=\min\{\nu,\zeta_{11}^{},\zeta_{22}^{},\zeta_{21}^{\Re},\zeta_{21}^{\Im}\}\in\mathbb{Z},\\
    \zeta_{ii}^{}=\eta-\left\lfloor\lg|a_{ii}^{}|\right\rfloor,\quad
    \zeta_{21}^{\Re}=\eta-\left\lfloor\lg|\Re{a_{21}^{}}|\right\rfloor,\quad
    \zeta_{21}^{\Im}=\eta-\left\lfloor\lg|\Im{a_{21}^{}}|\right\rfloor,
  \end{gathered}
  \label{e:z}
\end{equation}
while $\lg 0=-\infty$ and $\zeta$ is an integer exactly representable
as a floating-point value\footnote{For all standard floating-point
datatypes, the value represented by $\nu$ is an integer.}.

The eigenvalues of $A'$ in $\Lambda'$ cannot overflow, but in
$\Lambda\approx 2^{-\zeta}\Lambda'$ might, where the approximate sign
warns of a possibility that the values in $A$ small enough in
magnitude become subnormal and lose their least significant bits when
downscaling it with $\zeta<0$.  Otherwise, the scaling of $A$ is
exact, and some subnormal values might be raised into the normal range
when $\zeta>0$.  The scaled eigenvalues in $\Lambda'$ could be kept
alongside $\zeta$ in the cases where $\Lambda$ is expected to overflow
or underflow, and compatibility with the LAPACK's \texttt{xLAEV2}
routines is not required.  Else, $\Lambda$ is returned by backscaling.
\subsubsection{A serial eigendecomposition algorithm for Hermitian
  matrices of order two}\label{sss:2.3.1}
Listing~\ref{l:1} shows a serial C implementation of the described
eigendecomposition algorithm for a Hermitian matrix of order two.  All
untyped variables are in double precision.  The computed values are
equivalent to those from the vectorized \cref{a:z8jac2}.  The three
branches in lines~\ref{ll:9} and \ref{ll:10} are avoided in the
vectorized code.  Since the $\mathtt{frexp}$ function returns $0$ for
the exponent of zero instead of a huge negative value, the
lines~\ref{ll:5} to~\ref{ll:8} take care of this exception as
in~\cite[subsection~2.1.1]{Novakovic-Singer-22}.

\begin{lstlisting}[float=hbtp,caption=\texttt{zsjac2}: a serial eigendecomposition of a double precision Hermitian matrix $A$ of order two in C,label=l:1]
// input: $\!a_{11}$, $\!a_{22}$, $\!\Re{a_{21}}$, $\!\Im{a_{21}}$; output: $\!\cos\varphi$, $\!\cos\alpha\tan\varphi$, $\!\sin\alpha\tan\varphi$; $\!\lambda_1$, $\!\lambda_2$; $\!p$[,$\zeta'$]
int $\zeta_{11}^{}$, $\zeta_{22}^{}$, $\zeta_{21}^{\Re}$, $\zeta_{21}^{\Im}$, $\zeta_{\mathbb{R}}^{}$, $\zeta_{\mathbb{C}}^{}$, $\zeta$, $\zeta'$, $\eta'$ = DBL_MAX_EXP - 3; // $\eta'=\eta+1=1021$
double $\check{\mu}$ = DBL_TRUE_MIN, $\sqrt{\nu}$ = 1.34078079299425956E+154 /* sqrt(DBL_MAX) */;
// determine $\zeta$ assuming all inputs are finite; avoid taking the exponent of $0$
frexp(fmax(fabs($a_{11}$), $\check{\mu}$), &$\zeta_{11}$); $\zeta_{11}$ = $\eta'$ - $\zeta_{11}$;$\label{ll:5}$
frexp(fmax(fabs($a_{22}$), $\check{\mu}$), &$\zeta_{22}$); $\zeta_{22}$ = $\eta'$ - $\zeta_{22}$;$\label{ll:6}$
frexp(fmax(fabs($\Re{a_{21}^{}}$), $\check{\mu}$), &$\zeta_{21}^{\Re}$); $\zeta_{21}^{\Re}$ = $\eta'$ - $\zeta_{21}^{\Re}$;$\label{ll:7}$
frexp(fmax(fabs($\Im{a_{21}^{}}$), $\check{\mu}$), &$\zeta_{21}^{\Im}$); $\zeta_{21}^{\Im}$ = $\eta'$ - $\zeta_{21}^{\Im}$;$\label{ll:8}$
$\zeta_{\mathbb{R}}$ = (($\zeta_{11}$ <= $\zeta_{22}$) ? $\zeta_{11}$ : $\zeta_{22}$); $\zeta_{\mathbb{C}}^{}$ = (($\zeta_{21}^{\Re}$ <= $\zeta_{21}^{\Im}$) ? $\zeta_{21}^{\Re}$ : $\zeta_{21}^{\Im}$);$\label{ll:9}$
$\zeta$ = (($\zeta_{\mathbb{R}}$ <= $\zeta_{\mathbb{C}}$) ? $\zeta_{\mathbb{R}}$ : $\zeta_{\mathbb{C}}$); $\zeta'$ = -$\zeta$; // $\text{\cref{e:z}}\label{ll:10}$
// scale the input matrix $A$ by $2^{\zeta}$: $A'=2^{\zeta}A$
$\Re{a_{21}'}$ = scalbn($\Re{a_{21}^{}}$, $\zeta$); $\Im{a_{21}'}$ = scalbn($\Im{a_{21}^{}}$, $\zeta$);$\label{ll:12}$
$a_{11}'$ = scalbn($a_{11}^{}$, $\zeta$); $a_{22}'$ = scalbn($a_{22}^{}$, $\zeta$);$\label{ll:13}$
// find the polar form of $a_{21}'=2^{\zeta}a_{21}^{}$ as $|a_{21}'|\mathrm{e}^{\mathrm{i}\alpha}$ using $\text{\cref{e:zpolar,e:hypot,e:alpha}}$
$|\Re{a_{21}'}|$ = fabs($\Re{a_{21}'}$); $|\Im{a_{21}'}|$ = fabs($\Im{a_{21}'}$);$\label{ll:15}$
am = fmin($|\Re{a_{21}'}|$, $|\Im{a_{21}'}|$); aM = fmax($|\Re{a_{21}'}|$, $|\Im{a_{21}'}|$); mM = fmax(am / aM, 0.0);$\label{ll:16}$
$|a_{21}'|$ = sqrt(fma(mM, mM, 1.0)) * aM; // $\mathop{\mathrm{hypot}}(|\Re{a_{21}'}|,|\Im{a_{21}'}|)$ as in $\text{\cref{e:hypot}}\label{ll:17}$
$\cos\alpha$ = copysign(fmin($|\Re{a_{21}'}|$ / $|a_{21}'|$, 1.0), $\Re{a_{21}'}$);$\label{ll:18}$
$\sin\alpha$ = $\Im{a_{21}'}$ / fmax($|a_{21}'|$, $\check{\mu}$);$\label{ll:19}$
// compute the Jacobi rotation as in $\text{\cref{ss:2.1}}$
$o$ = $|a_{21}'|$ * 2.0; $a$ = $a_{11}'$ - $a_{22}'$;$\label{ll:21}$
$\tan{2\varphi}$ = copysign(fmin(fmax($o$ / fabs($a$), 0.0), $\sqrt{\nu}$), $a$); // $\text{\cref{e:tan2}}\label{ll:22}$
$\tan\varphi$ = $\tan{2\varphi}$ / (1.0 + sqrt(fma($\tan{2\varphi}$, $\tan{2\varphi}$, 1.0))); // $\text{\cref{e:jactan}}\label{ll:23}$
$\sec^2\varphi$ = fma($\tan\varphi$, $\tan\varphi$, 1.0); $\sec\varphi$ = sqrt($\sec^2\varphi$); $\cos\varphi$ = 1.0 / $\sec\varphi$; // $\text{\cref{e:jacsec}}\label{ll:24}$
$\cos\alpha\tan\varphi$ = $\cos\alpha$ * $\tan\varphi$; // optionally, $\cos\alpha\sin\varphi$ = $\cos\alpha\tan\varphi$ / $\sec\varphi$;$\label{ll:25}$
$\sin\alpha\tan\varphi$ = $\sin\alpha$ * $\tan\varphi$; // optionally, $\sin\alpha\sin\varphi$ = $\sin\alpha\tan\varphi$ / $\sec\varphi$;$\label{ll:26}$
// compute the (backscaled) eigenvalues as in $\text{\cref{e:lamsec}}$
$\lambda_1'$ = fma($\tan\varphi$, fma($a_{22}'$, $\tan\varphi$, $\hphantom{\text{-}}o$), $a_{11}'$) / $\sec^2\varphi$;$\label{ll:28}$
$\lambda_2'$ = fma($\tan\varphi$, fma($a_{11}'$, $\tan\varphi$, -$o$), $a_{22}'$) / $\sec^2\varphi$;$\label{ll:29}$
$\lambda_1^{}$ = scalbn($\lambda_1'$, $\zeta'$); $\lambda_2^{}$ = scalbn($\lambda_2'$, $\zeta'$); // optionally$\label{ll:30}$
return (($\zeta'$ << 1) | ($\lambda_1'$ < $\lambda_2'$)); // pack $\zeta'$ and the permutation bit $p\label{ll:31}$
\end{lstlisting}
\subsection{Intel AVX-512 vectorization}\label{ss:2.4}
In the following, a $\mathtt{function}$ in the teletype font is as
shorthand for the Intel AVX-512 C
intrinsic\footnote{\url{https://www.intel.com/content/www/us/en/docs/intrinsics-guide/index.html}}
$\mathtt{\_mm512\_function\_pd}$.  A $\mathsf{variable}$ in the
sans-serif font or a $\bm{bold}$ Greek letter without an explicitly
specified type is assumed to be a vector of the type
$\mathtt{\_\_m512d}$, with $\mathtt{s}=8$ double precision values
(lanes).  A variable named with a letter from the Fraktur font (e.g.,
$\mathfrak{m}$) represents an eight-bit lane mask of the type
$\mathtt{\_\_mmask8}$ (see~\cite{Intel-21} for the datatypes and the
vector instructions).  Most, but not all, intrinsics correspond to a
single instruction.

The chosen vectorization approach is Intel-specific for convenience,
but it is meant to determine the minimal set of vector operations
required (hardware-supported or emulated) on any platform that is
considered to be targeted, either by that platform's native
programming facilities, or via some multi-platform interface in the
spirit of the SLEEF~\cite{Shibata-Petrogalli-20} vectorized C math
library, e.g., or, if possible, by the compiler's auto-vectorizer, as
it has partially been done with the Intel Fortran compiler for the
batched generalized eigendecomposition of pairs of Hermitian matrices
of order two, one of them being positive definite, by the
Hari--Zimmermann
method~\cite[subsection~6.2]{Singer-DiNapoli-Novakovic-Caklovic-20}.
\subsubsection{Data storage and the split complex representation}\label{sss:2.4.1}
For simplicity and performance, a real vector $\mathbf{x}$ of length
$m$ is assumed to be contiguous in memory, aligned to a multiple of
the larger of the vector size and the cache-line size (on Intel CPUs
they are both $64\,\mathrm{B}$), and zero-padded at its end to the
optimal length $\tilde{m}$, where
\begin{equation}
  \tilde{m}=\begin{cases}
  m,&\text{if }m\bmod\mathtt{s}=0,\\
  m+(\mathtt{s}-(m\bmod\mathtt{s})),&\text{otherwise}.
  \end{cases}
  \label{e:pad0}
\end{equation}

If $\mathbf{z}$ is a complex array, it is kept as two real
non-overlapping ones, $\Re{\tilde{\mathbf{z}}}$ and
$\Im{\tilde{\mathbf{z}}}$ (for the real and the imaginary parts,
respectively), laid out in this split form as
\begin{equation}
  \Re{\tilde{\mathbf{z}}}=\begin{bmatrix}\Re{\mathbf{z}}\\\mathbf{0}_{\tilde{m}-m}\end{bmatrix},\qquad
  \Im{\tilde{\mathbf{z}}}=\begin{bmatrix}\Im{\mathbf{z}}\\\mathbf{0}_{\tilde{m}-m}\end{bmatrix},
  \label{e:layout1}
\end{equation}
where the blocks $\mathbf{0}_{\tilde{m}-m}$ of zeros are the minimal
padding that makes the length of the real as well as of the imaginary
block a multiple of the number of SIMD lanes ($\mathtt{s}$), with
$\tilde{m}$ from~\cref{e:pad0}.  A complex $m\times n$ matrix $G$ is
kept as two real $\tilde{m}\times n$ matrices,
\begin{displaymath}
  \addtolength{\arraycolsep}{-3pt}
  \Re{G}=\begin{bmatrix}
  \Re{g_1} & \Re{g_2} & \cdots & \Re{g_n}\\
  \mathbf{0}_{\tilde{m}-m} & \mathbf{0}_{\tilde{m}-m} & \cdots & \mathbf{0}_{\tilde{m}-m}
  \end{bmatrix},\quad
  \Im{G}=\begin{bmatrix}
  \Im{g_1} & \Im{g_2} & \cdots & \Im{g_n}\\
  \mathbf{0}_{\tilde{m}-m} & \mathbf{0}_{\tilde{m}-m} & \cdots & \mathbf{0}_{\tilde{m}-m}
  \end{bmatrix},
\end{displaymath}
i.e., each column $g_j$ of $G$, where $1\le j\le n$, is split into the
real ($\Re{g_j}$) and the imaginary ($\Im{g_j}$) part, as
in~\cref{e:layout1}.  The leading dimensions of $\Re{G}$ and $\Im{G}$
may differ, but each has to be a multiple of $\mathtt{s}$ to keep all
padded real columns properly aligned.

The split form has been used for efficient complex matrix
multiplication kernels~\cite{VanZee-Smith-17} and for the generalized
SVD computation by the implicit Hari--Zimmermann algorithm on
GPUs~\cite{Novakovic-Singer-21}.  Intel CPUs have no native
complex-specific arithmetic instructions operating on vectors of at
least single precision complex numbers in the customary, interleaved
representation, so a manual implementation of the vectorized complex
arithmetic is inevitable, for what the split representation is more
convenient.

A conversion of the customary representation of complex arrays to and
back from the split form, in both vectorized and parallel fashion, is
described in \cref{s:SM5}.

An input batch $(A^{(\ell)})_{\ell=1}^r$ of Hermitian matrices of
order two is kept as a collection of one-dimensional real arrays
$\tilde{\mathbf{a}}_{11}^{}$, $\tilde{\mathbf{a}}_{22}^{}$,
$\Re{\tilde{\mathbf{a}}_{21}^{}}$, $\Im{\tilde{\mathbf{a}}_{21}^{}}$,
of length $\tilde{r}$ and with layout~\cref{e:layout1}, where
$\tilde{r}$ is calculated from $r$ as in~\cref{e:pad0},
$(\tilde{\mathbf{a}}_{11}^{})_{\ell}^{}=a_{11}^{(\ell)}$,
$(\tilde{\mathbf{a}}_{22}^{})_{\ell}^{}=a_{22}^{(\ell)}$,
$(\Re{\tilde{\mathbf{a}}_{21}^{}})_{\ell}^{}=\Re{a_{21}^{(\ell)}}$,
$(\Im{\tilde{\mathbf{a}}_{21}^{}})_{\ell}^{}=\Im{a_{21}^{(\ell)}}$.
The output unpermuted eigenvalue matrices
$(\Lambda^{(\ell)})_{\ell=1}^r$ are stored as the arrays
$\tilde{\bm{\lambda}}_1^{}=(\lambda_1^{(\ell)})_{\ell=1}^{\tilde{r}}$
and
$\tilde{\bm{\lambda}}_2^{}=(\lambda_2^{(\ell)})_{\ell=1}^{\tilde{r}}$.
The corresponding Jacobi rotations' parameters are kept in the arrays
$\cos\tilde{\bm{\varphi}}$,
$\cos\tilde{\bm{\alpha}}\tan\tilde{\bm{\varphi}}$, and
$\sin\tilde{\bm{\alpha}}\tan\tilde{\bm{\varphi}}$.  If
$\lambda_1^{(\ell)}<\lambda_2^{(\ell)}$ then the permutation
$P^{(\ell)}=\left[\begin{smallmatrix}0&1\\1&0\end{smallmatrix}\right]$,
else
$P^{(\ell)}=\left[\begin{smallmatrix}1&0\\0&1\end{smallmatrix}\right]$,
what is compactly encoded by setting the
$((\ell-1)\bmod\mathtt{s})$-th bit of a bitmask $\mathfrak{p}$ to $1$
or $0$, respectively.
\subsubsection{Vectorized eigendecomposition of a batch of Hermitian matrices of order two}\label{sss:2.4.2}
\looseness=-1
Based on \cref{ss:2.1,ss:2.2,ss:2.3}, \cref{a:z8jac2} implements a
vectorized, branch-free, unconditionally reproducible
eigendecomposition method for at most $\mathtt{s}$ Hermitian matrices
of order two in double precision.  Although presented with the
AVX512DQ instruction subset, only the basic, AVX512F subset is
required\footnote{With the AVX512F instruction subset only, the
bitwise operations require reinterpreting casts to and from 64-bit
integer vectors (no values are changed, converted, or otherwise acted
upon); e.g.,\\
$\mathop{\mathtt{and}}(\mathsf{x},\mathsf{y})\mskip-2mu=\mskip-2mu\mathop{\mathtt{castsi512}}(\mathop{\mathtt{\_mm512\_and\_epi64}}(\mathop{\mathtt{\_mm512\_castpd\_si512}}(\mathsf{x}),\mathop{\mathtt{\_mm512\_castpd\_si512}}(\mathsf{y})))$.}.
The order of the statements slightly differs from the one in the
actual code, for legibility.  For easier understanding of the
vectorization, the lines of \cref{a:z8jac2} are suffixed by the
corresponding line numbers of the serial algorithm from
Listing~\ref{l:1} in brackets.

\Cref{a:zbjac2} builds on \cref{a:z8jac2} and computes the
eigendecomposition of a batch of Hermitian matrices of order two in
parallel, where each OpenMP thread executes \cref{a:z8jac2} in
sequence over a (not necessarily contiguous) subset of the input's
vectors.  No dependencies exist between the iterations of the
parallel-for loop of \cref{a:zbjac2}, so it can be generalized by
distributing its input over several machines.

\Cref{a:d8jac2,a:dbjac2}, detailed in \cref{s:SM2}, are the
specializations of \cref{a:z8jac2,a:zbjac2}, respectively, for real
symmetric matrices.  There,
$\mathrm{e}^{\mathrm{i}\tilde{\bm{\alpha}}}\tan\tilde{\bm{\varphi}}=\cos\tilde{\bm{\alpha}}\tan\tilde{\bm{\varphi}}=\mathop{\mathrm{sign}}\tilde{\mathbf{a}}_{21}\cdot\tan\tilde{\bm{\varphi}}$
elementwise.

Converting \cref{a:z8jac2,a:zbjac2} and \cref{a:d8jac2,a:dbjac2} to
single precision, with $\mathtt{s}=16$, requires redefining
$\nu=\mathtt{FLT\_MAX}$,
$\mathop{\mathrm{fl}}(\sqrt{\nu})=\text{\texttt{1.844674297E+19}}$,
$\check{\mu}=\mathtt{FLT\_TRUE\_MIN}$,
$\eta=\mathtt{FLT\_MAX\_EXP}-4$ and switching to
$\mathtt{\_mm512\_\cdots\_ps}$ intrinsics, $\mathtt{\_\_m512}$
vectors, and $\mathtt{\_\_mmask16}$ bitmasks, while the relative error
bounds in \cref{p:s2} still hold.  The serial code from
Listing~\ref{l:1} can be similarly converted.

\looseness=-1
The algorithms can be implemented in a scalar fashion
($\mathtt{s}=1$), as in Listing~\ref{l:1}, thus enabling access to the
higher-precision, scalar-only datatypes, such as extended precision,
or in a pseudo-scalar way of, e.g., GPU programming models, where each
thread executes the scalar code over a different data in the same
layout as proposed here.  An implementation in CUDA is discussed in
\cref{s:SM10}.

\begin{algorithm}[hbtp]
  \caption{$\mathtt{z8jac2}$: a vectorized eigendecomposition of at most $\mathtt{s}$ double precision Hermitian matrices of order two with the Intel's AVX-512 intrinsics.}
  \label{a:z8jac2}
  \begin{algorithmic}[1]
    \REQUIRE{$\mathtt{i}$; addresses of $\tilde{\mathbf{a}}_{11},\tilde{\mathbf{a}}_{22},\Re{\tilde{\mathbf{a}}_{21}},\Im{\tilde{\mathbf{a}}_{21}},\cos\tilde{\bm{\varphi}},\cos\tilde{\bm{\alpha}}\tan\tilde{\bm{\varphi}},\sin\tilde{\bm{\alpha}}\tan\tilde{\bm{\varphi}},\tilde{\bm{\lambda}}_1,\tilde{\bm{\lambda}}_2,\mathtt{p}$}
    \ENSURE{$\mathop{\mathsf{cos}}\varphi,\mathop{\mathsf{cos}}\alpha\mathop{\mathsf{tan}}\varphi,\mathop{\mathsf{sin}}\alpha\mathop{\mathsf{tan}}\varphi;\bm{\lambda}_1,\bm{\lambda}_2;\mathfrak{p}$}
    \COMMENT{a permutation-indicating bitmask\\[3pt]}
    \COMMENT{vectors with all lanes set to a compile-time constant\hfill}
    \STATE{$-\mathsf{0}=\mathop{\mathtt{set1}}(-0.0);\quad\mathsf{0}=\mathop{\mathtt{setzero}}();\quad\mathsf{1}=\mathop{\mathtt{set1}}(1.0);\quad\bm{\nu}=\mathop{\mathtt{set1}}(\mathtt{DBL\_MAX});$}
    \STATE{$\sqrt{\bm{\nu}}=\mathop{\mathtt{set1}}(\text{\texttt{1.34078079299425956E+154}});$}
    \COMMENT{$\mathop{\mathrm{fl}}(\sqrt{\nu})$ from~\cref{e:tan2}}
    \STATE{$\check{\bm{\mu}}=\mathop{\mathtt{set1}}(\mathtt{DBL\_TRUE\_MIN});\quad\bm{\eta}=\mathop{\mathtt{set1}}(1020.0);$}
    \COMMENT{$\eta=(\mathtt{DBL\_MAX\_EXP}-1)-3$\\[3pt]}
    \COMMENT{aligned loads of the $\mathtt{i}$th input vectors, e.g., $\mathsf{a}_{11}=\mathop{\mathtt{load}}(\tilde{\mathbf{a}}_{11}+\mathtt{i})$, happen here\hfill\null\\[3pt]}
    \COMMENT{determine the scaling exponents $\bm{\zeta}$ from~\cref{e:z} and scale $\mathsf{A}\to\mathsf{2}^{\bm{\zeta}}\mathsf{A}$\hfill}
    \STATE{$\bm{\zeta}_{11}=\mathop{\mathtt{sub}}(\bm{\eta},\mathop{\mathtt{getexp}}(\mathsf{a}_{11}));\quad\bm{\zeta}_{22}=\mathop{\mathtt{sub}}(\bm{\eta},\mathop{\mathtt{getexp}}(\mathsf{a}_{22}));$}
    \COMMENT{$[\ref{ll:5},\ref{ll:6}]$}
    \STATE{$\bm{\zeta}_{21}^{\Re}=\mathop{\mathtt{sub}}(\bm{\eta},\mathop{\mathtt{getexp}}(\Re{\mathsf{a}_{21}^{}}));\quad\bm{\zeta}_{21}^{\Im}=\mathop{\mathtt{sub}}(\bm{\eta},\mathop{\mathtt{getexp}}(\Im{\mathsf{a}_{21}^{}}));$}
    \COMMENT{$[\ref{ll:7},\ref{ll:8}]$}
    \STATE{$\bm{\zeta}=\mathop{\mathtt{min}}(\bm{\nu},\mathop{\mathtt{min}}(\mathop{\mathtt{min}}(\bm{\zeta}_{11}^{},\bm{\zeta}_{22}^{}),\mathop{\mathtt{min}}(\bm{\zeta}_{21}^{\Re},\bm{\zeta}_{21}^{\Im})));$}
    \COMMENT{$[\ref{ll:9},\ref{ll:10}]$}
    \STATE{$-\bm{\zeta}=\mathop{\mathtt{xor}}(\bm{\zeta},-\mathsf{0});$}
    \COMMENT{$\mathop{\mathtt{xor}}(\mathsf{x},-\mathsf{0})$ flips the sign bits in $\mathsf{x}$; optionally, store $-\bm{\zeta}\quad[\ref{ll:10}]$}
    \STATE{$\Re{\mathsf{a}_{21}}=\mathop{\mathtt{scalef}}(\Re{\mathsf{a}_{21}},\bm{\zeta});\quad\Im{\mathsf{a}_{21}}=\mathop{\mathtt{scalef}}(\Im{\mathsf{a}_{21}},\bm{\zeta});$}
    \COMMENT{$\mathsf{a}_{21}=\mathsf{2}^{\bm{\zeta}}\mathsf{a}_{21}\quad[\ref{ll:12}]$\label{al:8}}
    \STATE{$\mathsf{a}_{11}=\mathop{\mathtt{scalef}}(\mathsf{a}_{11},\bm{\zeta});\quad\mathsf{a}_{22}=\mathop{\mathtt{scalef}}(\mathsf{a}_{22},\bm{\zeta});$}
    \COMMENT{$\mathsf{a}_{ii}=\mathsf{2}^{\bm{\zeta}}\mathsf{a}_{ii}\quad[\ref{ll:13}]$\\[3pt]}
    \COMMENT{find the polar form of $\mathsf{a}_{21}$ using~\cref{e:alpha,e:zpolar}\hfill}
    \STATE{$|\Re{\mathsf{a}_{21}}|=\mathop{\mathtt{andnot}}(-\mathsf{0},\Re{\mathsf{a}_{21}});$}
    \COMMENT{$\mathop{\mathtt{andnot}}(-\mathsf{0},\mathsf{x})=\mathsf{x}\wedge\neg{-\mathsf{0}}\quad[\ref{ll:15}]$}
    \STATE{$|\Im{\mathsf{a}_{21}}|=\mathop{\mathtt{andnot}}(-\mathsf{0},\Im{\mathsf{a}_{21}});$}
    \COMMENT{$\mathtt{abs}$ could also be used to clear the sign bits$\quad[\ref{ll:15}]$}
    \STATE{$\mathop{\mathsf{sgn}}(\Re{\mathsf{a}_{21}})=\mathop{\mathtt{and}}(\Re{\mathsf{a}_{21}},-\mathsf{0});$}
    \COMMENT{$\mathop{\mathtt{and}}(\mathsf{x},-\mathsf{0})$ extracts the sign bits$\quad[\ref{ll:18}]$}
    \STATE{$|\mathsf{a}_{21}|=\mathop{\mathtt{hypot}}(|\Re{\mathsf{a}_{21}}|,|\Im{\mathsf{a}_{21}}|);$}
    \COMMENT{inline \cref{a:hypot}$\quad[\ref{ll:16},\ref{ll:17}]$}
    \STATE{$|\mathop{\mathsf{cos}}\alpha|=\mathop{\mathtt{min}}(\mathop{\mathtt{div}}(|\Re{\mathsf{a}_{21}}|,|\mathsf{a}_{21}|),\mathsf{1});$}
    \COMMENT{$\mathtt{min}$ replaces $0/0\to\mathtt{NaN}$ with $1\quad[\ref{ll:18}]$}
    \STATE{$\mathop{\mathsf{cos}}\alpha=\mathop{\mathtt{or}}(|\mathop{\mathsf{cos}}\alpha|,\mathop{\mathsf{sgn}}(\Re{\mathsf{a}_{21}}));$}
    \COMMENT{$\mathtt{or}$ transfers the sign bits to positive values$\ [\ref{ll:18}]$}
    \STATE{$\mathop{\mathsf{sin}}\alpha=\mathop{\mathtt{div}}(\Im{\mathsf{a}_{21}},\mathop{\mathtt{max}}(|\mathsf{a}_{21}|,\check{\bm{\mu}}));$}
    \COMMENT{$\mathtt{max}$ replaces $0$ with $\check{\mu}\quad[\ref{ll:19}]$\\[3pt]}
    \COMMENT{compute $\mathop{\mathsf{cos}}\varphi$ and $\mathsf{e}^{\mathsf{i}\alpha}\mathop{\mathsf{tan}}\varphi$ (or $\mathsf{e}^{\mathsf{i}\alpha}\mathop{\mathsf{sin}}\varphi$)\hfill}
    \STATE{$\mathsf{o}=\mathop{\mathtt{scalef}}(|\mathsf{a}_{21}|,\mathsf{1});\quad\mathsf{a}=\mathop{\mathtt{sub}}(\mathsf{a}_{11},\mathsf{a}_{22});$}
    \COMMENT{\cref{e:tan2}$\quad[\ref{ll:21}]$}
    \STATE{$|\mathsf{a}|=\mathop{\mathtt{andnot}}(-\mathsf{0},\mathsf{a});\quad\mathop{\mathsf{sgn}}(\mathsf{a})=\mathop{\mathtt{and}}(\mathsf{a},-\mathsf{0});$}
    \COMMENT{\cref{e:tan2}$\quad[\ref{ll:22}]$}
    \STATE{$\mathop{\mathsf{tan}}2\varphi=\mathop{\mathtt{or}}(\mathop{\mathtt{min}}(\mathop{\mathtt{max}}(\mathop{\mathtt{div}}(\mathsf{o},|\mathsf{a}|),\mathsf{0}),\sqrt{\bm{\nu}}),\mathop{\mathsf{sgn}}(\mathsf{a}));$}
    \COMMENT{\cref{e:tan2}$\quad[\ref{ll:22}]$}
    \STATE{$\mathop{\mathsf{sec}^{\mathsf{2}}}2\varphi=\mathop{\mathtt{fmadd}}(\mathop{\mathsf{tan}}2\varphi,\mathop{\mathsf{tan}}2\varphi,\mathsf{1});$}
    \COMMENT{$\sec^2{2\varphi}<\infty\quad[\ref{ll:23}]$}
    \STATE{$\mathop{\mathsf{tan}}\varphi=\mathop{\mathtt{div}}(\mathop{\mathsf{tan}}2\varphi,\mathop{\mathtt{add}}(\mathsf{1},\mathop{\mathtt{sqrt}}(\mathop{\mathsf{sec}^{\mathsf{2}}}2\varphi)));$}
    \COMMENT{\cref{e:jactan}$\quad[\ref{ll:23}]$}
    \STATE{$\mathop{\mathsf{sec}^\mathsf{2}}\varphi=\mathop{\mathtt{fmadd}}(\mathop{\mathsf{tan}}\varphi,\mathop{\mathsf{tan}}\varphi,1);$}
    \COMMENT{\cref{e:jacsec}$\quad[\ref{ll:24}]$}
    \STATE{$\mathop{\mathsf{sec}}\varphi=\mathop{\mathtt{sqrt}}(\mathop{\mathsf{sec}^{\mathsf{2}}}\varphi);\quad\mathop{\mathsf{cos}}\varphi=\mathop{\mathtt{div}}(\mathsf{1},\mathop{\mathsf{sec}}\varphi);$}
    \COMMENT{\cref{e:jacsec}$\quad[\ref{ll:24}]$}
    \STATE{$\mathop{\mathsf{cos}}\alpha\mathop{\mathsf{tan}}\varphi=\mathop{\mathtt{mul}}(\mathop{\mathsf{cos}}\alpha,\mathop{\mathtt{tan}}\varphi);$}
    \COMMENT{$\mathop{\mathsf{cos}}\alpha\mathop{\mathsf{sin}}\varphi=\mathop{\mathtt{div}}(\mathop{\mathsf{cos}}\alpha\mathop{\mathsf{tan}}\varphi,\mathop{\mathsf{sec}}\varphi);\quad[\ref{ll:25}]$\label{al:24}}
    \STATE{$\mathop{\mathsf{sin}}\alpha\mathop{\mathsf{tan}}\varphi=\mathop{\mathtt{mul}}(\mathop{\mathsf{sin}}\alpha,\mathop{\mathtt{tan}}\varphi);$}
    \COMMENT{$\mathop{\mathsf{sin}}\alpha\mathop{\mathsf{sin}}\varphi=\mathop{\mathtt{div}}(\mathop{\mathsf{sin}}\alpha\mathop{\mathsf{tan}}\varphi,\mathop{\mathsf{sec}}\varphi);\quad[\ref{ll:26}]$\label{al:25}\\[3pt]}
    \COMMENT{compute the eigenvalues (aligned stores of the results also happen here)\hfill}
    \STATE{$\bm{\lambda}_1'=\mathop{\mathtt{div}}(\mathop{\mathtt{fmadd}}(\mathop{\mathsf{tan}}\varphi,\mathop{\mathtt{fmadd}}(\mathsf{a}_{22}^{},\mathop{\mathsf{tan}}\varphi,\mathsf{o}),\mathsf{a}_{11}^{}),\mathop{\mathsf{sec}^{\mathsf{2}}}\varphi);$}
    \COMMENT{\cref{e:lamsec}$\quad[\ref{ll:28}]$}
    \STATE{$\bm{\lambda}_2'=\mathop{\mathtt{div}}(\mathop{\mathtt{fmadd}}(\mathop{\mathsf{tan}}\varphi,\mathop{\mathtt{fmsub}}(\mathsf{a}_{11}^{},\mathop{\mathsf{tan}}\varphi,\mathsf{o}),\mathsf{a}_{22}^{}),\mathop{\mathsf{sec}^{\mathsf{2}}}\varphi);$}
    \COMMENT{\cref{e:lamsec}$\quad[\ref{ll:29}]$}
    \STATE{$\bm{\lambda}_1^{}=\mathop{\mathtt{scalef}}(\bm{\lambda}_1',-\bm{\zeta});\quad\bm{\lambda}_2^{}=\mathop{\mathtt{scalef}}(\bm{\lambda}_2',-\bm{\zeta});$}
    \COMMENT{backscale $\bm{\lambda}_1^{}$ and $\bm{\lambda}_2^{}\quad[\ref{ll:30}]$}
    \STATE{$\mathfrak{p}=\mathop{\mathtt{\_mm512\_cmplt\_pd\_mask}}(\bm{\lambda}_1',\bm{\lambda}_2');$}
    \COMMENT{lane-wise check if $\lambda_1'<\lambda_2'\quad[\ref{ll:31}]$}
    \STATE{$\mathtt{p}[\mathtt{i}/\mathtt{s}]=\mathop{\mathtt{\_cvtmaskX\_u32}}(\mathfrak{p});$}
    \COMMENT{store $\mathfrak{p}$ to $\mathtt{p}$, $\mathtt{X}=\mathtt{8}$ ($\mathtt{16}$ for AVX512F)}
  \end{algorithmic}
\end{algorithm}

\begin{algorithm}[hbtp]
  \caption{$\mathtt{zbjac2}$: an OpenMP-parallel, AVX-512-vectorized eigendecomposition of a batch of $\tilde{r}$ double precision Hermitian matrices of order two.}
  \label{a:zbjac2}
  \begin{algorithmic}[1]
    \REQUIRE{$\tilde{r};\tilde{\mathbf{a}}_{11},\tilde{\mathbf{a}}_{22},\Re{\tilde{\mathbf{a}}_{21}},\Im{\tilde{\mathbf{a}}_{21}}$; $\mathtt{p}$ also, in the context of \cref{a:zvjsvd} only.}
    \ENSURE{$\cos\tilde{\bm{\varphi}},\cos\tilde{\bm{\alpha}}\tan\tilde{\bm{\varphi}},\sin\tilde{\bm{\alpha}}\tan\tilde{\bm{\varphi}};\ \ \tilde{\bm{\lambda}}_1,\tilde{\bm{\lambda}}_2;\ \ \mathtt{p}$}
    \COMMENT{\texttt{unsigned} array of length $\tilde{r}/\mathtt{s}$}
    \STATE{\texttt{\#pragma omp parallel for default(shared)}}
    \COMMENT{optional}
    \FOR[$\mathtt{i}=\ell-1$]{$\mathtt{i}=0$ \TO $\tilde{r}-1$ \textbf{step} $\mathtt{s}$}
    \STATE{\textbf{if} $\mathtt{p}[\mathtt{i}/\mathtt{s}]=0$ \textbf{then continue};}
    \COMMENT{skip this vector on request of \cref{a:zvjsvd}}
    \STATE{$\mathop{\mathtt{z8jac2}}(\mathtt{i},\tilde{\mathbf{a}}_{11},\tilde{\mathbf{a}}_{22},\Re{\tilde{\mathbf{a}}_{21}},\Im{\tilde{\mathbf{a}}_{21}},\cos\tilde{\bm{\alpha}}\tan\tilde{\bm{\varphi}},\sin\tilde{\bm{\alpha}}\tan\tilde{\bm{\varphi}},\cos\tilde{\bm{\varphi}},\tilde{\bm{\lambda}}_1,\tilde{\bm{\lambda}}_2,\mathtt{p}[,-\tilde{\bm{\zeta}}]);$}
    \COMMENT{\Cref{a:z8jac2} should be inlined above to avoid a function-call overhead}
    \ENDFOR\COMMENT{$-\tilde{\bm{\zeta}}$ has to be returned without the optional backscaling of $\tilde{\bm{\lambda}}_1'$ and $\tilde{\bm{\lambda}}_2'$}
  \end{algorithmic}
\end{algorithm}
\section{Column transformations, matrix scaling, and dot-products}\label{s:3}
A na\"{\i}ve formation of the Grammian pivot matrices by computing the
three required dot-products without prescaling the columns is
susceptible to overflow and underflow~\cite{Drmac-97} and thus
severely restricts the admissible exponent range of the elements of
the iteration matrix, as the analysis in \cref{s:SM3} further
demonstrates.  From here onwards a robust though not as performant
implementation of the Jacobi SVD is considered.
\subsection{Effects of the Jacobi rotations on the elements' magnitudes}\label{ss:3.1}
Transforming a pair of columns as
{$\addtolength{\arraycolsep}{-3pt}\begin{bmatrix}g_p^{}&g_q^{}\end{bmatrix}U=\begin{bmatrix}g_p'&g_q'\end{bmatrix}$},
by any Jacobi rotation $U$, cannot raise the larger magnitude of the
elements from any row $i$ of the pair by more than $\sqrt{2}$ in
\emph{exact} arithmetic.  For any $\varphi$,
$|\cos\varphi|+|\sin\varphi|\le\sqrt{2}$, and for any $\alpha$,
$|\pm\mathrm{e}^{\pm\mathrm{i}\alpha}|=1$, and for any index pair
$(p,q)$, where $p<q$ and $g_p'$ and $g_q'$ are not to be swapped,
\begin{displaymath}
  \begin{gathered}
    g_{ip}'=g_{ip}^{}\cos\varphi+g_{iq}^{}\mathrm{e}^{\mathrm{i}\alpha}\sin\varphi,\qquad
    g_{iq}'=g_{iq}^{}\cos\varphi-g_{ip}^{}\mathrm{e}^{-\mathrm{i}\alpha}\sin\varphi,\\
    \max\{|g_{ip}'|,|g_{iq}'|\}\le\max\{|g_{ip}^{}|,|g_{iq}^{}|\}(|\cos\varphi|+|\sin\varphi|)\le\sqrt{2}\max\{|g_{ip}^{}|,|g_{iq}^{}|\}.
  \end{gathered}
\end{displaymath}

\looseness=-1
All quantities involving $\alpha$ and $\varphi$ are computed, and
therefore may not be exact.  Applying a computed Jacobi rotation,
where
$\mathrm{e}'=\mathop{\mathrm{fl}}(\mathop{\mathrm{fl}}(\mathrm{e}^{\mathrm{i}\alpha})\mathop{\mathrm{fl}}(\tan\varphi))$,
is in fact done as
\begin{equation}
  \begin{gathered}
    g_{ip}'=(g_{ip}^{}+g_{iq}^{}\mathrm{e}')\mathop{\mathrm{fl}}(\cos\varphi),\quad
    g_{iq}'=(g_{iq}^{}-g_{ip}^{}\bar{\mathrm{e}}')\mathop{\mathrm{fl}}(\cos\varphi),\\
    \mathop{\mathrm{fl}}(g_{ip}')=\mathop{\mathrm{fl}}(\mathop{\mathrm{fma}}(g_{iq}^{},\mathrm{e}',g_{ip}^{})\mathop{\mathrm{fl}}(\cos\varphi)),\quad
    \mathop{\mathrm{fl}}(g_{iq}')=\mathop{\mathrm{fl}}(\mathop{\mathrm{fma}}(g_{ip}^{},-\bar{\mathrm{e}}',g_{iq}^{})\mathop{\mathrm{fl}}(\cos\varphi)),
  \end{gathered}
  \label{e:transf}
\end{equation}
using the complex fused multiply-add from~\cref{e:zfma}.  However, the
components of the result of either $\mathrm{fma}$ from~\cref{e:transf}
are larger in magnitude by the factor
$\approx 1/\mathop{\mathrm{fl}}(\cos\varphi)$ than those of the final
result.  Overflow in those intermediate computations is avoided by
rescaling the iteration matrix as in \cref{ss:3.2}.  Similar holds for
real transformations, with $\mathrm{e}^{\mathrm{i}\alpha}=\pm 1$ and
the real fused multiply-add in~\cref{e:transf}.  \Cref{l:rt,l:ct},
proven in \cref{s:SM4}, bound, in the real and the complex case,
respectively, relative growth of all, or the ``important'',
transformed element's magnitudes in~\cref{e:transf}, caused by the
rounding errors, by a modest multiple of $\varepsilon$.

\begin{lemma}\label{l:rt}
  Assume that all input floating-point values in~\cref{e:transf} are
  real and finite, and neither overflow nor underflow occurs in any
  rounding of those computations.  Let
  $\mathop{\mathrm{fl}}(g_{ip}')=g_{ip}'(1+\varepsilon_p')$ and
  $\mathop{\mathrm{fl}}(g_{iq}')=g_{iq}'(1+\varepsilon_q')$.  Then,
  \begin{displaymath}
    (1-\varepsilon)^2\le 1+\varepsilon_p'\le(1+\varepsilon)^2,\qquad(1-\varepsilon)^2\le 1+\varepsilon_q'\le(1+\varepsilon)^2.
  \end{displaymath}
  With no other assumptions than computing with the
  standard~\cite{IEEE-754-2019} floating-point arithmetic, if
  $\max\{|g_{ip}^{}|,|g_{iq}^{}|\}\le\nu/2$ then
  $\max\{|\mathop{\mathrm{fl}}(g_{ip}')|,|\mathop{\mathrm{fl}}(g_{iq}')|\}\le\nu$.
\end{lemma}

\begin{lemma}\label{l:ct}
  Assume that the complex $\mathrm{fma}$ from~\cref{e:zfma} is used
  in~\cref{e:transf}, no input value has a non-finite component, and
  neither overflow nor underflow occurs in any rounding of those
  computations.  With $p\ge 23$,
  $\tilde{\epsilon}=3.000001\,\varepsilon$, and
  $\tilde{\epsilon}''=5.656856\,\varepsilon$,
  \begin{compactenum}[1.]
  \item if $|g_{ip}'|>0$ and
    $|g_{ip}'|/\mathop{\mathrm{fl}}(\cos\varphi)\ge|\Re{g_{iq}^{}}|$,
    then
    $|\mathop{\mathrm{fl}}(g_{ip}')-g_{ip}'|/|g_{ip}'|\le\tilde{\epsilon}''$,
    else, if
    $|g_{iq}^{}|\le\nu/(1+\sqrt{2}\tilde{\epsilon})$,
    then $|\mathop{\mathrm{fl}}(g_{ip}')|<\nu$;
  \item if $|g_{iq}'|>0$ and
    $|g_{iq}'|/\mathop{\mathrm{fl}}(\cos\varphi)\ge|\Re{g_{ip}^{}}|$,
    then
    $|\mathop{\mathrm{fl}}(g_{iq}')-g_{iq}'|/|g_{iq}'|\le\tilde{\epsilon}''$,
    else, if
    $|g_{ip}^{}|\le\nu/(1+\sqrt{2}\tilde{\epsilon})$, then
    $|\mathop{\mathrm{fl}}(g_{iq}')|<\nu$.
  \end{compactenum}
  With no other assumptions than computing with the
  standard~\cite{IEEE-754-2019} floating-point arithmetic, if
  $\max\{|g_{ip}^{}|,|g_{iq}^{}|\}\le\nu/4$ then
  $\max\{|\mathop{\mathrm{fl}}(g_{ip}')|,|\mathop{\mathrm{fl}}(g_{iq}')|\}\le\nu$.
\end{lemma}

\Cref{p:t} follows directly from the last statements of
\cref{l:rt,l:ct}.  At the start of each step $k\ge 1$ it is assumed
anew that $G_k=\mathop{\mathrm{fl}}(G_k)$, i.e., the current
floating-point representation of the iteration matrix is taken as
exact.

\begin{proposition}\label{p:t}
  Assume $G_k=\mathop{\mathrm{fl}}(G_k)$.  If
  $\|G_k\|_{\max}\le\nu/\varsigma^{\mathbb{F}}$, where
  $\varsigma^{\mathbb{R}}=2$ for $G_k$ real and
  $\varsigma^{\mathbb{C}}=4$ for $G_k$ complex, then
  $\|\mathop{\mathrm{fl}}(G_{k+1})\|_{\max}\le\nu$.
\end{proposition}
\subsection{Periodic rescaling of the iteration matrix in the robust SVD}\label{ss:3.2}
\looseness=-1
The state-of-the-art construction of the diagonalizing Jacobi
rotation for a pivot Grammian matrix, prescaled by the inverse of the
product of the Frobenius pivot column norms, from~\cite[Eq.~(2.13) and
  Algorithm~2.5]{Drmac-97} and LAPACK, overcomes all range limitations
except a possible overflow/underflow of a quotient of those norms.
It requires a procedure for computing the column norms without undue
overflow, like the ones from the reference BLAS routines
(\texttt{S}/\texttt{D})\texttt{NRM2}~\cite[Algorithm~2]{Anderson-17}
and (\texttt{SC}/\texttt{DZ})\texttt{NRM2}.

There are several options for the Frobenius norm computation in the
robust Jacobi SVD, as shown in \cref{s:SM6}.  Three of them are:
\begin{compactenum}[1.]
\item rely on the BLAS routines, while ensuring that the column norms
  cannot overflow by an adequate rescaling of the iteration matrix
  (those routines are \emph{conditionally} reproducible on any fixed
  platform, depending on the instruction subset, e.g.), or;
\item compute the norms as the square roots of dot-products in a
  data\-type with a wider exponent range, e.g., in the Intel's
  extended 80-bit datatype or in the standard~\cite{IEEE-754-2019}
  quadruple 128-bit datatype for double precision inputs, with the
  latter choice being portably vectorizable with the SLEEF's
  quad-precision math library, or;
\item extend a typical vectorized dot-product procedure such that it
  computes the norms with the intermediate and the output data
  represented as having the same precision (i.e., the significands'
  width) as the input elements, but almost unbounded exponents, as
  proposed in~\cite[Appendix~A]{Novakovic-15} for the Jacobi-type SVD
  on GPUs.
\end{compactenum}

In \cref{s:SM6} the third option above is described in the context of
SIMD processing and all options are evaluated, with the conclusion
that on Intel's platforms the first one, but with the MKL's routines
instead of the reference ones, is the most performant and quite
accurate.  Also, for all options the following representation of the
computed norms is proposed.  Let $\|\mathbf{x}\|_F=(e,f)$, where $e$
and $f$ are quantities of the input's datatype $D$, such that $(e,f)$
represents the value $2^ef$.  Let for a non-zero finite value hold
$-\infty<e<\infty$ and $1\le f<2$, while $0=(-\infty,1)$.  All such
values from $D$ can thus be represented exactly, while the results
(but only those computed with wider exponents) that would overflow if
rounded to $D$ are preserved as finite.

Since $\|G\|_F\le\sqrt{mn}\|G\|_{\max}$, to ensure
$\|G\|_F\le\nu/\varsigma^{\mathbb{F}}$ it suffices to scale $G$
such that $\|G\|_{\max}\le\nu/(\varsigma^{\mathbb{F}}\sqrt{mn})$.
With $G$ complex,
$\|G\|_{\max}\le\sqrt{2}\max\{\|\Re{G}\|_{\max},\|\Im{G}\|_{\max}\}$,
so it suffices to have
$\max\{\|\Re{G}\|_{\max},\|\Im{G}\|_{\max}\}\le\nu/(\varsigma^{\mathbb{C}}\sqrt{2mn})$.
The Jacobi transformations are unitary, and the Frobenius norm is
unitarily invariant, so $\|G_k\|_F=\|G\|_F$ for all $k\ge 1$, if the
rounding errors are ignored.  No overflow can occur in~\cref{e:transf}
with $G$ scaled as described, due to \cref{p:t} and
$\|G_k\|_{\max}\le\|G_k\|_F=\|G\|_F\le\nu/\varsigma^{\mathbb{F}}$.

Without rounding errors, the previous paragraph would define the
required initial scaling of $G$ in terms of $\|G\|_{\max}$ (or
$\|\Re{G}\|_{\max}$ and $\|\Im{G}\|_{\max}$), and no overflow checks
should be needed, since $\|(G_k)_j\|_F\le\|G_k\|_F$ for any column
$j$.  However, as discussed, some algorithms for the Frobenius norm
might overflow even when the result should be finite, and the rounding
within the transformations could further raise the magnitudes of the
elements.  In the absence of another theoretical bound for a
particular Frobenius norm routine, assume\footnote{This assumption can
be turned into a user-provided parameter to the Jacobi SVD routine
that holds a theoretically or empirically established upper bound on
the array elements' magnitudes for non-overflowing execution of the
chosen Frobenius norm procedure, with~\cref{e:skn} adjusted
accordingly.} that for any $k\ge 1$ and $1\le j\le n$,
$\|(G_k)_j\|_{\max}\le\nu/(\varsigma^{\mathbb{F}}m)$, and so
$\|G_k\|_{\max}\le\nu/(\varsigma^{\mathbb{F}}m)$, implies
$\mathop{\mathrm{fl}}(\|(G_k)_j\|_F)\le\nu$.  Note that
$m\ge\sqrt{mn}$ since $m\ge n$.  At the start of each step $k$ it then
suffices to have $\|G_k\|_{\max}\le\nu/(\varsigma^{\mathbb{R}}m)$
in the real, and
$\max\{\|\Re{G}\|_{\max},\|\Im{G}\|_{\max}\}\le\nu/(\varsigma^{\mathbb{C}}m\sqrt{2})$
in the complex case, to avoid any overflow in that step.  Since
$\varsigma^{\mathbb{F}}$ is a power of two,
$\nu/\varsigma^{\mathbb{F}}$ is an exactly representable quantity
that, for its exponent, has the largest significand possible (all
ones), while this property might not hold for the other upper bounds
for the max-norms.

Let $2^{s_{[k]}^{\mathbb{F}}}G_k^{}$ be a scaling of $G_k^{}$, where
the exponent $s_{[k]}^{\mathbb{F}}\in\mathbb{Z}$ is
\begin{equation}
  s_{[k]}^{\mathbb{R}}=\lfloor\lg(\nu/(\varsigma^{\mathbb{R}}m))\rfloor-\lfloor\lg{\widetilde{M}_k^{\mathbb{R}}}\rfloor-1,\quad
  s_{[k]}^{\mathbb{C}}=\lfloor\lg(\nu/(\varsigma^{\mathbb{C}}m\sqrt{2}))\rfloor-\lfloor\lg{\widetilde{M}_k^{\mathbb{C}}}\rfloor-1,
  \label{e:skn}
\end{equation}
with $\widetilde{M}_k^{\mathbb{R}}=\|G_k^{}\|_{\max}^{}$ and
$\widetilde{M}_k^{\mathbb{C}}=\max\{\|\Re{G}\|_{\max}^{},\|\Im{G}\|_{\max}^{}\}$.
Instead of comparing the significand of the upper bound (that should
have then been rounded downwards) with that of
$\widetilde{M}_k^{\mathbb{F}}$ and deciding whether to subtract unity
from the scaling exponent if the former is smaller than the latter,
the easiest but potentially suboptimal way to build the scaling
exponents in~\cref{e:skn} and~\cref{e:skr} below are the unconditional
subtractions of unity when the upper bound is not
$\nu/\varsigma^{\mathbb{F}}$.  If a computation of the column norms
cannot overflow, then, due to \cref{p:t}, a more relaxed scaling
$2^{s_{(k)}^{\mathbb{F}}}G_k^{}$, where
\begin{equation}
  s_{(k)}^{\mathbb{R}}=\lfloor\lg(\nu/\varsigma^{\mathbb{R}})\rfloor-\lfloor\lg{\widetilde{M}_k^{\mathbb{R}}}\rfloor,\quad
  s_{(k)}^{\mathbb{C}}=\lfloor\lg(\nu/(\varsigma^{\mathbb{C}}\sqrt{2}))\rfloor-\lfloor\lg{\widetilde{M}_k^{\mathbb{C}}}\rfloor-1,
  \label{e:skr}
\end{equation}
is sufficient to protect the transformations from overflowing when
forming $G_{k+1}$.

Observe that~\cref{e:skr} protects from a \emph{destructive} action,
i.e., from overflowing while replacing a pivot column pair of the
iteration matrix with its transformed counterpart.  No recovery is
possible from such an event without either keeping a copy of the
original column pair or checking the magnitudes of the transformed
elements before storing them, both of which slow down the execution.
In contrast, overflow of a computed norm is \emph{non-destructive},
and can be recovered (and protected) from by downscaling $G_k$
according to~\cref{e:skn} and computing the norms of the scaled
columns.

It is expensive to rescale $G_k$ at the start of every step.  The
following rescaling heuristic is thus proposed, that delays scaling
unless a destructive operation is possible:
\begin{compactenum}[1.]
\item
  Let $G_0^{}=G$ and $s_0^{}=s_{[0]}^{\mathbb{F}}$, before the
  iterative part of the Jacobi SVD\@.  Then, let
  $G_1^{}=2^{s_0^{}}G_0^{}$ be the initial iteration matrix.  If all
  elements of $G$ are small enough by magnitude, this can imply
  upscaling ($s_0^{}>0$) and as many subnormal values as safely
  possible, if they exist in $G$, become normal\footnote{This
  upscaling, i.e., raising of the magnitudes tries to keep the
  elements of the rotated column pairs from falling into the subnormal
  range if a huge but not total cancellation in~\cref{e:transf}
  occurs.} in $G_1^{}$.  Otherwise, $s_0^{}\le 0$.  Also, let
  $\widetilde{M}_1^{\mathbb{F}}=2^{s_0^{}}\widetilde{M}_0^{\mathbb{F}}$,
  where $\widetilde{M}_0^{\mathbb{F}}$ has been found by a method
  described below.
\item At the start of each step $k\ge 1$, compute
  $s_{(k)}^{\mathbb{F}}$ from~\cref{e:skr} \emph{and}
  $s_{[k]}^{\mathbb{F}}$ from~\cref{e:skn}, using
  $\widetilde{M}_{k}^{\mathbb{F}}$ determined at the end of previous
  step.  If $s_{(k)}^{\mathbb{F}}<0$, $G_k^{}$ has to be downscaled to
  $G_k'$ and $\widetilde{M}_k^{\mathbb{F}}$ updated.  For that, take
  the lower exponent $s_k'=s_{[k]}^{\mathbb{F}}$, since it will
  protect the subsequent computation of the column norms as well.
  Otherwise, let $s_k'=0$ and $G_k'=G_k^{}$.  Define
  $s_k^{}=s_{k-1}^{}+s_k'$ as the effective scaling exponent of the
  initial $G$, i.e., $2^{-s_k^{}}G_k'$ is what the iteration matrix
  would be without any scaling.
\item If any column norm of $G_k'$ overflows, rescale $G_k'$ to
  $G_k''$ using $s_k''=s_{[k]}^{\mathbb{F}\prime}-s_k'$, where
  $s_{[k]}^{\mathbb{F}\prime}\le s_{[k]}^{\mathbb{F}}$, let
  $s_k^{}=s_k^{}+s_k''$, recompute the norms, and update
  $\widetilde{M}_k^{\mathbb{F}}$.  Else, $G_k''=G_k'$ and $s_k''=0$.
  In \cref{ss:4.3} a robust procedure for determining
  $s_{[k]}^{\mathbb{F}\prime}$ is described.
\item While applying the Jacobi rotations to transform $G_k''$ to
  $G_{k+1}^{}$, compute $\widetilde{M}_{k+1}^{\mathbb{F}}$.  This can
  be done efficiently by reusing portions of a transformed pivot
  column pair already present in the CPU registers, but at the expense
  of implementing the custom rotation kernels instead of relying on
  the BLAS routines \texttt{xROTM} and \texttt{xSCAL}.
\end{compactenum}
If $\widetilde{M}_0^{\mathbb{F}}=\infty$, i.e., if $G$ contains a
non-finite value, the Jacobi SVD algorithm fails.  It is assumed
that $G$ is otherwise of full column rank, so
$\widetilde{M}_0^{\mathbb{F}}>0$ (else, the routine stops).

When the Jacobi process has numerically converged after $K$ steps, for
some $K$, the scaled singular values $\Sigma'$ of $G$ have to be
scaled back by $2^{-s}$, $s=s_K^{}$.  If $\sigma_j'$ were represented
as an ordinary floating-point value, such a backscaling could have
caused the result's overflow or an undesired
underflow~\cite{Novakovic-20}.  However, $\sigma_j'$ is computed as
the Frobenius norm of the $j$th column of the final iteration matrix
and is thus represented as $\sigma_j'=(e_j',f_j')$, making any
overflow or underflow of the backscaled
$\sigma_j^{}=2^{-s}\sigma_j'=(e_j'-s,f_j')$ impossible, unless
$\sigma_j^{}$ is converted to a floating-point value.

Rescaling of each column of $G_k$ is trivially vectorizable by the
$\mathtt{scalef}$ intrinsic, and the columns can be processed
concurrently.  The $\max$-norm of a real matrix is computed as a
parallel $\max$-reduction of the columns' $\max$-norms.  For a column
$\mathbf{x}$ let $\mathsf{x}$ be a vector of zeros, and load
consecutive vector-sized chunks $\mathsf{y}$ of $\mathbf{x}$ in a
loop.  For each loaded chunk, update the partial maximums in
$\mathsf{x}$ as
$\mathsf{x}=\mathop{\mathtt{max}}(\mathsf{x},\mathop{\mathtt{min}}(\mathop{\mathtt{abs}}(\mathsf{y}),\mathsf{\infty}))$,
converting any encountered $\mathtt{NaN}$ in $\mathsf{y}$ into
$\infty$.  After the loop, let
$\|\mathbf{x}\|_{\max}=\mathop{\mathtt{reduce\_max}}(\mathsf{x})$.

The $\max$-norm of a complex matrix $G_k$ is approximated, as
described, by the maximum of two real $\max$-norms,
$\|G_k\|_{\max}\le\sqrt{2}\max\{\|\Re{G_k}\|_{\max},\|\Im{G_k}\|_{\max}\}$.
The Frobenius norm of a complex column $\mathbf{x}$ is obtained as
$\|\mathbf{x}\|_F=\mathop{\mathrm{hypot}}(\|\Re{\mathbf{x}}\|_F,\|\Im{\mathbf{x}}\|_F)$.
\subsection{The scaled dot-products}\label{ss:3.3}
\looseness=-1
Let $\begin{bmatrix}g_p\!\!&\!\!g_q\end{bmatrix}$, $p<q$, be a
pivot column pair from $G_k''$, and $0<\|g_j\|_F=(e_j,f_j)$,
$\check{g}_j=g_j/\|g_j\|_F$, for $j\in\{p,q\}$.  The scaled complex
dot-product $\check{g}_q^{\ast}\check{g}_p^{}$ is computed as in
\cref{a:zdpscl}, with a single division operation.  In the loop of
\cref{a:zdpscl}, $g_j^{}$ is prescaled to $g_j^{}/2^{e_j^{}}$, with
its Frobenius norm $\approx f_j^{}$.  The components of the resulting
$\hat{z}=\mathop{\mathrm{fl}}((g_q^{}/2^{e_q^{}})^{\ast}(g_p^{}/2^{e_p^{}}))$
and
$z=\mathop{\mathrm{fl}}(\check{g}_q^{\ast}\check{g}_p^{})=\mathop{\mathrm{fl}}(\hat{z}/(f_q^{}\cdot f_p^{}))$
thus cannot overflow.  A slower but possibly more accurate routine due
to the compensated summation of the partial scaled dot-products,
$\mathtt{zdpscl}'$, is given as \cref{a:zdpsclcs} and was used in the
testing from \cref{s:5}.

\begin{algorithm}[hbtp]
  \caption{$\mathtt{zdpscl}$: a vectorized complex scaled dot-product routine.}
  \label{a:zdpscl}
  \begin{algorithmic}[1]
    \REQUIRE{$g_q=(\Re{g_q},\Im{g_q}),0<\|g_q\|_F=(e_q,f_q);g_p=(\Re{g_p},\Im{g_p}),0<\|g_p\|_F=(e_p,f_p)$.}
    \ENSURE{$z=(\Re{z},\Im{z})=\mathop{\mathrm{fl}}(\check{g}_q^{\ast}\check{g}_p^{}=g_q^{\ast}g_p^{}/(\|g_q^{}\|_F^{}\|g_p^{}\|_F^{}))$.}
    \STATE{$\Re{\hat{\mathsf{z}}}=\mathop{\mathtt{setzero}}();\quad\Im{\hat{\mathsf{z}}}=\mathop{\mathtt{setzero}}();\quad-\mathsf{e}_j=\mathop{\mathtt{set1}}(-e_j);$}
    \COMMENT{$j\in\{p,q\}$}
    \FOR[sequentially]{$\mathtt{i}=0$ \TO $\tilde{m}-1$ \textbf{step} $\mathtt{s}$}
    \STATE{$\Re{\mathsf{g}_{\mathtt{i}j}}=\mathop{\mathtt{load}}(\Re{g_j}+\mathtt{i});\quad\Im{\mathsf{g}_{\mathtt{i}j}}=\mathop{\mathtt{load}}(\Im{g_j}+\mathtt{i});$}
    \COMMENT{$j\in\{p,q\}$ here and below}
    \STATE{$\Re{\check{\mathsf{g}}_{\mathtt{i}j}}=\mathop{\mathtt{scalef}}(\Re{\mathsf{g}_{\mathtt{i}j}},-\mathsf{e}_j);\quad\Im{\check{\mathsf{g}}_{\mathtt{i}j}}=\mathop{\mathtt{scalef}}(\Im{\mathsf{g}_{\mathtt{i}j}},-\mathsf{e}_j);$}
    \COMMENT{division by $\mathsf{2}^{\mathsf{e}_j}$\\}
    \COMMENT{$\Re{\hat{\mathsf{z}}}=\Re{\hat{\mathsf{z}}}+\Re{\check{\mathsf{g}}_{\mathtt{i}q}}\cdot\Re{\check{\mathsf{g}}_{\mathtt{i}p}}+\Im{\check{\mathsf{g}}_{\mathtt{i}q}}\cdot\Im{\check{\mathsf{g}}_{\mathtt{i}p}};\quad\Im{\hat{\mathsf{z}}}=\Im{\hat{\mathsf{z}}}+\Re{\check{\mathsf{g}}_{\mathtt{i}q}}\cdot\Im{\check{\mathsf{g}}_{\mathtt{i}p}}-\Im{\check{\mathsf{g}}_{\mathtt{i}q}}\cdot\Re{\check{\mathsf{g}}_{\mathtt{i}p}}$\hfill}
    \STATE{$\Re{\hat{\mathsf{z}}}=\mathop{\mathtt{fmadd}}(\Re{\check{\mathsf{g}}_{\mathtt{i}q}},\Re{\check{\mathsf{g}}_{\mathtt{i}p}},\Re{\hat{\mathsf{z}}});\quad\Im{\hat{\mathsf{z}}}=\mathop{\mathtt{fmadd}}(\Re{\check{\mathsf{g}}_{\mathtt{i}q}},\Im{\check{\mathsf{g}}_{\mathtt{i}p}},\Im{\hat{\mathsf{z}}});$}
    \STATE{$\Re{\hat{\mathsf{z}}}=\mathop{\mathtt{fmadd}}(\Im{\check{\mathsf{g}}_{\mathtt{i}q}},\Im{\check{\mathsf{g}}_{\mathtt{i}p}},\Re{\hat{\mathsf{z}}});\quad\Im{\hat{\mathsf{z}}}=\mathop{\mathtt{fnmadd}}(\Im{\check{\mathsf{g}}_{\mathtt{i}q}},\Re{\check{\mathsf{g}}_{\mathtt{i}p}},\Im{\hat{\mathsf{z}}});$}
    \ENDFOR\COMMENT{$g_q$ divided by $2^{e_q}$, $g_p$ by $2^{e_p}$}
    \STATE{$\Re{\hat{z}}=\mathop{\mathtt{reduce\_add}}(\Re{\hat{\mathsf{z}}});\quad\Im{\hat{z}}=\mathop{\mathtt{reduce\_add}}(\Im{\hat{\mathsf{z}}});$}
    \COMMENT{reduce the partial sums}
    \STATE{$\mskip-7mu\mathop{\mathtt{\_mm\_store\_pd}}(\&z,\!\mathop{\mathtt{\_mm\_div\_pd}}(\mathop{\mathtt{\_mm\_set\_pd}}(\Im{\hat{z}},\!\Re{\hat{z}}),\!\mathop{\mathtt{\_mm\_set1\_pd}}(f_q\!\cdot\!f_p)\mskip-2mu)\mskip-2mu);$}
    \RETURN{$z;$}
    \COMMENT{above: single $2\times 64$-bit division of $\hat{z}=(\Re{\hat{z}},\Im{\hat{z}})$ by $1\le f_q\cdot f_p<4$}
  \end{algorithmic}
\end{algorithm}
\subsubsection{The convergence criterion}\label{sss:3.3.1}
Following~\cite{Drmac-97} and the LAPACK's \texttt{xGESVJ} routines, a
pivot column pair $(p,q)$ of the iteration matrix is not transformed
(but the columns and their norms might be swapped) if it is
numerically orthogonal, i.e., if
\begin{equation}
  \mathop{\mathrm{fl}}(|\check{g}_q^{\ast}\check{g}_p^{}|)<\mathop{\mathrm{fl}}(\varepsilon\sqrt{m})=\upsilon.
  \label{e:tol}
\end{equation}

The Jacobi process stops successfully if no transformations in a sweep
over all pivot pairs have been performed (and thus the convergence has
been detected), or unsuccessfully if the convergence has not been
detected in the prescribed number of sweeps $\mathtt{S}$.  In LAPACK,
$\mathtt{S}=30$, but this might be insufficient, as shown in
\cref{s:5}.

Assume that, for a chosen sequence of $\mathtt{s}$ pivot pair indices
$(p_1,q_1),\cdots,(p_{\mathtt{s}},q_{\mathtt{s}})$, the respective scaled
dot-products have already been computed and packed into vectors
\begin{displaymath}
  \Re{\mathsf{a}_{21}'}=(\mathop{\mathrm{fl}}(\Re(\check{g}_{q_{\ell}^{}}^{\ast}\check{g}_{p_{\ell}^{}}^{})))_{\ell}^{},\quad
  \Im{\mathsf{a}_{21}'}=(\mathop{\mathrm{fl}}(\Im(\check{g}_{q_{\ell}^{}}^{\ast}\check{g}_{p_{\ell}^{}}^{})))_{\ell}^{},\quad
  1\le\ell\le\mathtt{s}.
\end{displaymath}
\Cref{a:cvg} checks the convergence criterion over all vectors' lanes,
encodes the result as a bitmask, and counts how many transformations
should be performed.

\begin{algorithm}[hbtp]
  \caption{Vectorized checking of the convergence criterion.}
  \label{a:cvg}
  \begin{algorithmic}[1]
    \REQUIRE{The vectors $\Re{\mathsf{a}_{21}'}$ and $\Im{\mathsf{a}_{21}'}$; $\upsilon$.}
    \ENSURE{The bitmask $\mathfrak{c}$, $\mathfrak{c}_{\ell-1}^{}=1\iff(p_{\ell}^{},q_{\ell}^{})$ should be transformed; $\sum_{\ell=1}^{\mathtt{s}}\mathfrak{c}_{\ell-1}^{}$.}
    \STATE{$\bm{\upsilon}=\mathop{\mathtt{set1}}(\upsilon);\qquad|\mathsf{a}_{21}'|=\mathop{\mathtt{hypot}}(\Re{\mathsf{a}_{21}'},\Im{\mathsf{a}_{21}'});$}
    \COMMENT{$\mathop{\mathrm{fl}}(|\mathsf{a}_{21}'|)$}
    \STATE{$\mathfrak{c}=\mathop{\mathtt{\_mm512\_cmple\_pd\_mask}}(\bm{\upsilon},|\mathsf{a}_{21}'|);$}
    \COMMENT{$\neg$\cref{e:tol}}
    \RETURN{$\mathop{\mathtt{\_mm\_popcnt\_u32}}(\mathop{\mathtt{\_cvtmaskX\_u32}}(\mathfrak{c}));$}
    \COMMENT{$\mathtt{X}=\mathtt{8}$ ($\mathtt{16}$ for AVX512F)}
  \end{algorithmic}
\end{algorithm}
\subsection{Formation of the scaled Grammians}\label{ss:3.4}
Let $\begin{bmatrix}g_p\!\!&\!\!g_q\end{bmatrix}$, $p<q$, be a
pivot column pair from $G_k''$, and $0<\|g_j\|_F=(e_j,f_j)$ for
$j\in\{p,q\}$.  The scaled Grammian
\begin{displaymath}
  A_{[pq]}'=\frac{1}{\|g_p^{}\|_F^{}\|g_q^{}\|_F^{}}\begin{bmatrix}g_p^{}\!\!&\!\!g_q^{}\end{bmatrix}^{\ast}\begin{bmatrix}g_p^{}\!\!&\!\!g_q^{}\end{bmatrix}=
  \begin{bmatrix}
    a_{11}'&\bar{a}_{21}'\\
    a_{21}'&a_{22}'
  \end{bmatrix}
\end{displaymath}
can be expressed in the terms of $\|g_j\|_F=2^{e_j}f_j$ as
\begin{equation}
  a_{11}'=\frac{\|g_p^{}\|_F^{}}{\|g_q^{}\|_F^{}}=2^{e_p^{}-e_q^{}}\frac{f_p^{}}{f_q^{}},\quad
  a_{22}'=\frac{\|g_q^{}\|_F^{}}{\|g_p^{}\|_F^{}}=2^{e_q^{}-e_p^{}}\frac{f_q^{}}{f_p^{}},\quad
  a_{21}'=\frac{g_q^{\ast}}{\|g_q^{}\|_F^{}}\frac{g_p^{}}{\|g_p^{}\|_F^{}}.
  \label{e:sGramEF}
\end{equation}

\looseness=-1
\Cref{a:zdpscl} computes $a_{21}'$ without overflow, and neither
$a_{11}'=(e_p^{}-e_q^{},f_p^{}/f_q^{})$ nor
$a_{22}'=(e_q^{}-e_p^{},f_q^{}/f_p^{})$ can overflow in this
``non-normalized'' representation (but can as floating-point values).
To call the batched EVD routine with scaled Grammians as inputs, they
have to be scaled further to the representable range by
$2^{s_{[pq]}^{}}$, $s_{[pq]}^{}\le 0$, i.e., to
$A_{[pq]}''=2^{s_{[pq]}^{}}A_{[pq]}'$.  Let
$e'=\left\lfloor\lg(f_p^{}/f_q^{})\right\rfloor$ and
$e''=\left\lfloor\lg(f_q^{}/f_p^{})\right\rfloor$.  Then, define
\begin{equation}
  e_{p/q}^{}=e_p^{}-e_q^{}+e',\quad
  e_{q/p}^{}=e_q^{}-e_p^{}+e'',\quad
  f_{p/q}^{}=2^{-e'}\frac{f_p^{}}{f_q^{}},\quad
  f_{q/p}^{}=2^{-e''}\frac{f_q^{}}{f_p^{}}.
  \label{e:efpq}
\end{equation}
The normalized representations of the diagonal elements of $A_{[pq]}'$
are $a_{11}'=(e_{p/q}^{},f_{p/q}^{})$ and
$a_{22}'=(e_{q/p}^{},f_{q/p}^{})$, where the fractional parts lie in
$[1,2\rangle$ and the exponents are still finite.  Let
$\hat{\eta}=\mathtt{DBL\_MAX\_EXP}-1$ be the largest exponent of a
finite floating-point value, and
$e_{\ge}^{}=\max\{e_{p/q}^{},e_{q/p}^{}\}$.  Then,
$s_{[pq]}^{}=\min\{\hat{\eta}-e_{\ge}^{},0\}$.  With a shorthand
\begin{equation}
  \mathop{\mathtt{F}}(\mathsf{x})=\mathop{\mathtt{getmant}}(\mathsf{x},\mathtt{\_MM\_MANT\_NORM\_1\_2},\mathtt{\_MM\_MANT\_SIGN\_zero})
  \label{e:F}
\end{equation}
for extraction of the fractional parts of $\mathsf{x}$ in
$[1,2\rangle$, this formation procedure\footnote{slightly simplified
  and with a different instruction order than in the prototype
  implementation} is vectorized in \cref{a:Gram} for complex
  Grammians, while the real ones do not require $\Im{a_{21}''}$.

\begin{algorithm}[hbtp]
  \caption{Vectorized formation of non-overflowing scaled Grammians.}
  \label{a:Gram}
  \begin{algorithmic}[1]
    \REQUIRE{The vectors $\Re{\mathsf{a}_{21}'},\Im{\mathsf{a}_{21}'};\mathsf{e}_1,\mathsf{f}_1;\mathsf{e}_2,\mathsf{f}_2$ of the scaled dot-products and the first and the second column norms, resp., of the chosen pivot pairs $(p_1,q_1),\cdots,(p_{\mathtt{s}},q_{\mathtt{s}})$.}
    \ENSURE{$\mathsf{a}_{11}'',\mathsf{a}_{22}'';\Re{\mathsf{a}_{21}''},\Im{\mathsf{a}_{21}''}$, where the $(\ell-1)$-th lane of $\mathsf{a}_{ij}''$ is $(A_{[p_{\ell}^{}q_{\ell}^{}]}'')_{ij}^{}$.}
    \STATE{$\mathsf{f}_{12}=\mathop{\mathtt{div}}(\mathsf{f}_1,\mathsf{f}_2);\ \mathsf{e}_{12}=\mathop{\mathtt{sub}}(\mathsf{e}_1,\mathsf{e}_2);\ \ \ \mathsf{f}_{21}=\mathop{\mathtt{div}}(\mathsf{f}_2,\mathsf{f}_1);\ \mathsf{e}_{21}=\mathop{\mathtt{sub}}(\mathsf{e}_2,\mathsf{e}_1);$}
    \COMMENT{\cref{e:sGramEF}}
    \STATE{$\mathsf{e}_{12}=\mathop{\mathtt{add}}(\mathsf{e}_{12},\mathop{\mathtt{getexp}}(\mathsf{f}_{12}));\quad\mathsf{f}_{12}=\mathop{\mathtt{F}}(\mathsf{f}_{12});$}
    \COMMENT{$(e_{p/q},f_{p/q})$ from \cref{e:efpq} with \cref{e:F}}
    \STATE{$\mathsf{e}_{21}=\mathop{\mathtt{add}}(\mathsf{e}_{21},\mathop{\mathtt{getexp}}(\mathsf{f}_{21}));\quad\mathsf{f}_{21}=\mathop{\mathtt{F}}(\mathsf{f}_{21});$}
    \COMMENT{$(e_{q/p},f_{q/p})$ from \cref{e:efpq} with \cref{e:F}}
    \STATE{$\mathsf{s}_A=\mathop{\mathtt{min}}(\mathop{\mathtt{sub}}(\mathop{\mathtt{set1}}(\mathtt{DBL\_MAX\_EXP}-1),\mathop{\mathtt{max}}(\mathsf{e}_{12},\mathsf{e}_{21})),\mathop{\mathtt{setzero}}());$}
    \COMMENT{$s_{[p_{\ell}^{}q_{\ell}^{}]}^{}$}
    \STATE{$\mathsf{e}_{12}=\mathop{\mathtt{add}}(\mathsf{e}_{12},\mathsf{s}_A);\quad\mathsf{e}_{21}=\mathop{\mathtt{add}}(\mathsf{e}_{21},\mathsf{s}_A);$}
    \COMMENT{each lane $\le\hat{\eta}$}
    \STATE{$\mathsf{a}_{11}''=\mathop{\mathtt{scalef}}(\mathsf{f}_{12},\mathsf{e}_{12});\quad\mathsf{a}_{22}''=\mathop{\mathtt{scalef}}(\mathsf{f}_{21},\mathsf{e}_{21});$}
    \COMMENT{scale the diagonal}
    \STATE{$\Re{\mathsf{a}_{21}''}=\mathop{\mathtt{scalef}}(\Re{\mathsf{a}_{21}'},\mathsf{s}_A^{});\quad\Im{\mathsf{a}_{21}''}=\mathop{\mathtt{scalef}}(\Im{\mathsf{a}_{21}'},\mathsf{s}_A^{});$}
    \COMMENT{scale the off-diagonal}
  \end{algorithmic}
\end{algorithm}
\subsection{The Jacobi transformations}\label{ss:3.5}
In the $k$th step, the postmultiplication of a pivot column pair
$\begin{bmatrix}x_p^{(k)}\!\!&\!\!x_q^{(k)}\end{bmatrix}$, by the
Jacobi rotation
$\mathop{U}\left(\alpha_{pq}^{(k)},\varphi_{pq}^{(k)}\right)$ is
performed as in \cref{a:zjrot}, with $l=m$ for $X=G_k''$ and $l=n$ for
$X=V_k^{}$, where $V_0^{}=I$ if the right singular vectors are to be
accumulated in the final $V$; otherwise, $V_0^{}$ can be set to any
matrix that is to be multiplied by them.  In both cases, $V_1^{}$ can
optionally be scaled similarly as $G_1^{}$ would be if there were no
concern for overflowing of its column norms, but then the protection
from overflow in the course of subsequent transformations of $V_k^{}$
has to be maintained.  If $V_1^{}=I$, the $\max$-norm approximation in
\cref{a:zjrot} is not needed for the columns of $V_k^{}$, and a
slightly faster routine, $\mathtt{zjrotf}$, is sufficient.

\begin{algorithm}[hbtp]
  \caption{$\mathtt{zjrot}$: a vectorized postmultiplication of a pair
    of complex columns $\begin{bmatrix}x_p\!\!&\!\!x_q\end{bmatrix}$
    of length $l$, in the split form, by the Jacobi transformation
    $\Phi=U(\alpha_{pq},\varphi_{pq})P$.}
  \label{a:zjrot}
  \begin{algorithmic}[1]
    \REQUIRE{$(\Re{x_p},\Im{x_p}),(\Re{x_q},\Im{x_q});c=\cos\varphi_{pq};C=\cos\alpha_{pq}\tan\varphi_{pq},S=\sin\alpha_{pq}\tan\varphi_{pq}$.}
    \ENSURE{$\begin{bmatrix}x_p'\!\!&\!\!x_q'\end{bmatrix}=\begin{bmatrix}x_p^{}\!\!&\!\!x_q^{}\end{bmatrix}\Phi$ and, optionally, an approximation of its $\max$-norm.}
    \STATE{$\mathsf{C}=\mathop{\mathtt{set1}}(C);\ \ \mskip-2mu-\mathsf{C}=\mathop{\mathtt{set1}}(-C);\ \ \mskip-1mu\mathsf{S}=\mathop{\mathtt{set1}}(S);\ \ \mskip-1mu\mathsf{c}=\mathop{\mathtt{set1}}(c);\ \ \mskip-1mu-\mathsf{0}=\mathop{\mathtt{set1}}(-0.0);$}
    \STATE{$\mathsf{M}=\mathop{\mathtt{setzero}}();$}
    \COMMENT{a vector of partial $\max$-norm approximations}
    \FOR[$l$ optimal in the sense of~\cref{e:pad0}]{$\mathtt{i}=0$ \TO $l-1$ \textbf{step} $\mathtt{s}$}
    \STATE{$\Re{\mathsf{x}_{\mathtt{i}j}}=\mathop{\mathtt{load}}(\Re{x_j}+\mathtt{i});\quad\Im{\mathsf{x}_{\mathtt{i}j}}=\mathop{\mathtt{load}}(\Im{x_j}+\mathtt{i});$}
    \COMMENT{$j\in\{p,q\}$\\}
    \COMMENT{\cref{e:zfma,e:transf}\hfill}
    \STATE{$\Re{\mathsf{x}_{\mathtt{i}p}'}=\mathop{\mathtt{mul}}(\mathop{\mathtt{fmadd}}(\Re{\mathsf{x}_{\mathtt{i}q}^{}},\mathsf{C},\mathop{\mathtt{fnmadd}}(\Im{\mathsf{x}_{\mathtt{i}q}^{}},\mathsf{S},\Re{\mathsf{x}_{\mathtt{i}p}^{}})),\mathsf{c});$}
    \COMMENT{$\mathop{\mathrm{fma}}(x_q,\mathrm{e}',x_p)c$}
    \STATE{$\Im{\mathsf{x}_{\mathtt{i}p}'}=\mathop{\mathtt{mul}}(\mathop{\mathtt{fmadd}}(\Re{\mathsf{x}_{\mathtt{i}q}^{}},\mathsf{S},\mathop{\mathtt{fmadd}}(\Im{\mathsf{x}_{\mathtt{i}q}^{}},\mathsf{C},\Im{\mathsf{x}_{\mathtt{i}p}^{}})),\mathsf{c});$}
    \COMMENT{$\mathrm{e}'=C+\mathrm{i}S$}
    \STATE{$\mathop{\mathtt{store}}(\Re{x_j^{}}+\mathtt{i},\Re{\mathsf{x}_{\mathtt{i}p}'});\quad\mathop{\mathtt{store}}(\Im{x_j^{}}+\mathtt{i},\Im{\mathsf{x}_{\mathtt{i}p}'});$}
    \COMMENT{$j=p$ if $P=I$, else $j=q$}
    \STATE{$\mathsf{M}=\mathop{\mathtt{max}}(\mathsf{M},\mathop{\mathtt{max}}(\mathop{\mathtt{andnot}}(-\mathsf{0},\Re{\mathsf{x}_{\mathtt{i}p}'}),\mathop{\mathtt{andnot}}(-\mathsf{0},\Im{\mathsf{x}_{\mathtt{i}p}'})));$}
    \COMMENT{update $\mathsf{M}$}
    \STATE{$\Re{\mathsf{x}_{\mathtt{i}q}'}=\mathop{\mathtt{mul}}(\mathop{\mathtt{fmadd}}(\Re{\mathsf{x}_{\mathtt{i}p}^{}},-\mathsf{C},\mathop{\mathtt{fnmadd}}(\Im{\mathsf{x}_{\mathtt{i}p}^{}},\mathsf{S},\Re{\mathsf{x}_{\mathtt{i}q}^{}})),\mathsf{c});$}
    \COMMENT{$\mathop{\mathrm{fma}}(x_p,-\bar{\mathrm{e}}',x_q)c$}
    \STATE{$\Im{\mathsf{x}_{\mathtt{i}q}'}=\mathop{\mathtt{mul}}(\mathop{\mathtt{fmadd}}(\Re{\mathsf{x}_{\mathtt{i}p}^{}},\mathsf{S},\mathop{\mathtt{fmadd}}(\Im{\mathsf{x}_{\mathtt{i}p}^{}},-\mathsf{C},\Im{\mathsf{x}_{\mathtt{i}q}^{}})),\mathsf{c});$}
    \COMMENT{$-\bar{\mathrm{e}}'=-C+\mathrm{i}S$}
    \STATE{$\mathop{\mathtt{store}}(\Re{x_j^{}}+\mathtt{i},\Re{\mathsf{x}_{\mathtt{i}q}'});\quad\mathop{\mathtt{store}}(\Im{x_j^{}}+\mathtt{i},\Im{\mathsf{x}_{\mathtt{i}q}'});$}
    \COMMENT{$j=q$ if $P=I$, else $j=p$}
    \STATE{$\mathsf{M}=\mathop{\mathtt{max}}(\mathsf{M},\mathop{\mathtt{max}}(\mathop{\mathtt{andnot}}(-\mathsf{0},\Re{\mathsf{x}_{\mathtt{i}q}'}),\mathop{\mathtt{andnot}}(-\mathsf{0},\Im{\mathsf{x}_{\mathtt{i}q}'})));$}
    \COMMENT{update $\mathsf{M}$}
    \ENDFOR\COMMENT{after each iteration: $\mathsf{M}=\max\{\mathsf{M},|\Re{\mathsf{x}_{\mathtt{i}p}'}|,|\Im{\mathsf{x}_{\mathtt{i}p}'}|,|\Re{\mathsf{x}_{\mathtt{i}q}'}|,|\Im{\mathsf{x}_{\mathtt{i}q}'}|\}$}
    \RETURN{$\mathop{\mathtt{reduce\_max}}(\mathsf{M});$}
    \COMMENT{maximum of the lanes of $\mathsf{M}$}
  \end{algorithmic}
\end{algorithm}

In \cref{a:zjrot} $P$ is a permutation matrix that is not identity if
the transformed columns have to be swapped.  In the (unconditionally
reproducible) implementation there are further optimizations, like
skipping the multiplications by $c=1$ and applying two real
transformations, on $(\Re{x_p},\Re{x_q})$ and $(\Im{x_p},\Im{x_q})$,
if the rotation is real ($S=0$).  No conditionals are present in the
loop; instead, each branch has a specialized version of the loop.  The
return value is
$\max\{\|\Re{x_p'}\|_{\max}^{},\|\Im{x_p'}\|_{\max}^{},\|\Re{x_q'}\|_{\max}^{},\|\Im{x_q'}\|_{\max}^{}\}$.
\subsubsection{The Gram--Schmidt orthogonalization}\label{sss:3.5.1}
\looseness=-1
In~\cite[Definition~2.7]{Drmac-97} the conditions and the formula for
the Gram--Schmidt orthogonalization of $g_q$ against $g_p$ are given,
that replaces their Jacobi transformation when
$\|g_q\|_F\ll\|g_p\|_F$ and $\tan\varphi$ underflows.
Using~\cref{e:sGramEF}, the orthogonalization
from~\cite[Eq.~(2.20)]{Drmac-97} is defined here as
\begin{equation}
  \|g_p^{}\|_F^{}>2^{\hat{\eta}}\upsilon\|g_q^{}\|_F^{}\implies g_p'=g_p^{}\,\wedge\,
  g_q'=\left(\frac{g_q^{}}{\|g_q^{}\|_F^{}}-a_{21}^{\prime\ast}\frac{g_p^{}}{\|g_p^{}\|_F^{}}\right)\|g_q^{}\|_F^{},
  \label{e:GScond}
\end{equation}
what, by representing $\|g_j\|_F=2^{e_j}f_j$ and after moving $f_q$
within the parenthesis, gives
\begin{equation}
  g_q'=\left(\frac{g_q^{}}{2^{e_q}}-\psi\frac{g_p^{}}{2^{e_p}}\right)2^{e_q},\qquad
  \psi=a_{21}^{\prime\ast}(f_q^{}/f_p^{}),\qquad
  1/2<f_q^{}/f_p^{}<2.
  \label{e:GS}
\end{equation}
If the roles of $g_p^{}$ and $g_q^{}$ are reversed in~\cref{e:GS},
i.e., if $\|g_q^{}\|_F^{}>2^{\hat{\eta}}\upsilon\|g_p^{}\|_F^{}$,
$a_{21}'$ should be used in~\cref{e:GS} instead of
$a_{21}^{\prime\ast}$, and the resulting $g_p'$ and $g_q'=g_q^{}$
should be swapped in place to keep the column norms sorted
non-increasingly.  \Cref{a:ZGS,a:DGS} vectorize~\cref{e:GS} but have
not been extensively tested.
\section{A parallel Jacobi-type SVD method}\label{s:4}
In this section the previously developed building blocks are put
together to form a robust, OpenMP-parallel Jacobi-type SVD method.
It is applicable to any $\tilde{m}\times\tilde{n}$ input matrix $G$ of
full column rank with finite elements, where $\tilde{m}$
satisfies~\cref{e:pad0} and $\tilde{n}\bmod 2\mathtt{s}=0$.  If the
dimensions of $G$ do not satisfy these constraints, $G$ is assumed to
be bordered beforehand, as explained in~\cite{Novakovic-Singer-11},
e.g.  The workspace required is $\tilde{n}/\mathtt{s}$ integers and
$5\tilde{n}$ or $7\tilde{n}$ reals for the real or the complex
variant, respectively.  Any number of OpenMP threads can be requested,
but at most $\tilde{n}$ will be used at any time.  The single
precision variants of the method, real and complex, have also been
implemented and tested, as described in \cref{s:SM8}.  The method can
be adapted for distributed memory, but it would be inefficient without
blocking (see, e.g.,~\cite{Novakovic-Singer-21} for a conceptual
overview).
\subsection{Parallel Jacobi strategies}\label{ss:4.1}
Even a sequential but vectorized method processes $\mathtt{s}>1$ pivot
column pairs in each step.  No sequential pivot strategy (such as
de~Rijk's~\cite{deRijk-89}, $\mathtt{dR}$, in LAPACK), that selects a
single pivot pair at a time, suffices, and a parallel one has to be
chosen.  It is then natural to select the maximal number of
$\tilde{n}/2$ pivot pairs each time, where all pivot indices are
different.  Since the number of all index pairs $(p,q)$ where $p<q$ is
$\tilde{n}(\tilde{n}-1)/2$, the chosen strategy is expected to require
at least $\mathtt{K}=\tilde{n}-1$ steps in a sweep.  It may require
more (e.g., $\mathtt{K}=\tilde{n}$), and transform a subset of pivot
pairs more than once in a sweep.  An example is a quasi-cyclic
strategy called the modified modulus~\cite{Novakovic-Singer-11}
(henceforth, $\mathtt{MM}$), applicable for $\tilde{n}$ even.  A
cyclic (i.e., repeating the same pivot sequence in each sweep)
strategy, with $\tilde{n}-1$ steps, could be, e.g., the
Mantharam--Eberlein~\cite{Mantharam-Eberlein-93} one, or its
generalization $\mathtt{ME}$ beyond $\tilde{n}$ being a power of two,
from~\cite{Novakovic-15}.  Available in theory for all even
$\tilde{n}$, $\mathtt{ME}$ is restricted in practice to
$\tilde{n}=2^lo$, $l\ge 1$, $o\le 21$ odd, with a noticeably faster
convergence than
$\mathtt{MM}$~\cite{Novakovic-15,Novakovic-Singer-21}.

The method is executed on a shared-memory system, so the cost of
``communication'' could be visible only if the data spans more than
one NUMA domain; otherwise, any communication topology underlying a
strategy can be disregarded when looking for a suitable one.  More
important is to assess if a strategy is convergent (provably, as
$\mathtt{MM}$, or at least in practice, as $\mathtt{ME}$) and the cost
of its implementation (a lookup table of at most $\tilde{n}^2$
integers encoding the pivot pairs in each step of a sweep for
$\mathtt{MM}$ and $\mathtt{ME}$ is set up before the execution for a
given $\tilde{n}$ in $\mathop{\mathcal{O}}(\tilde{n}^2)$ and
$\mathop{\mathcal{O}}(\tilde{n}^2\lg\tilde{n})$ time, respectively).

A dynamic
ordering~\cite{Becka-Oksa-Vajtersic-02,Becka-Oksa-Vajtersic-15} would
be a viable alternative to cyclic parallel strategies, but it is
expensive for pointwise (i.e., non-blocked) one-sided methods (for a
two-sided, pointwise Kogbetliantz-type SVD method with a dynamic
ordering, see~\cite{Novakovic-Singer-22}).

Among other advantages of the sequential Jacobi-type methods, the
quasi-cubic convergence speedup of Mascarenhas~\cite{Mascarenhas-95}
remains elusive with a parallel strategy.  A pointwise one-sided
method is the most ungrateful Jacobi-type SVD for parallelization,
with no performance benefits of blocking but with all the issues such
a constrained choice of parallel strategies brings, as the slow
convergence and a probably excessive amount of slightly non-orthogonal
transformations in \cref{ss:5.3} show.
\subsection{Data representation}\label{ss:4.2}
Complex arrays are kept in the split form.  Splitting the input matrix
$G$ and merging the output matrices $U$ (occupying the space of $G$)
and $V$ happen before and after the method is invoked, respectively,
and take less than $1\permil$ of the method's run-time.  The resulting
singular values are kept as two properly aligned arrays, $\mathtt{e}$
and $\mathtt{f}$, such that $\sigma_j=(\mathtt{e}_j,\mathtt{f}_j)$.
The integer work arrays are $\mathtt{p}$ and $\mathtt{c}$, each with
$\tilde{n}/(2\mathtt{s})$ elements, while the real workspace
$\mathtt{w}$ is divided into several properly aligned subarrays that
are denoted by a tilde over their names in \cref{a:zvjsvd}.
\subsubsection{Column norms and the singular vectors}\label{sss:4.2.1}
\looseness=-1
The Frobenius norms of the columns of the iteration matrix are held in
$(\mathtt{e},\mathtt{f})$.  If two columns of $G_k''$ (and the same
ones of $V_k^{}$) are swapped, so are their norms.  If a column is
transformed, its norm will be recomputed (not updated, as
in~\cite{Drmac-97}) at the beginning of the $(k+1)$-th step.

If the method converges, the iteration matrix, holding $U\Sigma'$, has
to be normalized to $U$.  For all $j$ in parallel, each component of
every element of the $j$th column of the iteration matrix is scaled by
$2^{-\mathtt{e}_j}$ and divided by $\mathtt{f}_j$. Finally,
$\mathtt{e}_j=\mathtt{e}_j-s$.  If $V_k$ has been scaled by a power of
two, the final $V_k$ has to be backscaled to $V$ in a similar way.
\subsection{The method}\label{ss:4.3}
\Cref{a:zvjsvd} shows a simplified implementation of the complex
double precision method.  The real variant, $\mathtt{dvjsvd}$, is
derived straightforwardly.

If the assumption from \cref{ss:3.2} on a safe upper bound of the
magnitudes of the columns' elements for the Frobenius norm computation
is adequate, line~\ref{bl:11} in \cref{a:zvjsvd} cannot cause an
infinite loop, but an inadequate assumption can.  In the testing from
\cref{s:5}, after the initial scaling of $G_0\to G_1$, no rescaling of
the iteration matrix was ever triggered, except of $U\Sigma'\to U$, so
it should be a rare event.

\looseness=-1
Even an inadequate assumption can be incrementally improved.  If a
norm overflow is detected in line~\ref{bl:11} more than once in
succession, the assumed upper bound can be divided by two each time,
until the iteration matrix is downscaled enough to prevent overflow
and break this \texttt{goto}-loop.  \Cref{s:SM3} gives the
\emph{lower} bounds on the assumption that would eventually be
reached, when the column norms could be computed by ordinary
dot-products as $\sqrt{\mathbf{x}^{\ast}\mathbf{x}}$.  If this
happens, the flawed norm-computing routine can be replaced by a
wrapper around the dot-product by a function pointer swap, without
stopping the execution.  This safeguard has not yet been implemented.

All \emph{innermost} loops of \cref{a:zvjsvd} are (but do not have to
be) \emph{parallel} and, at least in the first few sweeps over a
general matrix, have a balanced workload across all threads (i.e.,
most bitmasks $\mathfrak{c}$ in line~\ref{bl:19} are all-ones or close
to that).  Near the end of the execution, in the last sweeps, the
number of pivot pairs that have to be processed should diminish,
depending on the asymptotic convergence rate of the pivot strategy.

\begin{algorithm}[hbtp]
  \caption{$\mathtt{zvjsvd}$: a vectorized, OpenMP-parallel Jacobi-type SVD method.}
  \label{a:zvjsvd}
  \begin{algorithmic}[1]
    \REQUIRE{$(\Re{G},\Im{G});\mathtt{S},\mathtt{K};\mathtt{w},\mathtt{p},\mathtt{c}$; a strategy lookup table $\mathop{\mathtt{J}}((k-1)\!\bmod\!\mathtt{K},\ell)\to(p_{\ell}^{(k)},q_{\ell}^{(k)})$.}
    \ENSURE{$(\Re{U},\Im{U}),(\Re{V},\Im{V}),(\mathtt{e},\mathtt{f})$; the actual number of sweeps performed $\mathtt{C}\le\mathtt{S}$.}
    \STATE{determine $\widetilde{M}_0^{\mathbb{C}},s_0^{\mathbb{C}}$ from $G_0^{}$ and scale $G_0^{},\widetilde{M}_0^{\mathbb{C}}$ to $G_1^{},\widetilde{M}_1^{\mathbb{C}}$ as in \cref{ss:3.2};}
    \STATE{set $V_1=I$ in parallel over its columns and let $k=1$;}
    \COMMENT{$k$ is the step counter}
    \FOR[sweep loop]{$\mathtt{C}=0$ \TO $\mathtt{S}-1$}
    \STATE{$T=0;$}
    \COMMENT{$T$ counts the number of transformations in the whole sweep}
    \FOR[step loop, $k'=(k-1)\bmod\mathtt{K}$]{$k'=0$ \TO $\mathtt{K}-1$}
    \STATE{find $s_{(k)}^{\mathbb{C}},s_{[k]}^{\mathbb{C}}$ from $\widetilde{M}_k^{\mathbb{C}}$, \cref{e:skn,e:skr}, and downscale $G_k^{}\to G_k'$ if $s_{(k)}^{\mathbb{C}}<0$;}
    \FOR[$+$-reduce $\mathtt{o}$ (initially $0$); $(p_{\ell},q_{\ell})=\mathop{\mathtt{J}}(k',\ell)$]{$\ell=1$ \TO $\tilde{n}/2$ \textbf{parallel}\label{bl:7}}
    \STATE{$\|g_{\jmath_{\ell}}\|_F=\mathop{\mathrm{hypot}}(\|\Re{g_{\jmath_{\ell}}}\|_F,\|\Im{g_{\jmath_{\ell}}}\|_F);$}
    \COMMENT{compute $\|g_{\jmath_{\ell}}\|_F$ for $\jmath\in\{p,q\}$}
    \STATE{\textbf{if} $\|g_{\jmath_{\ell}}\|_F=\infty$ \textbf{then} $\mathtt{o}=\mathtt{o}+1$ \textbf{else} represent $\|g_{\jmath_{\ell}}\|_F$ as $(\mathtt{e}_{\jmath_{\ell}},\mathtt{f}_{\jmath_{\ell}});$}
    \ENDFOR\COMMENT{$\mathtt{o}$ holds the number of overflowed column norms\label{bl:10}}
    \STATE{\textbf{if} $\mathtt{o}>0$ \textbf{then} downscale $G_k'\to G_k''$ as in \cref{ss:3.2} and \textbf{goto} line~\ref{bl:7};\label{bl:11}}
    \FOR[the scaled dot-products as in \cref{ss:3.3}]{$\ell=1$ \TO $\tilde{n}/2$ \textbf{parallel}\label{bl:12}}
    \STATE{$(a_{21}')_{\ell}^{}\mskip-1mu=\mskip-1mu\check{g}_{q_{\ell}}^{\ast}\check{g}_{p_{\ell}}^{}\mskip-3mu=\mskip-1mu\mathop{\mathtt{zdpscl}^{[\prime]}}(\Re{g_{q_{\ell}}^{}},\Im{g_{q_{\ell}}^{}},\Re{g_{p_{\ell}}^{}},\Im{g_{p_{\ell}}^{}},\mathtt{e}_{q_{\ell}}^{},\mathtt{e}_{p_{\ell}}^{},\mathtt{f}_{q_{\ell}}^{},\mathtt{f}_{p_{\ell}}^{});$\label{bl:13}}
    \COMMENT{\cref{e:sGramEF}}
    \STATE{$\Re{\tilde{\mathtt{a}}_{21}^{}}[\ell-1]=(\Re{a_{21}'})_{\ell}^{},\ \Im{\tilde{\mathtt{a}}_{21}^{}}[\ell-1]=(\Im{a_{21}'})_{\ell}^{};$}
    \COMMENT{pack contiguously in $\mathtt{w}$\label{bl:14}}
    \STATE{$\tilde{\mathtt{e}}_1[\ell-1]=\mathtt{e}_{p_{\ell}},\ \tilde{\mathtt{f}}_1[\ell-1]=\mathtt{f}_{p_{\ell}},\quad\tilde{\mathtt{e}}_2[\ell-1]=\mathtt{e}_{q_{\ell}},\ \tilde{\mathtt{f}}_2[\ell-1]=\mathtt{f}_{q_{\ell}};$\label{bl:15}}
    \ENDFOR\COMMENT{use either \cref{a:zdpsclcs} or \cref{a:zdpscl} in line~\ref{bl:13}}
    \FOR[$+$-reduce $t$ (initially $0$)\\\hfill/\hskip-2pt/ check the convergence criterion, assemble $\mathtt{s}$ scaled Grammians $(A_{[p_{\ell}^{},q_{\ell}^{}]}'')_{\ell}^{}$]{$\mathtt{i}=0$ \TO $\tilde{n}/2-1$ \textbf{step} $\mathtt{s}$ \textbf{parallel}}
    \STATE{load $\Re{\mathsf{a}_{21}'}\mskip-1mu,\Im{\mathsf{a}_{21}'}\mskip-1mu,\mathsf{e}_l^{}\mskip-1mu,\mathsf{f}_l^{}$ from $\Re{\tilde{\mathtt{a}}_{21}^{}}\mskip-1mu+\mskip-1mu\mathtt{i}\mskip-1mu,\Im{\tilde{\mathtt{a}}_{21}^{}}\mskip-1mu+\mskip-1mu\mathtt{i}\mskip-1mu,\tilde{\mathtt{e}}_l^{}\mskip-1mu+\mskip-1mu\mathtt{i}\mskip-1mu,\tilde{\mathtt{f}}_l^{}\mskip-1mu+\mskip-1mu\mathtt{i}$, resp.;}
    \COMMENT{$l\mskip-3mu\in\mskip-3mu\{\mskip-1mu 1,\mskip-2mu 2\}$}
    \STATE{$(\mathfrak{c}\to\mathtt{c}_j^{},\sum_{\imath}\mathfrak{c}_{\imath}^{}\to\mathtt{p}_j^{})=\mathop{\text{\cref{a:cvg}}}(\Re{\mathsf{a}_{21}'},\Im{\mathsf{a}_{21}'});$}
    \COMMENT{$j=\mathtt{i}/\mathtt{s}$\label{bl:19}}
    \STATE{\textbf{if} $\mathtt{p}_j>0$ \textbf{then} $t=t+\mathtt{p}_j$ \textbf{else} skip the lines \ref{bl:21}, \ref{bl:22}, and \ref{bl:23};}
    \STATE{check~\cref{e:GScond} if \cref{a:ZGS} has to be used in line~\ref{bl:31} instead of $\mathtt{zjrot}$;\label{bl:21}}
    \STATE{$(\mathsf{a}_{11}'',\mathsf{a}_{22}'',\Re{\mathsf{a}_{21}''},\Im{\mathsf{a}_{21}''})\!=\!\mathop{\text{\cref{a:Gram}}}(\Re{\mathsf{a}_{21}'},\Im{\mathsf{a}_{21}'},\mathsf{e}_1^{},\mathsf{f}_1^{},\mathsf{e}_2^{},\mathsf{f}_2^{});$}
    \COMMENT{$(A_{[p_{\ell}^{},q_{\ell}^{}]}'')_{\ell}^{}$\label{bl:22}}
    \STATE{store $\mathsf{a}_{11}'',\mathsf{a}_{22}'',\Re{\mathsf{a}_{21}''},\Im{\mathsf{a}_{21}''}$ to $\tilde{\mathbf{a}}_{11}^{}+\mathtt{i},\tilde{\mathbf{a}}_{22}^{}+\mathtt{i},\Re{\tilde{\mathbf{a}}_{21}^{}}+\mathtt{i},\Im{\tilde{\mathbf{a}}_{21}^{}}+\mathtt{i}$, resp.;\label{bl:23}}
    \ENDFOR\COMMENT{$t$ holds the number of transformations in the current step\label{bl:24}\\}
    \COMMENT{\Cref{a:zbjac2} computes the EVDs of the scaled Grammians in parallel:\hfill}
    \STATE{$\mathop{\mathtt{zbjac2}}(\tilde{\mathbf{a}}_{11},\tilde{\mathbf{a}}_{22},\Re{\tilde{\mathbf{a}}_{21}},\Im{\tilde{\mathbf{a}}_{21}};\cos\tilde{\bm{\varphi}},\cos\tilde{\bm{\alpha}}\tan\tilde{\bm{\varphi}},\sin\tilde{\bm{\alpha}}\tan\tilde{\bm{\varphi}},\tilde{\bm{\lambda}}_1,\tilde{\bm{\lambda}}_2;\mathtt{p});$\label{bl:25}}
    \FOR[$\max$-reduce $M$ (initially $0$); $j=(\ell-1)/\mathtt{s}$]{$\ell=1$ \TO $\tilde{n}/2$ \textbf{parallel}\label{bl:26}}
    \IF[bit index $\imath=(\ell-1)\bmod\mathtt{s}$]{$(\mathtt{c}_j)_{\imath}=0$}
    \STATE{\textbf{if} $(\mathtt{p}_j)_{\imath}=1$ \textbf{then} swap $g_{p_{\ell}}$ and $g_{q_{\ell}}$, $v_{p_{\ell}}$ and $v_{q_{\ell}}$, $(\mathtt{e}_{p_{\ell}},\mathtt{f}_{p_{\ell}})$ and $(\mathtt{e}_{q_{\ell}},\mathtt{f}_{q_{\ell}})$;}
    \ELSE[postmultiply the column pairs by $U_{\ell}P_{\ell}$ computed in line~\ref{bl:25}]
    \STATE{\textbf{if} $(\mathtt{p}_j)_{\imath}=1$ \textbf{then} $P_{\ell}=\left[\begin{smallmatrix}0&1\\1&0\end{smallmatrix}\right]$ \textbf{else} $P_{\ell}=I$;}
    \COMMENT{decode $P_{\ell}$ from $\mathtt{p}$\\}
    \COMMENT{\Cref{a:zjrot} applied to $(g_{p_{\ell}},g_{q_{\ell}})$, $\tilde{m}$ rows, and to $(v_{p_{\ell}},v_{q_{\ell}})$, $\tilde{n}$ rows:\hfill}
    \STATE{$M'\!=\!\mathop{\mathtt{zjrot}}(\Re{g_{p_{\ell}}},\mskip-2mu\Im{g_{p_{\ell}}},\mskip-2mu\Re{g_{q_{\ell}}},\mskip-2mu\Im{g_{q_{\ell}}},\mskip-2mu\cos\tilde{\bm{\varphi}}_{\ell},\mskip-2mu\cos\tilde{\bm{\alpha}}_{\ell}\tan\tilde{\bm{\varphi}}_{\ell},\mskip-2mu\sin\tilde{\bm{\alpha}}_{\ell}\tan\tilde{\bm{\varphi}}_{\ell},\mskip-2muP_{\ell})\mskip-1mu;$\label{bl:31}}
    \STATE{$\mathop{\mathtt{zjrotf}}(\Re{v_{p_{\ell}}},\Im{v_{p_{\ell}}},\Re{v_{q_{\ell}}},\Im{v_{q_{\ell}}},\cos\tilde{\bm{\varphi}}_{\ell},\cos\tilde{\bm{\alpha}}_{\ell}\tan\tilde{\bm{\varphi}}_{\ell},\sin\tilde{\bm{\alpha}}_{\ell}\tan\tilde{\bm{\varphi}}_{\ell},P_{\ell})\mskip-1mu;$}
    \STATE{$M=\max\{M,M'\};$}
    \COMMENT{$0<M'<\infty$ should hold}
    \ENDIF
    \COMMENT{$M$ unchanged if the columns are not transformed}
    \ENDFOR\COMMENT{see \cref{ss:3.5}}
    \STATE{$T=T+t;\quad\widetilde{M}_{k+1}^{\mathbb{C}}=\max\{\widetilde{M}_k^{\mathbb{C}},M\};$}
    \COMMENT{$\widetilde{M}_{k+1}^{\mathbb{C}}$ might be a safe overestimate\label{bl:36}}
    \ENDFOR\COMMENT{$k=k+1$}
    \STATE{\textbf{if} $T=0$ \textbf{then break};}
    \COMMENT{convergence if no transformations in a sweep}
    \ENDFOR\COMMENT{$\mathtt{C}=\mathtt{C}+1$}
    \STATE{\textbf{if} $\mathtt{C}<\mathtt{S}$ \textbf{then} normalize $U\Sigma'\to U$ and scale $\Sigma'\to\Sigma$ as in \cref{sss:4.2.1};}
    \RETURN{$\mathtt{C};$}
    \COMMENT{$\mathtt{C}=\mathtt{S}\iff\text{convergence not detected}$\\}
    \COMMENT{\hskip-1pt For referencing, let lines $\mskip-1mu\ref{bl:7}$--$\ref{bl:10}\!=\!\spadesuit$, $\mskip-1mu\ref{bl:12}$--$\ref{bl:14}\!=\!\clubsuit$, $\mskip-1mu\ref{bl:15}$--$\ref{bl:24}\!=\!\heartsuit$, $\mskip-1mu\ref{bl:25}\!=\!\diamondsuit$, and $\mskip-1mu\ref{bl:26}$--$\ref{bl:36}\!=\!\bigstar\!$.}
  \end{algorithmic}
\end{algorithm}
\subsubsection{Reproducibility}\label{sss:4.3.1}
Apart from the Frobenius norm computation and
\cref{a:zdpscl,a:zdpsclcs}, the results of which are reproducible in
the same environment, all other parts of the method are
unconditionally reproducible.  The method's results by design do
\emph{not} depend on the requested number of threads, as long as the
external routines (only \texttt{xNRM2}) are sequential, but do on the
choice of parallel strategy.
\section{Numerical testing}\label{s:5}
The testing was performed on the Intel DevCloud for oneAPI cluster
with two Intel Xeon Platinum 8358 CPUs per node, each with 32 cores
nominally clocked at $2.60\,\mathrm{GHz}$ but running at variable
frequencies due to TurboBoost.  Under 64-bit Linux, the Intel oneAPI C
(\texttt{icc}), C++ (\texttt{icpc}), and Fortran (\texttt{ifort})
compilers, versions 2021.6.0 (for \cref{ss:5.2,ss:SM6.4}) and
2021.7.1, and the sequential MKL libraries 2022.0.1 and 2022.0.2,
respectively, were used, with the ILP64 ABI and the Conditional
Numerical Reproducibility mode set to \texttt{MKL\_CBWR\_AVX512\_E1}.

The CPU's per-core cache sizes are: $48\,\mathrm{kB}$ for level~1
(data), $1280\,\mathrm{kB}$ for level~2, and $1536\,\mathrm{kB}$ for
level~3 (assuming that the $48\,\mathrm{MB}$ in total of the
last-level cache is equally distributed among the cores).  The OpenMP
environment was set up for all tests as
$\text{\texttt{OMP\_PROC\_BIND}}=\text{\texttt{SPREAD}}$,
$\text{\texttt{OMP\_PLACES}}=\text{\texttt{CORES}}$,
$\text{\texttt{OMP\_DYNAMIC}}=\text{\texttt{FALSE}}$, and
$\text{\texttt{KMP\_DETERMINISTIC\_REDUCTION}}=\text{\texttt{TRUE}}$,
on an exclusive-use node.  For the EVD testing $32$ threads were used,
while for the SVD testing
$\text{\texttt{OMP\_NUM\_THREADS}}\in\{\text{\texttt{16}},\text{\texttt{32}},\text{\texttt{64}}\}$.

The batched EVD \cref{a:zbjac2} (henceforth, \texttt{z}),
\cref{a:dbjac2} (\texttt{d}), and their single precision complex
(\texttt{c}) and real (\texttt{s}) counterparts were tested also in
isolation, on (huge, for more reliable timing) batches of
Hermitian/symmetric matrices of order two, comparing them to the
inlineable, manually translated C versions of the reference LAPACK
routines \texttt{xLAEV2}, for
$\text{\texttt{x}}\in\{\text{\texttt{Z}},\text{\texttt{D}},\text{\texttt{C}},\text{\texttt{S}}\}$,
respectively, since the MKL's and the reference implementations were
slower to call, with no observed difference in accuracy.

The SVD method in \cref{a:zvjsvd} and its real variant were compared
to the \texttt{ZGESVJ} and \texttt{DGESVJ} routines, respectively,
with
$\text{\texttt{JOBA}}=\text{`\texttt{G}'}$,
$\text{\texttt{JOBU}}=\text{`\texttt{U}'}$, and
$\text{\texttt{JOBV}}=\text{`\texttt{V}'}$.

All error testing was done in quadruple precision datatypes,
\texttt{\_\_float128} in C and \texttt{REAL(KIND=REAL128)} in Fortran,
including the final scaling of the singular values
\texttt{SVA(j)*WORK(1)} from \texttt{xGESVJ} and
$\sigma_j=2^{\mathtt{e}_j}\mathtt{f}_j$ from the proposed SVD method.
\subsection{Matrices under test}\label{ss:5.1}
\looseness=-1
For the EVDs, $256$ batches in single and $256$ batches in double
precision, each with $2^{28}$ Hermitian and $2^{28}$ symmetric
matrices of order two, were generated using~\cref{e:2} from random
$\lambda_1$, $\lambda_2$, $\tan\varphi$, and $\cos\alpha$, where the
random bits were provided by the \texttt{RDRAND} CPU facility.  For
the eigenvalues $\lambda_j$, a 32-bit or 64-bit quantity was
reinterpreted as a single or a double precision value, respectively,
and accepted if $|\lambda_j|\le\nu/2^4$.  Random 64-bit signed
integers were converted to quadruple precision, scaled by $2^{-63}$ to
the $[-1,1\rangle$ range, and assigned to $\tan\varphi$ and
$\cos\alpha$.  Then, $|\sin\alpha|=\sqrt{1-\cos^2\alpha}$ in quadruple
precision, and the sign of $\tan\varphi$ was absorbed into
$\mathrm{e}^{\mathrm{i}\alpha}$ (making
$\mathrm{e}^{\mathrm{i}\alpha}=\pm 1$ in the real case).  The three
required matrix elements from the lower triangle were computed in
quadruple precision and rounded to single or double precision without
overflow.  Five real values were generated in total for $\lambda_1$,
$\lambda_2$, $\tan\varphi$, and $\cos\alpha$: four of them for one
complex and two real matrix elements, and the last one for the
$(2,1)$-element of a real symmetric matrix, implicitly generated from
the same eigenvalues and $\tan\varphi$, but as if $\cos\alpha=1$
initially.  These values were stored to binary files and later read
from them into memory, one batch at a time, in the layout described in
\cref{sss:2.4.1}.  The eigenvalues were similarly preserved for
comparison.

To make $|\mathrm{e}^{\mathrm{i}\alpha}|$ as close to unity as
practicable, $\cos\alpha$ was in fact rounded from quadruple to
double precision (with $52$ bits of significand) and converted back,
before computing $\sin^2\alpha=1-\cos^2\alpha$ with $112$ bits of
significand.  Thus, $\sin^2\alpha$ was exact.

For the SVD testing, the datasets $\Xi_1^{\mathbb{F}}$ and
$\Xi_2^{\mathbb{F}}$, parametrized by $\xi=\xi_1^{}=-23$ and
$\xi=\xi_2^{}=-52$, respectively, were generated\footnote{See
\url{https://github.com/venovako/JACSD/tree/master/tgensvd} for the
implementation.}, each one with complex ($\mathbb{F}=\mathbb{C}$) and
real ($\mathbb{F}=\mathbb{R}$) square double precision matrices, from
the given singular values $\Sigma[\xi,n,P_n]$ (same for both
$\mathbb{F}$).  In $\Xi_1^{\mathbb{F}}$, $n=\tilde{n}=128i$,
$1\le i\le 42$.  In $\Xi_2^{\mathbb{F}}$, $n=\tilde{n}=512i$,
$1\le i\le 10$.

\looseness=-1
The unpermuted ($P_n=I_n$) singular values are
\emph{logarithmically equidistributed},
\begin{displaymath}
  \sigma_i=\sigma[\xi,n,I_n]_i=2^y,\quad
  y=\xi\left(1-\frac{i-1}{n-1}\right),\quad
  1\le i\le n;\quad
  \lg\sigma_{i+1}-\lg\sigma_i=\frac{\xi}{1-n}.
\end{displaymath}
For example,
$\mathop{\mathrm{diag}}(\Sigma[\xi=-3,n=4,I_4])=\begin{bmatrix}1/8\!&\!1/4\!&\!1/2\!&\!1\end{bmatrix}^T$.
The permuted singular values
$\Sigma[\xi,n,P_n^{}]=P_n^{}\Sigma[\xi,n,I_n^{}]P_n^T$ can be in
ascending, descending, or any random order.  In the former two cases,
the smallest singular values are tightly clustered.

Let an input matrix
$G[\xi,n,P_n^{}]=U_n^{}\Sigma[\xi,n,P_n^{}]V_n^{\ast}$.  For
$\Xi_1^{\mathbb{F}}$ a random $P_n^{}$, same for both $\mathbb{F}$,
was taken for each $n$.  For $\Xi_2^{\mathbb{F}}$, three input
matrices were generated for each $n$ and $\mathbb{F}$, with ascending,
descending, and a random $P_n^{}$ order of $\Sigma$, same for both
$\mathbb{F}$.  The random unitary matrices $U_n^{}$ and $V_n^{\ast}$
were implicitly generated by two applications, from the left and from
the right, of the LAPACK's testing routine \texttt{xLAROR},
$\text{\texttt{x}}\in\{\text{\texttt{D}},\text{\text{\texttt{Z}}}\}$,
converted to work in quadruple precision.  First, $U_n^{}\Sigma$, and
then $G=(U_n^{}\Sigma)V_n^{\ast}$ were obtained.  The resulting $G$
was rounded to double precision and stored, as well as $\Sigma$.
\subsection{The batched EVD results}\label{ss:5.2}
\looseness=-1
\Cref{f:5.1} shows the run-time ratio, batch by batch, of calling the
LAPACK-like routine for each matrix in a batch and invoking the
vectorized EVD for eight (\texttt{d} and \texttt{z}) or $16$
(\texttt{s} and \texttt{c}) matrices at once.  In lines~\ref{al:24}
and~\ref{al:25} of \cref{a:z8jac2}, and in line~\ref{dl:16} of
\cref{a:d8jac2},
$\mathrm{\mathop{fl}}(\mathrm{e}^{\mathrm{i}\alpha}\tan\varphi)$ was
further divided by $\mathop{\mathrm{fl}}(\sec\varphi)$ to get
$\mathop{\mathrm{fl}}(\mathrm{e}^{\mathrm{i}\alpha}\sin\varphi)$, as
in~\cref{e:jacsec}.  This was also done in \texttt{s} and \texttt{c}.
The complex LAPACK-like routines were adapted for taking the
$(2,1)$ matrix element, the complex conjugate of the $(1,2)$ element
\texttt{B}, as input, and
$\text{\texttt{SN1}}=\mathop{\mathrm{fl}}(\mathrm{e}^{\mathrm{i}\alpha}\sin\varphi)$
was kept in the split form, to eliminate any otherwise unavoidable
pre-/post-processing overhead.

\looseness=-1
A parallel OpenMP \texttt{for} loop split the work within a batch
evenly among the threads.  Each thread thus processed
$2^{28}/32=2^{23}$ matrices per batch.  Every \texttt{xLAEV2}
invocation, had it not been inlined, would have involved several
function calls, so these results are a lower bound on run-time of any
semantically unchanged library routine.  The results are satisfactory
despite their noticeable dispersion, with the single precision
versions of the batched EVD being more performant than the double
precision ones due to twice the number of single versus double
precision lanes per widest vector.

\begin{figure}[hbtp]
  \begin{center}
    \includegraphics{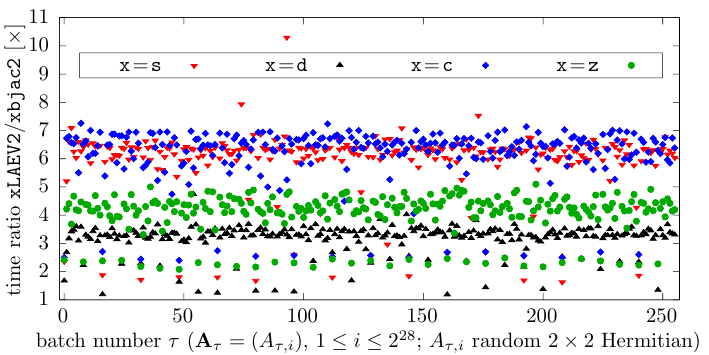}
  \end{center}
  \caption{Batch run-time ratios of the LAPACK-based and the vectorized batched EVDs.}
  \label{f:5.1}
\end{figure}

\looseness=-1
The EVD's relative residual\footnote{Here and for all other error
measures that involve division, assume $0/0=0$.} is
$\rho_A^{\text{\texttt{x}}}=\|U\Lambda U^{\ast}-A\|_F^{}/\|A\|_F^{}$,
where
$\text{\texttt{x}}\in\{\text{\texttt{d}},\text{\texttt{z}},\text{\texttt{s}},\text{\texttt{c}}\}$
for the batched EVD in the corresponding precision and datatypes, and
$\text{\texttt{x}}\in\{\text{\texttt{D}},\text{\texttt{Z}},\text{\texttt{S}},\text{\texttt{C}}\}$
for the respective LAPACK-based one.  For a batch $\tau$, let
$\rho_{\tau}^{\text{\texttt{x}}}=\max_i^{}\rho_{A_{\tau,i}}^{\text{\texttt{x}}}$,
$1\le i\le 2^{28}$.  Then, in \cref{f:5.2} the ratios
$\rho_{\tau}^{\text{\texttt{S}}}/\rho_{\tau}^{\text{\texttt{s}}}$ and
$\rho_{\tau}^{\text{\texttt{D}}}/\rho_{\tau}^{\text{\texttt{d}}}$ show
that, on average, real batched EVDs are a bit more accurate than the
LAPACK-based ones, but
$\rho_{\tau}^{\text{\texttt{C}}}/\rho_{\tau}^{\text{\texttt{c}}}$ and
$\rho_{\tau}^{\text{\texttt{Z}}}/\rho_{\tau}^{\text{\texttt{z}}}$, as
indicated in \cref{r:hypot}, demonstrate that a catastrophic loss of
accuracy of the eigenvectors (which in this case are no longer of the
unit norm) is possible when the components of \texttt{B} are of small
\emph{subnormal} and close enough magnitudes.  If they had been (close
to) normal, this issue would have been avoided.  A further explanation
is left for \cref{ss:SM7.1}, along with more EVD testing results.
Observe that \emph{the upscaling from~\cref{e:z} could have preserved
accuracy of the complex LAPACK routines in many problematic instances
by preventing $\bar{\hbox{\tt B}}/|\hbox{\tt B}|$ to be computed with
both components of similar, close to unit magnitudes}.  However,
certain pathological cases are unavoidable even with \cref{a:z8jac2}.
Consider the following matrix
\begin{equation}
  A=\begin{bmatrix}
  \nu/8&\check{\mu}\mp\mathrm{i}\check{\mu}\\
  \check{\mu}\pm\mathrm{i}\check{\mu}&\nu/8
  \end{bmatrix},\qquad
  \mathop{\mathrm{fl}}(|a_{12}|)=\check{\mu}=\mathop{\mathrm{fl}}(|a_{21}|)\implies
  \mathop{\mathrm{fl}}(\mathrm{e}^{\mathrm{i}\alpha})=1\pm\mathrm{i}.
  \label{e:evil}
\end{equation}
Then, from~\cref{e:z}, $\zeta=0$, while $\mathop{\mathrm{fl}}(\tan\varphi)=1$,
$\mathop{\mathrm{fl}}(\cos\varphi)=\mathop{\mathrm{fl}}(1/\mathop{\mathrm{fl}}(\sqrt{2}))$,
and therefore
\begin{displaymath}
  \mathop{\mathrm{fl}}(U)=\mathop{\mathrm{fl}}(\cos\varphi)
  \begin{bmatrix}
    1 & -1\pm\mathrm{i}\\
    1\pm\mathrm{i} & 1
  \end{bmatrix},\ \ \
  \det(\mathop{\mathrm{fl}}(U))\approx\frac{3}{\sqrt{2}},\ \ \
  \|\mathop{\mathrm{fl}}(u_1)\|_F=\|\mathop{\mathrm{fl}}(u_2)\|_F\approx\frac{\sqrt{3}}{\sqrt{2}}.
\end{displaymath}

\begin{figure}[hbtp]
  \begin{center}
    \includegraphics{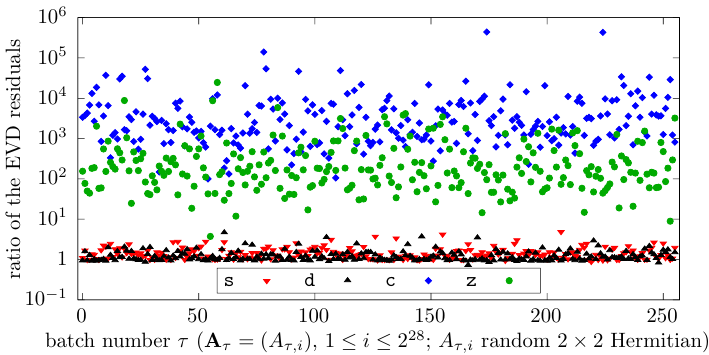}
  \end{center}
  \caption{Per-batch ratios
    $\rho_{\tau}^{\text{\texttt{S}}}/\rho_{\tau}^{\text{\texttt{s}}}$,
    $\rho_{\tau}^{\text{\texttt{D}}}/\rho_{\tau}^{\text{\texttt{d}}}$,
    $\rho_{\tau}^{\text{\texttt{C}}}/\rho_{\tau}^{\text{\texttt{c}}}$,
    and
    $\rho_{\tau}^{\text{\texttt{Z}}}/\rho_{\tau}^{\text{\texttt{z}}}$
    of the EVD's relative residuals.}
  \label{f:5.2}
\end{figure}

\looseness=-1
Neither \cref{a:z8jac2} nor \texttt{ZLAEV2} can escape this
miscomputing of the eigenvectors $U$ of $A$ from~\cref{e:evil}.  If
the strict standard conformance were not required, setting the
Denormals Are Zero (DAZ) CPU flag would convert $\pm\check{\mu}$ on
input to zero and the EVD of (now diagonal) $\widetilde{A}$ would be
correctly computed, even with \texttt{ZLAEV2}, but, e.g., matrices
with all subnormal elements would be zeroed out by both algorithms.
If only the \emph{post}-scaling subnormal values were zeroed out
(e.g., by setting the Flush To Zero (FTZ) CPU flag, but not DAZ,
before line~\ref{al:8} in \cref{a:z8jac2}), then the issues with $A$
and fully subnormal matrices would vanish, but this ``fix'' could turn
a nonsingular ill-conditioned matrix into an exactly singular one (it
depends on the context if this is an issue).  Thus, if the input data
range is too wide for the scaling to make all matrix elements normal,
it is safest to compute the EVD in a datatype with wider exponents.
\subsection{The SVD results}\label{ss:5.3}
Only a subset of the complex variant's results is shown here, with the
rest presented in \cref{ss:SM7.2}.
\subsubsection{Dataset $\Xi_1^{\mathbb{F}}$}\label{sss:5.3.1}
\Cref{f:5.3} shows the speedup of $\mathtt{zvjsvd}$ versus
\texttt{ZGESVJ} in two regimes.  The \texttt{max} values come from
comparing $\mathtt{zvjsvd}$ with the average of the corresponding
run-times of \texttt{ZGESVJ} under full machine load, i.e., when all
cores were busy running an instance of the latter on the same input at
the same time.  The \texttt{min} values are the result of a comparison
with the run-times of \texttt{ZGESVJ} when only one instance of it was
running on one core of an otherwise idle machine.  For $64$ threads,
e.g., the expected speedup for a given matrix order lies between the
corresponding \texttt{min} and \texttt{max} values.  A higher speedup
might have been expected, given that both the thread-based and the
vector parallelism were employed in $\mathtt{zvjsvd}$, but these
results can be at least partially explained by the reasons independent
of the actual hardware.

Foremost, $\mathtt{zvjsvd}$ with $\mathtt{MM}$ took $\approx\!10$
sweeps more on bigger matrices (see \cref{f:SM7.4}) than
\texttt{ZGESVJ} with $\mathtt{dR}$ ($\mathtt{ME}$, where applicable,
lowered the difference by $1$--$2$ sweeps).  This demonstrated need
for better parallel strategies for pointwise one-sided methods will
remain an issue even with the most optimized parallel
implementations.

\looseness=-1
As \cref{f:5.4} shows, $\mathtt{zvjsvd}$ with $64$ threads spent most
of its run-time on transforming the columns ($\bigstar$) and on
computing the Frobenius norms ($\spadesuit$) and the scaled
dot-products ($\clubsuit$, less so if the compensated summation was
left out).  For $\bigstar$, the $\max$-norm approximation of the
transformed columns of the iteration matrix was also computed.  For
$\spadesuit$, the Frobenius norms of those columns were recomputed,
while \texttt{ZGESVJ} updated them, with a periodic
recomputation~\cite{Drmac-97}.  The prescaling ($\heartsuit$) and the
EVD ($\diamondsuit$) of the Grammians jointly took less than $2\%$ of
the run-time on the bigger inputs.

\begin{figure}[hbtp]
  \begin{center}
    \includegraphics{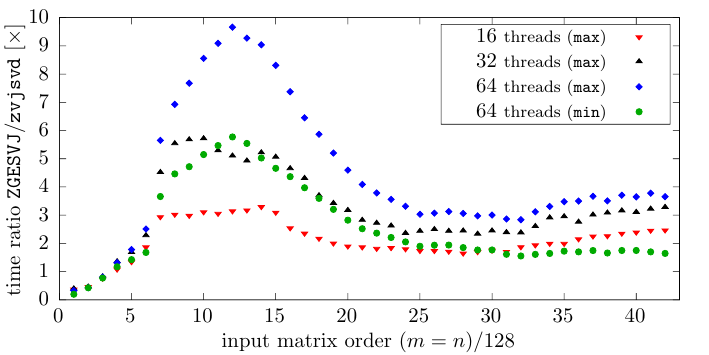}
  \end{center}
  \caption{Run-time ratios of \texttt{ZGESVJ} and $\mathtt{zvjsvd}$
    with $\mathtt{MM}$ on $\Xi_1^{\mathbb{C}}$.}
  \label{f:5.3}
\end{figure}

\begin{figure}[hbtp]
  \begin{center}
    \includegraphics{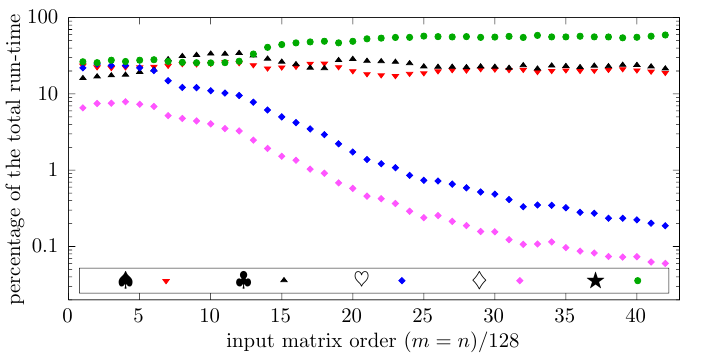}
  \end{center}
  \caption{Breakdown of the run-time of $\mathtt{zvjsvd}$
    (\cref{a:zvjsvd}) with $\mathtt{MM}$ on $\Xi_1^{\mathbb{C}}$ with
    $64$ threads.}
  \label{f:5.4}
\end{figure}

Define the relative error measures
$r_G^{}=\|U\Sigma V^{\ast}-G\|_F^{}/\|G\|_F^{}$,
$r_U^{}=\|U^{\ast}U-I\|_F^2$, $r_V^{}=\|V^{\ast}V-I\|_F^2$, and
$r_{\Sigma}^{}=\max_j^{}|\sigma_j'-\sigma_j^{}|/|\sigma_j^{}|$ for
\texttt{ZGESVJ}, where $\sigma_j'$ and $\sigma_j^{}$ are the $j$th
computed and exact singular value, respectively, and let $r_G'$,
$r_U'$, $r_V'$, and $r_{\Sigma}'$ be the same measures for
$\mathtt{zvjsvd}$.  \Cref{f:5.5} suggests that the singular values are
relatively accurate and the left singular vectors are orthogonal with
$\mathtt{zvjsvd}$ almost as with \texttt{ZGESVJ}, while the relative
SVD residuals are somewhat worse, probably due to a mild
($\max_n^{}{r_V'}<5\cdot 10^{-20}$) loss of orthogonality of the right
singular vectors.  These extra errors in the proposed method might be
caused by transforming $G$ and $V$ too many times (due to more sweeps)
compared to \texttt{ZGESVJ}, by the ``improper'' rotations that near
the end of the process lose orthogonality due to $\cos\varphi=1$ and
$|\tan\varphi|\lessapprox\sqrt{\varepsilon}$.

\begin{figure}[hbtp]
  \begin{center}
    \includegraphics{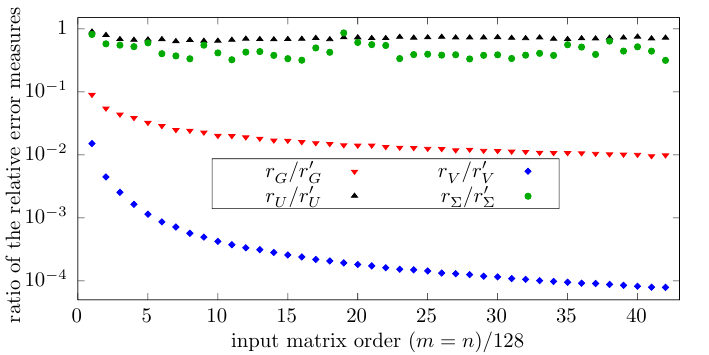}
  \end{center}
  \caption{Ratios of the relative error measures for the SVDs on
    $\Xi_1^{\mathbb{C}}$.}
  \label{f:5.5}
\end{figure}
\subsubsection{Dataset $\Xi_2^{\mathbb{F}}$}\label{sss:5.3.2}
For $\Xi_2^{\mathbb{F}}$, $\mathtt{ME}$ was used with
$\mathtt{zvjsvd}$.  The MKL routines required at least
$\text{\texttt{NSWEEP}}=30$ sweeps with a majority inputs (see
\cref{f:SM7.5}).  To circumvent that, \texttt{z}/\texttt{dgesvj.f}
source files were taken from the LAPACK repository and modified to
\texttt{z}/\texttt{dnssvj.f}, with \texttt{NSWEEP} as a function
argument, instead of being a hard-coded parameter, what could in
general benefit the users of the \texttt{xGESVJ} routines.

Both \texttt{ZGESVJ} and $\mathtt{zvjsvd}$ behaved as expected, with
$r_{\Sigma}^{}<6\cdot 10^{-2}$, $r_{\Sigma}'<4\cdot 10^{-1}$;
$r_G^{}<2\cdot 10^{-14}$, $r_G'<3\cdot 10^{-12}$;
$\max\{r_U^{},r_U'\}<2\cdot 10^{-22}$; while $r_V^{}<6\cdot 10^{-24}$
and $r_V'<9\cdot 10^{-20}$ indicate the same problem as with
$\Xi_1^{\mathbb{C}}$.  The results were similar for
$\Xi_2^{\mathbb{R}}$.
\section{Conclusions and future work}\label{s:6}
The strongest contribution of this paper is the vectorized algorithm
for the batched EVD of Hermitian matrices of order two.  It requires
no branching but only the basic bitwise and arithmetic operations,
with $\mathrm{fma}$, and $\max$ and $\min$ that filter out a single
$\mathtt{NaN}$ argument.  It is not applicable if trapping on
floating-point exceptions is enabled, and should be tuned for
non-default rounding modes, but it is faster and often more accurate
in every aspect than the matching sequence of \texttt{xLAEV2} calls.
Also, the computed scaled eigenvalues cannot overflow.

The proposed SVD method should be several times faster than
\texttt{xGESVJ} on modern CPUs and scale (sublinearly) with the number
of cores.  A fully tuned implementation, as well as ever-increasing
vector lengths on various platforms, should provide significantly
better speedups, while the stored column norms could be updated as
in~\cite{Drmac-97}.  The scaling principles from
\cref{ss:3.2,ss:3.4,ss:4.3} cost little performance-wise, do not
depend on parallelism or the pivot strategy, and thus could be
incorporated into \texttt{xGESVJ}, as well as
$\mathtt{zdpscl}^{[\prime]}$ (\cref{a:zdpscl,a:zdpsclcs}),
$\mathtt{zjrot}$ (\cref{a:zjrot}), and $\mathtt{zgsscl}$
(\cref{a:ZGS}), along with their single precision and/or real
counterparts.
\section*{Acknowledgments}
The author is thankful to Zlatko Drma\v{c} for mentioning a long time
ago that there might be room for low-level optimizations in his
Jacobi-type SVD routines in LAPACK, and to
Sanja\textsuperscript{$\dagger$}, Sa\v{s}a\textsuperscript{$\dagger$},
and Dean Singer, without whose material support this research would
never have been completed.  The author is also grateful for the
anonymous reviewers' comments that improved clarity of the paper, to
Hartwig Anzt, who was supportive in finding a modern testing machine,
and to Intel for a DevCloud account that provided a free remote access
to such machines.
\hrule\appendix
\section{Derivation and accuracy of the formulas from~\cref{ss:2.1}}\label{s:SM1}
Let $A$, $U$, and $\Lambda$ be as in~\cref{e:1}.  In \cref{ss:SM1.1}
the formulas from~\cref{ss:2.1} are derived.  In \cref{ss:SM1.2}
the relative errors induced while computing some of those formulas in
finite precision are given as a part of the proof of \cref{p:s2}.
\subsection{Derivation of the formulas from~\cref{ss:2.1}}\label{ss:SM1.1}
Equating the corresponding elements on both sides of
$U^{\ast}AU=\Lambda$ and assuming $\cos\varphi\ne 0$ it follows
\begin{equation}
  \begin{aligned}
    \lambda_{11}/\cos^2\varphi&=a_{11}+a_{22}\tan^2\varphi+2\Re\left(a_{21}\mathrm{e}^{-\mathrm{i}\alpha}\right)\tan\varphi,\\
    \lambda_{22}/\cos^2\varphi&=a_{11}\tan^2\varphi+a_{22}-2\Re\left(a_{21}\mathrm{e}^{-\mathrm{i}\alpha}\right)\tan\varphi,
  \end{aligned}
  \label{e:1.1}
\end{equation}
for the diagonal elements of $\Lambda$, and
\begin{equation}
  \frac{\lambda_{21}}{\cos^2\varphi}\mathrm{e}^{-\mathrm{i}\alpha}=a_{21}\mathrm{e}^{-\mathrm{i}\alpha}+(a_{22}-a_{11})\tan\varphi-\bar{a}_{21}\mathrm{e}^{\mathrm{i}\alpha}\tan^2\varphi=0,
  \label{e:1.2}
\end{equation}
for one of the remaining off-diagonal zeros, as well as
\begin{equation}
  \frac{\lambda_{12}}{\cos^2\varphi}\mathrm{e}^{\mathrm{i}\alpha}=\bar{a}_{21}\mathrm{e}^{\mathrm{i}\alpha}+(a_{22}-a_{11})\tan\varphi-a_{21}\mathrm{e}^{-\mathrm{i}\alpha}\tan^2\varphi=0
  \label{e:1.3}
\end{equation}
for the other.

From~\cref{e:1.2} the first two, and from~\cref{e:1.3} the last two
equations in
\begin{displaymath}
  \bar{a}_{21}\mathrm{e}^{\mathrm{i}\alpha}\tan^2\varphi-a_{21}\mathrm{e}^{-\mathrm{i}\alpha}=(a_{22}-a_{11})\tan\varphi=a_{21}\mathrm{e}^{-\mathrm{i}\alpha}\tan^2\varphi-\bar{a}_{21}\mathrm{e}^{\mathrm{i}\alpha}
\end{displaymath}
are obtained.  Ignoring the middle equation and regrouping the terms,
it follows
\begin{displaymath}
  a_{21}\mathrm{e}^{-\mathrm{i}\alpha}(1+\tan^2\varphi)=\bar{a}_{21}\mathrm{e}^{\mathrm{i}\alpha}(1+\tan^2\varphi),
\end{displaymath}
and, by canceling $1+\tan^2\varphi<\infty$ on both sides,
$a_{21}\mathrm{e}^{-\mathrm{i}\alpha}=\bar{a}_{21}\mathrm{e}^{\mathrm{i}\alpha}$,
i.e., $z=\bar{z}$, what is possible if and only if $z$ is real.  Therefore,
\begin{equation}
  \alpha=\arg(a_{21}),\qquad\mathrm{e}^{\mathrm{i}\alpha}=a_{21}/|a_{21}|.
  \label{e:1.4}
\end{equation}
If $a_{21}$ is real,
$\mathrm{e}^{\mathrm{i}\alpha}=\mathop{\mathrm{sign}}(a_{21})=\pm 1$.

Using~\cref{e:1.4} and $|\bar{a}_{21}|=|a_{21}|$,
\cref{e:1.1,e:1.2,e:1.3} become
\begin{align*}
  \lambda_{11}/\cos^2\varphi&=a_{11}+a_{22}\tan^2\varphi+2|a_{21}|\tan\varphi,\\
  \lambda_{22}/\cos^2\varphi&=a_{11}\tan^2\varphi+a_{22}-2|a_{21}|\tan\varphi,\\
  \frac{\lambda_{21}}{\cos^2\varphi}\mathrm{e}^{-\mathrm{i}\alpha}&=|a_{21}|+(a_{22}-a_{11})\tan\varphi-|a_{21}|\tan^2\varphi=0,\\
  \frac{\lambda_{12}}{\cos^2\varphi}\mathrm{e}^{\mathrm{i}\alpha}&=|a_{21}|+(a_{22}-a_{11})\tan\varphi-|a_{21}|\tan^2\varphi=0.
\end{align*}
The last two equations above are identical, so from either one it follows
\begin{equation}
  |a_{21}|(1-\tan^2\varphi)=(a_{11}-a_{22})\tan\varphi.
  \label{e:1.5}
\end{equation}

If $a_{21}=0$, and since $a_{11}$ and $a_{22}$ are arbitrary,
from~\cref{e:1.5} follows $\tan\varphi=0$, i.e., $U$ is the identity
matrix.  Else, if $a_{11}=a_{22}$, then $\tan\varphi=\pm 1$
satisfies~\cref{e:1.5}.  In all other cases,
\begin{displaymath}
  \frac{2|a_{21}|}{a_{11}-a_{22}}=\frac{2\tan\varphi}{1-\tan^2\varphi}=\tan(2\varphi),
\end{displaymath}
and, since $|\tan(2\varphi)|<\infty$, i.e., $|\varphi|<\pi/4$,
\begin{displaymath}
  \tan\varphi=\frac{\tan(2\varphi)}{1+\sqrt{1+\tan^2(2\varphi)}}.
\end{displaymath}
\subsubsection{Proof of \cref{l:hypot}}\label{sss:SM1.1.1}
See the main paper for its statement.

\begin{proof}
  Observe that $m$ and $M$ in~\cref{e:hypot} are exact, being normal
  or not, so any underflow in their formation is \emph{harmless}.  A
  possible underflow of $\mathop{\mathrm{fl}}(q)$ is harmless as well.
  If $\mathop{\mathrm{fl}}(q)$ is subnormal or zero,
  $\mathop{\mathrm{fma}}(\mathop{\mathrm{fl}}(q),\mathop{\mathrm{fl}}(q),1)=\mathop{\mathrm{fma}}(q,q,1)$,
  as if $\varepsilon_1=0$ below.

  Let $\mathop{\mathrm{fl}}(q)=q(1+\varepsilon_1)$, where
  $|\varepsilon_1|\le\varepsilon$.  Then,
  \begin{displaymath}
    \mathop{\mathrm{fl}}((\mathop{\mathrm{fl}}(q))^2+1)=\mathop{\mathrm{fma}}(\mathop{\mathrm{fl}}(q),\mathop{\mathrm{fl}}(q),1)=(q^2(1+\varepsilon_1)^2+1)(1+\varepsilon_3)=r(1+\varepsilon_3),
  \end{displaymath}
  where $|\varepsilon_3|\le\varepsilon$.  Expressing $r$ as the
  wanted, exact quantity times a relative error factor,
  $r=(q^2+1)(1+\varepsilon_2)$, and solving this equation for
  $\varepsilon_2$, it follows
  \begin{displaymath}
    \varepsilon_2=\frac{q^2}{q^2+1}\varepsilon_1(2+\varepsilon_1).
  \end{displaymath}
  Since $0\le q\le 1$, the fraction above ranges from $0$ (for $q=0$)
  to $1/2$ (for $q=1$), inclusive, irrespectively of $\varepsilon_1$.
  Therefore, $\varepsilon_2$, as a function of $\varepsilon_1$ when
  $q=1$, attains the extremal values for
  $\varepsilon_1=\mp\varepsilon$,
  \begin{displaymath}
    -\frac{\varepsilon(2-\varepsilon)}{2}\le\frac{1}{2}\varepsilon_1(2+\varepsilon_1)\le\frac{\varepsilon(2+\varepsilon)}{2}.
  \end{displaymath}
  However, if $q=1$ (equivalently, if $m=M$), then $\varepsilon_1=0$
  since the division is exact, and the two inequalities above are in
  fact strict, i.e., neither equality is possible.

  Now,
  $\mathop{\mathrm{fl}}(\sqrt{\mathop{\mathrm{fma}}(\mathop{\mathrm{fl}}(q),\mathop{\mathrm{fl}}(q),1)})=\sqrt{(q^2+1)(1+\varepsilon_2)(1+\varepsilon_3)}(1+\varepsilon_4)$,
  where $|\varepsilon_4|\le\varepsilon$.  The final multiplication by
  $M$ gives, with $|\varepsilon_5|\le\varepsilon$,
  \begin{displaymath}
    \mathop{\mathrm{fl}}(\mathop{\mathrm{hypot}}(x,y))=M\sqrt{q^2+1}\sqrt{1+\varepsilon_2}\sqrt{1+\varepsilon_3}(1+\varepsilon_4)(1+\varepsilon_5)=\delta_2\mathop{\mathrm{hypot}}(x,y),
  \end{displaymath}
  where
  $\delta_2=\sqrt{1+\varepsilon_2}\sqrt{1+\varepsilon_3}(1+\varepsilon_4)(1+\varepsilon_5)$
  is minimized for
  $\varepsilon_3=\varepsilon_4=\varepsilon_5=-\varepsilon$ and
  $\varepsilon_2=-\varepsilon(2-\varepsilon)/2$, and maximized for
  $\varepsilon_3=\varepsilon_4=\varepsilon_5=\varepsilon$ and
  $\varepsilon_2=\varepsilon(2+\varepsilon)/2$.
\end{proof}

\begin{remark}\label{r:uflexact}
  When considering relative floating-point accuracy of a computation,
  any underflow by itself, in isolation, is harmless if the result is
  \emph{exact}, since the relative error is zero.  But it is too
  cumbersome to always state this obvious exception.
\end{remark}
\subsection{Proof of~\cref{p:s2}}\label{ss:SM1.2}
See the main paper for its statement.
\begin{proof}
  If $a_{21}^{}$ is real,
  $\mathop{\mathrm{fl}}(|a_{21}^{}|)=\delta_o^{\mathbb{R}}|a_{21}^{}|$,
  $\delta_o^{\mathbb{R}}=1$, else
  $\mathop{\mathrm{fl}}(|a_{21}^{}|)=\delta_o^{\mathbb{C}}|a_{21}^{}|$,
  where $\delta_o^{\mathbb{C}}=\delta_2^{}$ (see \cref{l:hypot}).
  Thus, in the complex case, from~\cref{e:alpha} it follows
  \begin{displaymath}
    \mathop{\mathrm{fl}}(\cos\alpha)=\cos\alpha\frac{1+\varepsilon_1^{}}{\delta_2^{}}=\delta_{\alpha}'\cos\alpha,\qquad
    \mathop{\mathrm{fl}}(\sin\alpha)=\sin\alpha\frac{1+\varepsilon_2^{}}{\delta_2^{}}=\delta_{\alpha}''\sin\alpha,
  \end{displaymath}
  where $\max\{|\varepsilon_1^{}|,|\varepsilon_2^{}|\}\le\varepsilon$.
  The error factors $\delta_{\alpha}'$ and $\delta_{\alpha}''$ are
  maximized for $\varepsilon_1^{}=\varepsilon_2^{}=\varepsilon$ and
  $\delta_2^{}=\delta_2^-$, and minimized for
  $\varepsilon_1^{}=\varepsilon_2^{}=-\varepsilon$ and
  $\delta_2^{}=\delta_2^+$.  Let
  $\delta_{\alpha}^+=(1+\varepsilon)/\delta_2^-$ and
  $\delta_{\alpha}^-=(1-\varepsilon)/\delta_2^+$.  Since
  $\min\{|x|,|y|\}\le\sqrt{x^2\cos^2\alpha+y^2\sin^2\alpha}\le\max\{|x|,|y|\}$,
  \begin{equation}
    1-4.000000\,\varepsilon\lessapprox\delta_{\alpha}^-<|\mathop{\mathrm{fl}}(\mathrm{e}^{\pm\mathrm{i}\alpha})|<\delta_{\alpha}^+\lessapprox 1+4.000001\,\varepsilon.
    \label{e:alphare}
  \end{equation}
  The approximated multiples of $\varepsilon$ come
  from evaluating $(1-\delta_{\alpha}^-)/\varepsilon$ and
  $(\delta_{\alpha}^+-1)/\varepsilon$, as explained
  for~\cref{e:maxtanre,e:mintanre} below.

  From~\cref{e:tan2},
  $\mathop{\mathrm{fl}}(|a|)=|a_{11}^{}-a_{22}^{}|(1+\varepsilon_3^{})$,
  where $|\varepsilon_3^{}|\le\varepsilon$, and
  $\mathop{\mathrm{fl}}(2|a_{21}^{}|)=2\mathop{\mathrm{fl}}(|a_{21}^{}|)$.
  First, assume that
  $\mathop{\mathrm{fl}}(2|a_{21}^{}|/|a|)\le\mathop{\mathrm{fl}}(\sqrt{\nu})$.
  Then,
  \begin{displaymath}
    \mathop{\mathrm{fl}}(\tan 2\varphi)=\mathop{\mathrm{fl}}(2|a_{21}^{}|/|a|)\mathop{\mathrm{sign}}{a}=\delta_{2\varphi}^{\mathbb{F}}(2|a_{21}^{}|/|a|)\mathop{\mathrm{sign}}{a}=\delta_{2\varphi}^{\mathbb{F}}\tan 2\varphi,
  \end{displaymath}  
  where
  $\delta_{2\varphi}^{\mathbb{C}}=\delta_2^{}/(1+\varepsilon_3^{})$
  and $\delta_{2\varphi}^{\mathbb{R}}=1/(1+\varepsilon_3^{})$.
  Minimizing and maximizing these fractions, similarly as above, it
  follows
  \begin{equation}
    \begin{gathered}
      \delta_{2\varphi}^{\mathbb{C}-}=\frac{\delta_2^-}{1+\varepsilon}<\delta_{2\varphi}^{\mathbb{C}}<\frac{\delta_2^+}{1-\varepsilon}=\delta_{2\varphi}^{\mathbb{C}+},\\
      \delta_{2\varphi}^{\mathbb{R}-}=\frac{1}{1+\varepsilon}\le\delta_{2\varphi}^{\mathbb{R}}\le\frac{1}{1-\varepsilon}=\delta_{2\varphi}^{\mathbb{R}+}.
    \end{gathered}
    \label{e:tan2re}
  \end{equation}

  Else, let
  $\nu\ge\mathop{\mathrm{fl}}(2|a_{21}|/|a|)>\mathop{\mathrm{fl}}(\sqrt{\nu})$
  and assume that $|\mathop{\mathrm{fl}}(\tan 2\varphi)|$ was not
  bounded above by $\mathop{\mathrm{fl}}(\sqrt{\nu})$.  Then
  $|\mathop{\mathrm{fl}}(\tan\varphi)|$ would have been the exact
  unity if the square root in~\cref{e:jactan} was computed as finite,
  e.g., using
  $\mathop{\mathrm{hypot}}(1,\mathop{\mathrm{fl}}(\tan 2\varphi))$ and
  assuming
  \begin{displaymath}
    |x|\ge\mathop{\mathrm{fl}}(\sqrt{\nu})\implies\mathop{\mathrm{hypot}}(1,x)=|x|\
    \wedge\ 1+|x|=|x|
  \end{displaymath}
  for a representable $x$.  Therefore, regardless of the relative
  error in $\mathop{\mathrm{fl}}(\tan 2\varphi)$, the relative error
  in $\mathop{\mathrm{fl}}(\tan\varphi)$ could have only
  \emph{decreased} in magnitude from the one obtained with
  $\mathop{\mathrm{fl}}(\tan 2\varphi)=\mathop{\mathrm{fl}}(\sqrt{\nu})$,
  when also $|\mathop{\mathrm{fl}}(\tan\varphi)|=1$, since
  $\tan\varphi\to\pm 1$ monotonically for $\tan 2\varphi\to\pm\infty$.
  The Jacobi rotation is built from
  $\mathop{\mathrm{fl}}(\tan\varphi)$ (and the functions of $\alpha$
  in the complex case), while $\mathop{\mathrm{fl}}(\tan 2\varphi)$ is
  just an intermediate result and thus the relative error in it is not
  relevant as long as the one in $\mathop{\mathrm{fl}}(\tan\varphi)$
  is kept in check.

  By substituting $\mathop{\mathrm{fl}}(\tan 2\varphi)$ for
  $\tan 2\varphi$ in~\cref{e:jactan} and using the fused multiply-add
  for the argument of the square root, with
  $\max\{|\varepsilon_4|,|\varepsilon_5|,|\varepsilon_6|,|\varepsilon_7|\}\le\varepsilon$
  it follows that
  \begin{equation}
    \mathop{\mathrm{fl}}(\tan\varphi)=\frac{\delta_{2\varphi}^{\mathbb{F}}\tan 2\varphi(1+\varepsilon_7^{})}{(1+\sqrt{((\delta_{2\varphi}^{\mathbb{F}}\tan 2\varphi)^2+1)(1+\varepsilon_4^{})}(1+\varepsilon_5^{}))(1+\varepsilon_6^{})}=\delta_{\varphi}^{\mathbb{F}}\tan\varphi,
    \label{e:fltan}
  \end{equation}
  where $\varepsilon_4$, $\varepsilon_5$, $\varepsilon_6$, and
  $\varepsilon_7$ stand for the relative rounding errors of the
  $\mathrm{fma}$, the square root, the addition of one, and the
  division, respectively, and $\delta_{\varphi}^{\mathbb{F}}$ remains
  to be bounded.

  Let $y=\tan 2\varphi$.
  From~\cref{e:fltan}, by solving the equation
  \begin{displaymath}
    (\delta_{2\varphi}^{\mathbb{F}})^2 y^2+1=(y^2+1)(1+x)
  \end{displaymath}
  for $x$, it is possible to express the relative error present in the
  intermediate result of the $\mathrm{fma}$ before its rounding, as a
  function of $\varphi$ (due to $y$) and $\varepsilon$ (due to
  $(\delta_{2\varphi}^{\mathbb{F}})^2$),
  \begin{equation}
    x=\frac{y^2}{y^2+1}\varepsilon_{2\varphi}^{\mathbb{F}},\qquad
    \varepsilon_{2\varphi}^{\mathbb{F}}=(\delta_{2\varphi}^{\mathbb{F}})^2-1,\qquad
    0\le|x|<|\varepsilon_{2\varphi}^{\mathbb{F}}|,
    \label{e:x}
  \end{equation}
  since $0\le y^2/(y^2+1)<1$ for all $y$.  From~\cref{e:tan2re}
  $\varepsilon_{2\varphi}^{\mathbb{F}}$, and thus $x$, can be
  bounded as
  \begin{equation}
    (\delta_{2\varphi}^{\mathbb{F}-})^2-1=\varepsilon_{2\varphi}^{\mathbb{F}-}\le\varepsilon_{2\varphi}^{\mathbb{F}}\le\varepsilon_{2\varphi}^{\mathbb{F}+}=(\delta_{2\varphi}^{\mathbb{F}+})^2-1.
    \label{e:eps2phi}
  \end{equation}

  By rewriting the fma operation as above,~\cref{e:fltan} can be
  expressed as
  \begin{displaymath}
    \mathop{\mathrm{fl}}(\tan\varphi)=\frac{\tan 2\varphi}{1+\sqrt{\tan^2 2\varphi+1}\sqrt{1+x}\sqrt{1+\varepsilon_4^{}}(1+\varepsilon_5^{})}\frac{\delta_{2\varphi}^{\mathbb{F}}(1+\varepsilon_7^{})}{1+\varepsilon_6^{}},
  \end{displaymath}
  or, letting $r=\sqrt{\tan^2 2\varphi+1}$,
  $b=\sqrt{1+x}\sqrt{1+\varepsilon_4^{}}(1+\varepsilon_5^{})$, and
  $b'=\delta_{2\varphi}^{\mathbb{F}}(1+\varepsilon_7^{})/(1+\varepsilon_6^{})$,
  \begin{displaymath}
    \mathop{\mathrm{fl}}(\tan\varphi)=\frac{\tan 2\varphi}{1+br}b'=\frac{\tan 2\varphi}{(1+r)(1+d)}b'=\tan\varphi\frac{b'}{1+d}=\delta_{\varphi}^{\mathbb{F}}\tan\varphi,\qquad
    r\ge 1.
  \end{displaymath}
  Similarly as before, the equation $1+br=(1+r)(1+d)$ has to be solved
  for $d$ to factor out the relative error (i.e., $d$) from the
  remaining exact value (i.e., $1+r$).  Then,
  \begin{equation}
    d=\frac{r}{r+1}(b-1),\quad
    \frac{|b-1|}{2}\le|d|<|b-1|,\quad
    \delta_{\varphi}^{\mathbb{F}}=\frac{\delta_{2\varphi}^{\mathbb{F}}(1+\varepsilon_7^{})}{(1+d)(1+\varepsilon_6^{})}.
    \label{e:tanre}
  \end{equation}

  Maximizing
  $|\delta_{\varphi}^{\mathbb{F}}|=\delta_{\varphi}^{\mathbb{F}}$ is
  equivalent to maximizing
  $|\delta_{2\varphi}^{\mathbb{F}}|=\delta_{2\varphi}^{\mathbb{F}}$ as
  $\delta_{2\varphi}^{\mathbb{F}+}$ from~\cref{e:tan2re} and
  minimizing $d$, while setting $\varepsilon_6=-\varepsilon$ and
  $\varepsilon_7=\varepsilon$.  Minimizing $d$ is equivalent to
  letting $r\to\infty$ in~\cref{e:tanre} and minimizing $b$, what
  amounts to setting $\varepsilon_4^{}=\varepsilon_5^{}=-\varepsilon$
  and minimizing $x$ by letting $y^2\to\infty$ in~\cref{e:x} and
  taking $\varepsilon_{2\varphi}^{\mathbb{F}-}$ from~\cref{e:eps2phi}
  as the limiting value.  Therefore, the maximal value of
  $\delta_{\varphi}^{\mathbb{F}}$ is bounded above as
  \begin{equation}
    \max\delta_{\varphi}^{\mathbb{F}}<\hat{\delta}_{\varphi}^{\mathbb{F}}=\frac{\delta_{2\varphi}^{\mathbb{F}+}(1+\varepsilon)}{\sqrt{1+\varepsilon_{2\varphi}^{\mathbb{F}-}}(1-\varepsilon)^{5/2}}\lessapprox
    \begin{cases}
      1+\hphantom{0}5.500001\,\varepsilon,&\mathbb{F}=\mathbb{R},\\
      1+11.500004\,\varepsilon,&\mathbb{F}=\mathbb{C}.
    \end{cases}
    \label{e:maxtanre}
  \end{equation}
  First, $(\hat{\delta}_{\varphi}^{\mathbb{F}}-1)/\varepsilon$, i.e.,
  the factors multiplying $\varepsilon$ above, were expressed as
  functions of $\varepsilon$ in the scripts from \cref{ss:SM1.3}.  The
  factors were symbolically computed for
  $\varepsilon\in\{2^{-11},2^{-24},2^{-53},2^{-113}\}$, evaluated with
  $50$ digits of precision, manually rounded upwards to six
  decimal places, and the maximums over $p\in\{23,52,112\}$ were
  taken\footnote{Even though half precision ($p=10$) is not otherwise
  considered in the context of this proof, it is worth noting that the
  factors in that case differ from the presented ones by less than
  $0.1$.}.

  Minimizing $\delta_{\varphi}^{\mathbb{F}}$ is equivalent to
  minimizing $\delta_{2\varphi}^{\mathbb{F}}$ as
  $\delta_{2\varphi}^{\mathbb{F}-}$ from~\cref{e:tan2re} and
  maximizing $d$, while setting $\varepsilon_6=\varepsilon$ and
  $\varepsilon_7=-\varepsilon$.  Maximizing $d$ is equivalent to
  letting $r\to\infty$ in~\cref{e:tanre} and maximizing $b$, what
  amounts to setting $\varepsilon_4^{}=\varepsilon_5^{}=\varepsilon$
  and maximizing $x$ by letting $y^2\to\infty$ in~\cref{e:x} and
  taking $\varepsilon_{2\varphi}^{\mathbb{F}+}$ from~\cref{e:eps2phi}
  as the limiting value.  Therefore, the minimal value of
  $\delta_{\varphi}^{\mathbb{F}}$ is bounded below (the factors
  multiplying $\varepsilon$ come from
  $(1-\check{\delta}_{\varphi}^{\mathbb{F}})/\varepsilon$), as
  \begin{equation}
    \min\delta_{\varphi}^{\mathbb{F}}>\check{\delta}_{\varphi}^{\mathbb{F}}=\frac{\delta_{2\varphi}^{\mathbb{F}-}(1-\varepsilon)}{\sqrt{1+\varepsilon_{2\varphi}^{\mathbb{F}+}}(1+\varepsilon)^{5/2}}\gtrapprox
    \begin{cases}
      1-\hphantom{0}5.500000\,\varepsilon,&\quad\mathbb{F}=\mathbb{R},\\
      1-11.500000\,\varepsilon,&\quad\mathbb{F}=\mathbb{C}.
    \end{cases}
    \label{e:mintanre}
  \end{equation}

  Due to monotonicity of all arithmetic operations involved in
  computing $\mathop{\mathrm{fl}}(\tan\varphi)$, its absolute value
  cannot exceed unity, regardless of $\delta_{\varphi}^{\mathbb{F}}$.
  Therefore, the magnitude of (either component of)
  $\mathrm{e}'=\mathop{\mathrm{fl}}(\mathrm{e}^{\mathrm{i}\alpha})\mathop{\mathrm{fl}}(\tan\varphi)$
  cannot increase from that of (the corresponding component of)
  $\mathop{\mathrm{fl}}(\mathrm{e}^{\mathrm{i}\alpha})$ after its
  multiplication by $\mathop{\mathrm{fl}}(\tan\varphi)$.

  From~\cref{e:jactan,e:jacsec,e:d1} it follows
  \begin{equation}
      \mathop{\mathrm{fl}}(\cos\varphi)=\frac{\delta_1^{}}{\sqrt{((\delta_{\varphi}^{\mathbb{F}})^2\tan^2\varphi+1)(1+\varepsilon_8^{})}}.
    \label{e:flcos}
  \end{equation}
  As done previously, let $t=\tan\varphi$ and solve
  $(\delta_{\varphi}^{\mathbb{F}})^2t^2+1=(t^2+1)(1+c)$ for $c$, to
  get
  \begin{equation}
    c=\frac{t^2}{t^2+1}\varepsilon_{\varphi}^{\mathbb{F}},\qquad
    \varepsilon_{\varphi}^{\mathbb{F}}=(\delta_{\varphi}^{\mathbb{F}})^2-1,\qquad
    0\le|c|<|\varepsilon_{\varphi}^{\mathbb{F}}|,
    \label{e:c}
  \end{equation}
  and re-express~\cref{e:flcos}, with $|\varepsilon_8|\le\varepsilon$
  coming from the $\mathrm{fma}$'s rounding, as
  \begin{equation}
    \mathop{\mathrm{fl}}(\cos\varphi)=\frac{1}{\sqrt{\tan^2\varphi+1}}\frac{\delta_1^{}}{\sqrt{1+c}\sqrt{1+\varepsilon_8^{}}}=\delta_c^{\mathbb{F}}\cos\varphi.
    \label{e:cosre}
  \end{equation}
  Maximizing $|\delta_c^{\mathbb{F}}|=\delta_c^{\mathbb{F}}$ amounts
  to setting $\varepsilon_8^{}=-\varepsilon$ and
  $\delta_1^{}=\delta_1^+$, while minimizing $c$ by letting
  $t^2\to\infty$ in~\cref{e:c} and taking
  $\check{\varepsilon}_{\varphi}^{\mathbb{F}}=(\check{\delta}_{\varphi}^{\mathbb{F}})^2-1$
  as the limiting value.  Therefore,
  \begin{equation}
    \max\delta_c^{\mathbb{F}}<\hat{\delta}_c^{\mathbb{F}}=\frac{\delta_1^+}{\sqrt{1+\check{\varepsilon}_{\varphi}^{\mathbb{F}}}\sqrt{1-\varepsilon}}\lessapprox
    \begin{cases}
      1+\hphantom{0}8.000002\,\varepsilon,&\mathbb{F}=\mathbb{R},\\
      1+14.000006\,\varepsilon,&\mathbb{F}=\mathbb{C}.
    \end{cases}
    \label{e:maxcosre}
  \end{equation}
  Minimizing $\delta_c^{\mathbb{F}}$ amounts to setting
  $\varepsilon_8^{}=\varepsilon$ and $\delta_1^{}=\delta_1^-$, while
  maximizing $c$ by letting $t^2\to\infty$ in~\cref{e:c} and taking
  $\hat{\varepsilon}_{\varphi}^{\mathbb{F}}=(\hat{\delta}_{\varphi}^{\mathbb{F}})^2-1$
  as the limiting value.  Therefore,
  \begin{equation}
    \min\delta_c^{\mathbb{F}}>\check{\delta}_c^{\mathbb{F}}=\frac{\delta_1^-}{\sqrt{1+\hat{\varepsilon}_{\varphi}^{\mathbb{F}}}\sqrt{1+\varepsilon}}\gtrapprox
    \begin{cases}
      1-\hphantom{0}8.000000\,\varepsilon,&\mathbb{F}=\mathbb{R},\\
      1-14.000000\,\varepsilon,&\mathbb{F}=\mathbb{C}.
    \end{cases}
    \label{e:mincosre}
  \end{equation}
  The floating-point arithmetic operations involved in computing
  $\mathop{\mathrm{fl}}(\cos\varphi)$ are monotonic, so
  $0\le\mathop{\mathrm{fl}}(\cos\varphi)\le 1$ regardless of
  $\delta_c^{\mathbb{F}}$.

  Computation of the eigenvalues proceeds, with $o=2|a_{21}^{}|$ and
  $\tilde{o}=\delta_o^{\mathbb{F}}o$, as
  \begin{equation}
    \begin{aligned}
      \mathop{\mathrm{fl}}(\lambda_1')&=\mathop{\mathrm{fma}}(\mathop{\mathrm{fma}}(a_{22}^{},\tilde{t},\hphantom{-}\tilde{o}),\tilde{t},a_{11}^{}),\\
      \mathop{\mathrm{fl}}(\lambda_2')&=\mathop{\mathrm{fma}}(\mathop{\mathrm{fma}}(a_{11}^{},\tilde{t},-\tilde{o}),\tilde{t},a_{22}^{}),
    \end{aligned}
    \label{e:fllam}
  \end{equation}
  where $\tilde{t}=\mathop{\mathrm{fl}}(\tan\varphi)$ and
  $|\tilde{t}|\le 1$.  Then, with
  $\max\{|\varepsilon_{10}^{}|,|\varepsilon_{11}^{}|,|\varepsilon_{12}^{}|,|\varepsilon_{13}^{}|\}\le\varepsilon$,
  \begin{displaymath}
    \begin{aligned}
      \mathop{\mathrm{fl}}(\lambda_1')&=((a_{22}^{}\tilde{t}+\tilde{o})(1+\varepsilon_{10}^{})\tilde{t}+a_{11}^{})(1+\varepsilon_{12}^{}),\\
      \mathop{\mathrm{fl}}(\lambda_2')&=((a_{11}^{}\tilde{t}-\tilde{o})(1+\varepsilon_{11}^{})\tilde{t}+a_{22}^{})(1+\varepsilon_{13}^{}),
    \end{aligned}
  \end{displaymath}
  what gives, after taking the absolute values, applying the triangle
  inequality, and using~\cref{e:scla} to bound $|a_{ij}^{}|$ from
  above, with $\hat{\delta}_o^{\mathbb{R}}=1$ and
  $\hat{\delta}_o^{\mathbb{C}}=\delta_2^+$,
  \begin{equation}
    |\mathop{\mathrm{fl}}(\lambda')|\le((\hat{a}+2\hat{a}\hat{\delta}_o^{\mathbb{F}})(1+\varepsilon)+\hat{a})(1+\varepsilon)=\hat{a}((1+2\hat{\delta}_o^{\mathbb{F}})(1+\varepsilon)+1)(1+\varepsilon),
    \label{e:fll'}
  \end{equation}
  where
  $|\mathop{\mathrm{fl}}(\lambda')|=\max\{|\mathop{\mathrm{fl}}(\lambda_1')|,|\mathop{\mathrm{fl}}(\lambda_2')|\}$.
  Bounding $\hat{a}$ by $\tilde{\nu}$ as in~\cref{e:scla} gives
  \begin{displaymath}
    |\mathop{\mathrm{fl}}(\lambda')|\le\tilde{\nu}((1+2\hat{\delta}_o^{\mathbb{F}})(1+\varepsilon)+1)(1+\varepsilon)=\tilde{\nu}\delta_{\mathbb{F}}'.
  \end{displaymath}
  Evaluating the scripts from \cref{ss:SM1.3} for all $\varepsilon$
  considered above shows that
  $4<\delta_{\mathbb{F}}'\ll 4\sqrt{2}\gtrapprox 5.656854$, so
  $|\mathop{\mathrm{fl}}(\lambda')|<\nu$.

  Since $\mathop{\mathrm{fl}}(\sec^2\varphi)\ge 1$, dividing
  $\mathop{\mathrm{fl}}(\lambda_1')$ and
  $\mathop{\mathrm{fl}}(\lambda_2')$ by it therefore cannot raise
  their magnitudes, and the final computed eigenvalues cannot
  overflow---thus far, with the assumption that no final result of any
  previous computation has underflowed.

  Regardless of the consequences of any underflow leading to
  $\tilde{t}$, $|\tilde{t}|\le 1$ in~\cref{e:fllam} due
  to~\cref{e:jactan}, and $|\mathop{\mathrm{fl}}(\lambda')|$ can be
  bounded above, with
  $a_2^1=\max\{|a_{11}^{}|,|a_{22}^{}|\}\le\tilde{\nu}$, by
  \begin{equation}
    ((a_2^1+\tilde{o})(1+\varepsilon)+a_2^1)(1+\varepsilon),
    \label{e:sml}
  \end{equation}
  similarly as in~\cref{e:fll'}, letting in both inequalities
  $\tilde{t}=1$.  If $a_2^1$ and $\tilde{o}$ are small, no overflow
  occurs.  If $a_2^1$ is large enough, a small enough $\tilde{o}$, no
  matter if accurate or not, cannot affect it by addition or
  subtraction in the default rounding mode, so~\cref{e:sml} becomes
  $a_2^1(2+\varepsilon)(1+\varepsilon)$.  Vice versa, a small enough
  $a_2^1$ cannot affect a large enough $\tilde{o}$, so~\cref{e:sml}
  simplifies to $\tilde{o}(1+\varepsilon)^2$.  No overflow is possible
  in either case.
\end{proof}
\subsection{The Wolfram Language scripts used in \cref{ss:SM1.2}}\label{ss:SM1.3}
These scripts were executed by the Wolfram Language Engine, version
12.3.1 for macOS.
\subsubsection{A script computing the relative error bounds for a real $A$}\label{sss:SM1.3.1}
In \cref{f:SM1.1}, \texttt{n} is the number of digits of precision for
\texttt{N[\ldots]}, and
\begin{displaymath}
  \text{\texttt{fem}}=(1-\check{\delta}_{\varphi}^{\mathbb{R}})/\varepsilon,\ \ \
  \text{\texttt{fep}}=(\hat{\delta}_{\varphi}^{\mathbb{R}}-1)/\varepsilon,\ \ \
  \text{\texttt{cem}}=(1-\check{\delta}_c^{\mathbb{R}})/\varepsilon,\ \ \
  \text{\texttt{cep}}=(\hat{\delta}_c^{\mathbb{R}}-1)/\varepsilon,\ \ \
  \text{\texttt{pel}}=\delta_{\mathbb{R}}'.
\end{displaymath}

\begin{figure}[hbtp]
\begin{verbatim}
#!/usr/bin/env wolframscript -print all
If[Length[$ScriptCommandLine]<3,Quit[]];
p=ToExpression[$ScriptCommandLine[[2]]];
n=ToExpression[$ScriptCommandLine[[3]]];
p1=-p-1; (* change p1 to -p if rounding is not to the nearest *)
p2=2^p1;
d1m[e_]:=(1-e)/(1+e); (* \delta_1^- *)
d1p[e_]:=(1+e)/(1-e); (* \delta_1^+ *)
dam[e_]:=1;
dap[e_]:=1;
ddm[e_]:=1/(1+e);
ddp[e_]:=1/(1-e);
edm[e_]:=(ddm[e])^2-1;
edp[e_]:=(ddp[e])^2-1;
dfm[e_]:=(ddm[e]*(1-e))/(Sqrt[1+edp[e]]*((1+e)^(5/2)));
dfp[e_]:=(ddp[e]*(1+e))/(Sqrt[1+edm[e]]*((1-e)^(5/2)));
fem[e_]:=(1-dfm[e])/e;
fep[e_]:=(dfp[e]-1)/e;
"fem="<>ToString[N[FullSimplify[fem[p2]],n]]
"fep="<>ToString[N[FullSimplify[fep[p2]],n]]
efm[e_]:=dfm[e]^2-1;
efp[e_]:=dfp[e]^2-1;
dcm[e_]:=d1m[e]/(Sqrt[1+efp[e]]*Sqrt[1+e]);
dcp[e_]:=d1p[e]/(Sqrt[1+efm[e]]*Sqrt[1-e]);
cem[e_]:=(1-dcm[e])/e;
cep[e_]:=(dcp[e]-1)/e;
"cem="<>ToString[N[FullSimplify[cem[p2]],n]]
"cep="<>ToString[N[FullSimplify[cep[p2]],n]]
pel[e_]:=(3*(1+e)+1)*(1+e);
"pel="<>ToString[N[FullSimplify[pel[p2]],n]]
\end{verbatim}
\caption{A script computing the relative error bounds for a real $A$ in \cref{p:s2}.}
\label{f:SM1.1}
\end{figure}
\subsubsection{A script computing the relative error bounds for a complex $A$}\label{sss:SM1.3.2}
In \cref{f:SM1.2}, \texttt{n} is the number of digits of precision for
\texttt{N[\ldots]}, and
\begin{displaymath}
  \begin{gathered}
    \text{\texttt{fam}}=(1-\delta_{\alpha}^-)/\varepsilon,\quad\text{\texttt{fap}}=(\delta_{\alpha}^+-1)/\varepsilon,\\
    \text{\texttt{fem}}=(1-\check{\delta}_{\varphi}^{\mathbb{C}})/\varepsilon,\
    \text{\texttt{fep}}=(\hat{\delta}_{\varphi}^{\mathbb{C}}-1)/\varepsilon,\
    \text{\texttt{cem}}=(1-\check{\delta}_c^{\mathbb{C}})/\varepsilon,\
    \text{\texttt{cep}}=(\hat{\delta}_c^{\mathbb{C}}-1)/\varepsilon,\
    \text{\texttt{pel}}=\delta_{\mathbb{C}}'.
  \end{gathered}
\end{displaymath}

\begin{figure}[hbtp]
\begin{verbatim}
#!/usr/bin/env wolframscript -print all
If[Length[$ScriptCommandLine]<3,Quit[]];
p=ToExpression[$ScriptCommandLine[[2]]];
n=ToExpression[$ScriptCommandLine[[3]]];
p1=-p-1; (* change p1 to -p if rounding is not to the nearest *)
p2=2^p1;
d1m[e_]:=(1-e)/(1+e); (* \delta_1^- *)
d1p[e_]:=(1+e)/(1-e); (* \delta_1^+ *)
d2m[e_]:=((1-e)^(5/2))*Sqrt[1-(e*(2-e))/2]; (* \delta_2^- *)
d2p[e_]:=((1+e)^(5/2))*Sqrt[1+(e*(2+e))/2]; (* \delta_2^+ *)
dam[e_]:=(1-e)/d2p[e];
dap[e_]:=(1+e)/d2m[e];
fam[e_]:=(1-dam[e])/e;
fap[e_]:=(dap[e]-1)/e;
"fam="<>ToString[N[FullSimplify[fam[p2]],n]]
"fap="<>ToString[N[FullSimplify[fap[p2]],n]]
ddm[e_]:=d2m[e]/(1+e);
ddp[e_]:=d2p[e]/(1-e);
edm[e_]:=(ddm[e])^2-1;
edp[e_]:=(ddp[e])^2-1;
dfm[e_]:=(ddm[e]*(1-e))/(Sqrt[1+edp[e]]*((1+e)^(5/2)));
dfp[e_]:=(ddp[e]*(1+e))/(Sqrt[1+edm[e]]*((1-e)^(5/2)));
fem[e_]:=(1-dfm[e])/e;
fep[e_]:=(dfp[e]-1)/e;
"fem="<>ToString[N[FullSimplify[fem[p2]],n]]
"fep="<>ToString[N[FullSimplify[fep[p2]],n]]
efm[e_]:=dfm[e]^2-1;
efp[e_]:=dfp[e]^2-1;
dcm[e_]:=d1m[e]/(Sqrt[1+efp[e]]*Sqrt[1+e]);
dcp[e_]:=d1p[e]/(Sqrt[1+efm[e]]*Sqrt[1-e]);
cem[e_]:=(1-dcm[e])/e;
cep[e_]:=(dcp[e]-1)/e;
"cem="<>ToString[N[FullSimplify[cem[p2]],n]]
"cep="<>ToString[N[FullSimplify[cep[p2]],n]]
pel[e_]:=((1+2*d2p[e])*(1+e)+1)*(1+e);
"pel="<>ToString[N[FullSimplify[pel[p2]],n]]
(* for Lemma 3.2 only *)
eps[e_]:=FullSimplify[e*(2+dap[e])+(e^2)*(1+dap[e])];
"eps="<>ToString[N[FullSimplify[eps[p2]/p2],n]]
epp[e_]:=FullSimplify[Sqrt[2]*(eps[e]*(1+e)+e)];
"epp="<>ToString[N[FullSimplify[epp[p2]/p2],n]]
\end{verbatim}
\caption{A script computing the relative error bounds for a complex $A$ in \cref{p:s2}.}
\label{f:SM1.2}
\end{figure}
\section{Several vectorized routines mentioned in the main paper}\label{s:SM2}
\subsection{Vectorized eigendecomposition of a batch of real symmetric matrices of order two}\label{ss:SM2.1}
\Cref{a:d8jac2,a:dbjac2} are the real counterparts of
\cref{a:z8jac2,a:zbjac2}.  The complex algorithms work also with real
symmetric matrices on input as a special case, but the real ones are
faster.  The real algorithms return
$\pm\tan\varphi=\mathrm{e}^{\mathrm{i}\alpha}\tan\varphi=\mathop{\mathrm{sign}}a_{21}\tan\varphi$,
as the complex ones do with a real input.

\begin{algorithm}[hbtp]
  \caption{$\mathtt{d8jac2}$: a vectorized eigendecomposition of at most $\mathtt{s}$ double precision real symmetric matrices of order two with the Intel's AVX-512 intrinsics.}
  \label{a:d8jac2}
  \begin{algorithmic}[1]
    \REQUIRE{$\mathtt{i}$; addresses of $\tilde{\mathbf{a}}_{11},\tilde{\mathbf{a}}_{22},\tilde{\mathbf{a}}_{21},\pm\tan\tilde{\bm{\varphi}},\cos\tilde{\bm{\varphi}},\tilde{\bm{\lambda}}_1,\tilde{\bm{\lambda}}_2,\tilde{\mathtt{p}}$}
    \ENSURE{$\pm\mathop{\mathsf{tan}}\varphi$, $\mathop{\mathsf{cos}}\varphi$; $\bm{\lambda}_1$, $\bm{\lambda}_2$; $\mathfrak{p}$}
    \COMMENT{a permutation-indicating bitmask\\[4pt]}
    \COMMENT{vectors with all lanes set to a compile-time constant\hfill}
    \STATE{$\mathsf{0}=\mathop{\mathtt{setzero}}();\quad\mathsf{1}=\mathop{\mathtt{set1}}(1.0);\quad-\mathsf{0}=\mathop{\mathtt{set1}}(-0.0);\quad\bm{\nu}=\mathop{\mathtt{set1}}(\mathtt{DBL\_MAX});$}
    \STATE{$\sqrt{\bm{\nu}}=\mathop{\mathtt{set1}}(\text{\texttt{1.34078079299425956E+154}});\quad\bm{\eta}=\mathop{\mathtt{set1}}(1020.0);$}
    \COMMENT{$\mathop{\mathrm{fl}}(\sqrt{\nu})$:~\cref{e:tan2}\\[3pt]}
    \COMMENT{aligned loads of the $\mathtt{i}$th input vectors\hfill}
    \STATE{$\mathsf{a}_{11}=\mathop{\mathtt{load}}(\tilde{\mathbf{a}}_{11}+\mathtt{i});\quad\mathsf{a}_{22}=\mathop{\mathtt{load}}(\tilde{\mathbf{a}}_{22}+\mathtt{i});\quad\mathsf{a}_{21}=\mathop{\mathtt{load}}(\tilde{\mathbf{a}}_{21}+\mathtt{i});$}
    \COMMENT{$\Im{a_{21}}=0$\\[3pt]}
    \COMMENT{the scaling exponents $\bm{\zeta}$;\ \ \ lane-wise $\mathop{\mathrm{getexp}}(x)=\left\lfloor\lg|x|\right\rfloor,\left\lfloor\lg 0\right\rfloor=-\infty$\hfill}
    \STATE{$\bm{\zeta}_{11}\mskip-2mu=\mskip-2mu\mathop{\mathtt{sub}}(\bm{\eta},\mathop{\mathtt{getexp}}(\mathsf{a}_{11}));\mskip3.5mu\bm{\zeta}_{22}\mskip-2mu=\mskip-2mu\mathop{\mathtt{sub}}(\bm{\eta},\mathop{\mathtt{getexp}}(\mathsf{a}_{22}));\mskip3.5mu\bm{\zeta}_{21}\mskip-2mu=\mskip-2mu\mathop{\mathtt{sub}}(\bm{\eta},\mathop{\mathtt{getexp}}(\mathsf{a}_{21}));$}
    \STATE{$\bm{\zeta}=\mathop{\mathtt{min}}(\mathop{\mathtt{min}}(\bm{\zeta}_{11}^{},\bm{\zeta}_{22}^{}),\mathop{\mathtt{min}}(\bm{\zeta}_{21},\bm{\nu}));$}
    \COMMENT{finalize $\zeta$ from~\cref{e:z}}
    \STATE{$-\bm{\zeta}=\mathop{\mathtt{xor}}(\bm{\zeta},-\mathsf{0});$}
    \COMMENT{$\mathop{\mathtt{xor}}(\mathsf{x},-\mathsf{0})$ flips the sign bits in $\mathsf{x}$; optionally, store $-\bm{\zeta}$\\[3pt]}
    \COMMENT{the scaling of $\mathsf{A}\to\mathsf{2}^{\bm{\zeta}}\mathsf{A}$\hfill}
    \STATE{$\mathsf{a}_{21}=\mathop{\mathtt{scalef}}(\mathsf{a}_{21},\bm{\zeta});\ \ \ \mathsf{a}_{11}=\mathop{\mathtt{scalef}}(\mathsf{a}_{11},\bm{\zeta});\ \ \ \mathsf{a}_{22}=\mathop{\mathtt{scalef}}(\mathsf{a}_{22},\bm{\zeta});$}
    \COMMENT{$\mathsf{a}_{ij}=\mathsf{2}^{\bm{\zeta}}\mathsf{a}_{ij}$\\[3pt]}
    \COMMENT{the ``polar form'' of $\mathsf{a}_{21}$ ($\mathrm{e}^{\mathrm{i}\alpha}=\mathop{\mathrm{sign}}a_{21}$)\hfill}
    \STATE{$|\mathsf{a}_{21}|=\mathop{\mathtt{andnot}}(-\mathsf{0},\mathsf{a}_{21});$}
    \COMMENT{$\mathop{\mathtt{andnot}}(-\mathsf{0},\mathsf{x})=\mathsf{x}\wedge\neg{-\mathsf{0}}$ clears the sign bits}
    \STATE{$\mathop{\mathsf{sgn}}(\mathsf{a}_{21})=\mathop{\mathtt{and}}(\mathsf{a}_{21},-\mathsf{0});$}
    \COMMENT{$\mathop{\mathtt{and}}(\mathsf{x},-\mathsf{0})$ extracts the sign bits\\[3pt]}
    \COMMENT{$\mathop{\mathsf{cos}}\varphi$ and $\pm\mathop{\mathsf{tan}}\varphi$  (or $\mathop{\mathsf{sin}}\varphi=\mathop{\mathtt{div}}(\pm\mathop{\mathsf{tan}}\varphi,\mathop{\mathsf{sec}}\varphi)$)\hfill}
    \STATE{$\mathsf{o}\mskip-1mu=\mskip-1mu\mathop{\mathtt{scalef}}(|\mathsf{a}_{21}|,\mathsf{1});\ \mathsf{a}\mskip-1mu=\mskip-1mu\mathop{\mathtt{sub}}(\mathsf{a}_{11},\mathsf{a}_{22});\ |\mathsf{a}|\mskip-1mu=\mskip-1mu\mathop{\mathtt{andnot}}(-\mathsf{0},\mathsf{a});\ \mathop{\mathsf{sgn}}(\mathsf{a})\mskip-1mu=\mskip-1mu\mathop{\mathtt{and}}(\mathsf{a},-\mathsf{0});$}
    \STATE{$\mathop{\mathsf{tan}}2\varphi=\mathop{\mathtt{or}}(\mathop{\mathtt{min}}(\mathop{\mathtt{max}}(\mathop{\mathtt{div}}(\mathsf{o},|\mathsf{a}|),\mathsf{0}),\sqrt{\bm{\nu}}),\mathop{\mathsf{sgn}}(\mathsf{a}));$}
    \COMMENT{\cref{e:tan2}, here and above}
    \STATE{$\mathop{\mathsf{sec}^{\mathsf{2}}}2\varphi=\mathop{\mathtt{fmadd}}(\mathop{\mathsf{tan}}2\varphi,\mathop{\mathsf{tan}}2\varphi,\mathsf{1});$}
    \COMMENT{$\sec^2{2\varphi}<\infty$}
    \STATE{$\mathop{\mathsf{tan}}\varphi=\mathop{\mathtt{div}}(\mathop{\mathsf{tan}}2\varphi,\mathop{\mathtt{add}}(\mathsf{1},\mathop{\mathtt{sqrt}}(\mathop{\mathsf{sec}^{\mathsf{2}}}2\varphi)));$}
    \COMMENT{\cref{e:jactan}; $\tan\varphi$ \emph{without} $\mathrm{e}^{\mathrm{i}\alpha}$}
    \STATE{$\mathop{\mathsf{sec}^\mathsf{2}}\varphi=\mathop{\mathtt{fmadd}}(\mathop{\mathsf{tan}}\varphi,\mathop{\mathsf{tan}}\varphi,1);$}
    \COMMENT{\cref{e:jacsec}}
    \STATE{$\mathop{\mathsf{sec}}\varphi=\mathop{\mathtt{sqrt}}(\mathop{\mathsf{sec}^{\mathsf{2}}}\varphi);\quad\mathop{\mathsf{cos}}\varphi=\mathop{\mathtt{div}}(\mathsf{1},\mathop{\mathsf{sec}}\varphi);$}
    \COMMENT{\cref{e:jacsec}}
    \STATE{$\pm\mathop{\mathsf{tan}}\varphi=\mathop{\mathtt{xor}}(\mathop{\mathsf{tan}}\varphi,\mathop{\mathsf{sgn}}(\mathsf{a}_{21}));$}
    \COMMENT{$\mathtt{xor}$ multiplies the signs of its arguments\label{dl:16}}
    \STATE{$\mathop{\mathtt{store}}(\pm\tan\tilde{\bm{\varphi}}+\mathtt{i},\pm\mathop{\mathsf{tan}}\varphi);$}
    \COMMENT{$\mathrm{e}^{\mathrm{i}\alpha}\tan\varphi=\mathop{\mathrm{sign}}{a_{21}}\tan\varphi$\\[3pt]}
    \COMMENT{the eigenvalues\hfill}
    \STATE{$\bm{\lambda}_1'=\mathop{\mathtt{div}}(\mathop{\mathtt{fmadd}}(\mathop{\mathsf{tan}}\varphi,\mathop{\mathtt{fmadd}}(\mathsf{a}_{22}^{},\mathop{\mathsf{tan}}\varphi,\mathsf{o}),\mathsf{a}_{11}^{}),\mathop{\mathsf{sec}^{\mathsf{2}}}\varphi);$}
    \COMMENT{\cref{e:lamsec}}
    \STATE{$\mathop{\mathtt{store}}(\cos\tilde{\bm{\varphi}}+\mathtt{i},\mathop{\mathsf{cos}}\varphi);$}
    \COMMENT{below: $\mathop{\mathtt{fmsub}}(\mathsf{x},\mathsf{y},\mathsf{z})\equiv\mathop{\mathtt{fmadd}}(\mathsf{x},\mathsf{y},-\mathsf{z})$}
    \STATE{$\bm{\lambda}_2'=\mathop{\mathtt{div}}(\mathop{\mathtt{fmadd}}(\mathop{\mathsf{tan}}\varphi,\mathop{\mathtt{fmsub}}(\mathsf{a}_{11}^{},\mathop{\mathsf{tan}}\varphi,\mathsf{o}),\mathsf{a}_{22}^{}),\mathop{\mathsf{sec}^{\mathsf{2}}}\varphi);$}
    \COMMENT{\cref{e:lamsec}}
    \STATE{$\bm{\lambda}_1^{}=\mathop{\mathtt{scalef}}(\bm{\lambda}_1',-\bm{\zeta});\quad\mathop{\mathtt{store}}(\tilde{\bm{\lambda}}_1^{}+\mathtt{i},\bm{\lambda}_1^{});$}
    \COMMENT{backscale and store $\bm{\lambda}_1^{}$\label{dl:21}}
    \STATE{$\mathfrak{p}=\mathop{\mathtt{\_mm512\_cmplt\_pd\_mask}}(\bm{\lambda}_1',\bm{\lambda}_2');$}
    \COMMENT{lane-wise check if $\lambda_1'<\lambda_2'$}
    \STATE{$\bm{\lambda}_2^{}=\mathop{\mathtt{scalef}}(\bm{\lambda}_2',-\bm{\zeta});\quad\mathop{\mathtt{store}}(\tilde{\bm{\lambda}}_2^{}+\mathtt{i},\bm{\lambda}_2^{});$}
    \COMMENT{backscale and store $\bm{\lambda}_2^{}$\label{dl:23}}
    \STATE{$\tilde{\mathtt{p}}[\mathtt{i}/\mathtt{s}]=\mathop{\mathtt{\_cvtmaskX\_u32}}(\mathfrak{p});$}
    \COMMENT{store $\mathfrak{p}$ to $\tilde{\mathtt{p}}$, $\mathtt{X}=\mathtt{8}$ ($\mathtt{16}$ for AVX512F)\\[5pt]}
    \COMMENT{alternatively, in lines~\ref{dl:21} and~\ref{dl:23}, return $\bm{\lambda}_1'$ and $\bm{\lambda}_2'$, resp., instead of backscaling\hfill\null}
  \end{algorithmic}
\end{algorithm}

\begin{algorithm}[hbtp]
  \caption{$\mathtt{dbjac2}$: an OpenMP-parallel, AVX-512-vectorized eigendecomposition of a batch of $\tilde{r}$ double precision real symmetric matrices of order two.}
  \label{a:dbjac2}
  \begin{algorithmic}[1]
    \REQUIRE{$\tilde{r}$; $\tilde{\mathbf{a}}_{11}$, $\tilde{\mathbf{a}}_{22}$, $\tilde{\mathbf{a}}_{21}$; $\mathtt{p}$ also, in the context of \cref{a:zvjsvd}$[\mathbb{R}]$ only.}
    \ENSURE{$\pm\tan\tilde{\bm{\varphi}}$, $\cos\tilde{\bm{\varphi}}$; $\tilde{\bm{\lambda}}_1$, $\tilde{\bm{\lambda}}_2$; $\mathtt{p}$}
    \COMMENT{\texttt{unsigned} array of length $\tilde{r}/\mathtt{s}$}
    \STATE{\texttt{\#pragma omp parallel for default(shared)}}
    \COMMENT{optional}
    \FOR[$\mathtt{i}=\ell-1$]{$\mathtt{i}=0$ \TO $\tilde{r}-1$ \textbf{step} $\mathtt{s}$}
    \STATE{\textbf{if} $\mathtt{p}[\mathtt{i}/\mathtt{s}]\!=\!0$ \textbf{then continue};}
    \COMMENT{skip this vector on request of \cref{a:zvjsvd}$[\mathbb{R}]$}
    \STATE{$\mathop{\mathtt{d8jac2}}(\mathtt{i},\tilde{\mathbf{a}}_{11},\tilde{\mathbf{a}}_{22},\tilde{\mathbf{a}}_{21},\pm\tan\tilde{\bm{\varphi}},\cos\tilde{\bm{\varphi}},\tilde{\bm{\lambda}}_1,\tilde{\bm{\lambda}}_2,\mathtt{p}[,-\tilde{\bm{\zeta}}]);$}
    \COMMENT{\Cref{a:d8jac2}}
    \ENDFOR\COMMENT{$-\tilde{\bm{\zeta}}$ has to be returned without the optional backscaling of $\tilde{\bm{\lambda}}_1'$ and $\tilde{\bm{\lambda}}_2'$}
  \end{algorithmic}
\end{algorithm}
\subsection{Vectorized scaled dot-products with the compensated summation}\label{ss:SM2.2}
\Cref{a:zdpsclcs} for a vectorized scaled dot-product of two double
precision complex arrays with a possibly enhanced accuracy of the
result combines these ideas:
\begin{compactitem}
\item[G.] a trick from~\cite{Graillat-et-al-2015} to extract the
  truncated bits of a floating-point product by using one
  multiplication with rounding to $-\infty$, $\mathrm{RD}$, and one
  $\mathrm{fma}$ with rounding to nearest, as
  $c=\mathop{\mathrm{RD}}(a\cdot b)$,
  $c'=\mathop{\mathrm{fma}}(a,b,-c)\ge 0$, and $a\cdot b\approx c+c'$,
\item[M.] the 2Sum algorithm~\cite{Moller-65}, summarized
  in~\cite[Algorithm~4.4]{Muller-et-al-10} and vectorized in double
  precision in~\cref{e:d2sum}, for summation of two floating-point
  (scalar or vector) values $a$ and $b$ as $a+b=s+t$, where
  $s=\mathop{\mathrm{fl}}(a+b)$, and
\item[K.] the Kahan's compensated summation~\cite{Kahan-65} of a
  stream of floating-point values, modified as
  in~\cite[Algorithm~6.7]{Muller-et-al-10}, but with 2Sum instead of
  Fast2Sum.
\end{compactitem}
G.\ is easily implemented with the vector multiplication intrinsic
that takes the rounding mode indicator as an argument.  M., and its
combination with K., are branch-free.
\begin{equation}
  \begin{gathered}
    (\mathsf{s},\mathsf{t})=\mathop{\mathtt{d2sum}}(\mathsf{a},\mathsf{b});\\
    \mathsf{s}=\mathop{\mathtt{add}}(\mathsf{a},\mathsf{b}),\quad\mathsf{a}'=\mathop{\mathtt{sub}}(\mathsf{s},\mathsf{b}),\quad\mathsf{b}'=\mathop{\mathtt{sub}}(\mathsf{s},\mathsf{a}'),\\
    \mathsf{a}'=\mathop{\mathtt{sub}}(\mathsf{a},\mathsf{a}'),\quad\mathsf{b}'=\mathop{\mathtt{sub}}(\mathsf{b},\mathsf{b}'),\quad\mathsf{t}=\mathop{\mathtt{add}}(\mathsf{a}',\mathsf{b}').
  \end{gathered}
  \label{e:d2sum}
\end{equation}

\begin{algorithm}[hbtp]
  \caption{$\mathtt{zdpscl}'$: an enhanced vectorized complex scaled dot-product.}
  \label{a:zdpsclcs}
  \begin{algorithmic}[1]
    \REQUIRE{$g_q=(\Re{g_q},\Im{g_q}),0<\|g_q\|_F=(e_q,f_q);g_p=(\Re{g_p},\Im{g_p}),0<\|g_p\|_F=(e_p,f_p)$.}
    \ENSURE{$z=\mathop{\mathrm{fl}}(\check{g}_q^{\ast}\check{g}_p^{}=g_q^{\ast}g_p^{}/(\|g_q^{}\|_F^{}\|g_p^{}\|_F^{}))$.}
    \STATE{$-\mathsf{0}=\mathop{\mathtt{set1}}(-0.0);\quad-\mathsf{e}_p=\mathop{\mathtt{set1}}(-e_p),\ -\mathsf{e}_q=\mathop{\mathtt{set1}}(-e_q);$}
    \COMMENT{$-$exponents\\}
    \COMMENT{the $\Re$ and $\Im$ components of the partial scaled dot-product's\ldots\hfill}
    \STATE{$\Re{\mathsf{s}}=\mathop{\mathtt{setzero}}();\ \Re{\mathsf{t}}=\mathop{\mathtt{setzero}}();\ \Im{\mathsf{s}}=\mathop{\mathtt{setzero}}();\ \Im{\mathsf{t}}=\mathop{\mathtt{setzero}}();$}
    \COMMENT{values}
    \STATE{$\Re{\mathsf{s}}'=\mathop{\mathtt{setzero}}();\ \Re{\mathsf{t}}'=\mathop{\mathtt{setzero}}();\ \Im{\mathsf{s}}'=\mathop{\mathtt{setzero}}();\ \Im{\mathsf{t}}'=\mathop{\mathtt{setzero}}();$}
    \COMMENT{errors}
    \STATE{$r_{-\infty}\equiv\mathtt{\_MM\_FROUND\_TO\_NEG\_INF}|\mathtt{\_MM\_FROUND\_NO\_EXC};$}
    \COMMENT{rounding to $-\infty$}
    \FOR[sequentially]{$\mathtt{i}=0$ \TO $\tilde{m}-1$ \textbf{step} $\mathtt{s}$}
    \STATE{$\Re{\mathsf{g}_{\mathtt{i}j}}=\mathop{\mathtt{load}}(\Re{g_j}+\mathtt{i});\quad\Im{\mathsf{g}_{\mathtt{i}j}}=\mathop{\mathtt{load}}(\Im{g_j}+\mathtt{i});$}
    \COMMENT{$j\in\{p,q\}$}
    \STATE{$\Re{\check{\mathsf{g}}_{\mathtt{i}j}}=\mathop{\mathtt{scalef}}(\Re{\mathsf{g}_{\mathtt{i}j}},-\mathsf{e}_j);\quad\Im{\check{\mathsf{g}}_{\mathtt{i}j}}=\mathop{\mathtt{scalef}}(\Im{\mathsf{g}_{\mathtt{i}j}},-\mathsf{e}_j);$}
    \COMMENT{division by $\mathsf{2}^{\mathsf{e}_j}$\\}
    \COMMENT{\cref{e:CDP}, update $(\Re{\mathbf{w}})^T\Re{\mathbf{z}}$\hfill}
    \STATE{$\mathsf{c}=\mathop{\mathtt{mul\_round}}(\Re{\check{\mathsf{g}}_{\mathtt{i}p}},\Re{\check{\mathsf{g}}_{\mathtt{i}q}},r_{-\infty});$}
    \COMMENT{multiply and round down}
    \STATE{$\mathsf{c}'=\mathop{\mathtt{fmsub}}(\Re{\check{\mathsf{g}}_{\mathtt{i}p}},\Re{\check{\mathsf{g}}_{\mathtt{i}q}},\mathsf{c});$}
    \COMMENT{$\mathsf{c}+\mathsf{c}'\approx\Re{\check{\mathsf{g}}_{\mathtt{i}p}}\cdot\Re{\check{\mathsf{g}}_{\mathtt{i}q}}$, lane-wise, as per G.}
    \STATE{$(\Re{\mathsf{s}},\Re{\mathsf{t}})=\mathop{\mathtt{d2sum}}(\Re{\mathsf{s}},\mathop{\mathtt{add}}(\mathsf{c},\Re{\mathsf{t}}));$}
    \COMMENT{K.~\&~\cref{e:d2sum} on the $\Re$-stream}
    \STATE{$(\Re{\mathsf{s}}',\Re{\mathsf{t}}')=\mathop{\mathtt{d2sum}}(\Re{\mathsf{s}}',\mathop{\mathtt{add}}(\mathsf{c}',\Re{\mathsf{t}}'));$}
    \COMMENT{K.~\&~\cref{e:d2sum} on the $\Re'$-stream\\}
    \COMMENT{\cref{e:CDP}, update $(\Im{\mathbf{w}})^T\Im{\mathbf{z}}$\hfill}
    \STATE{$\mathsf{c}=\mathop{\mathtt{mul\_round}}(\Im{\check{\mathsf{g}}_{\mathtt{i}p}},\Im{\check{\mathsf{g}}_{\mathtt{i}q}},r_{-\infty});$}
    \COMMENT{multiply and round down}
    \STATE{$\mathsf{c}'=\mathop{\mathtt{fmsub}}(\Im{\check{\mathsf{g}}_{\mathtt{i}p}},\Im{\check{\mathsf{g}}_{\mathtt{i}q}},\mathsf{c});$}
    \COMMENT{$\mathsf{c}+\mathsf{c}'\approx\Im{\check{\mathsf{g}}_{\mathtt{i}p}}\cdot\Im{\check{\mathsf{g}}_{\mathtt{i}q}}$, lane-wise, as per G.}
    \STATE{$(\Re{\mathsf{s}},\Re{\mathsf{t}})=\mathop{\mathtt{d2sum}}(\Re{\mathsf{s}},\mathop{\mathtt{add}}(\mathsf{c},\Re{\mathsf{t}}));$}
    \COMMENT{K.~\&~\cref{e:d2sum} on the $\Re$-stream}
    \STATE{$(\Re{\mathsf{s}}',\Re{\mathsf{t}}')=\mathop{\mathtt{d2sum}}(\Re{\mathsf{s}}',\mathop{\mathtt{add}}(\mathsf{c}',\Re{\mathsf{t}}'));$}
    \COMMENT{K.~\&~\cref{e:d2sum} on the $\Re'$-stream\\}
    \COMMENT{\cref{e:CDP}, update $(\Re{\mathbf{w}})^T\Im{\mathbf{z}}$\hfill}
    \STATE{$\mathsf{c}=\mathop{\mathtt{mul\_round}}(\Re{\check{\mathsf{g}}_{\mathtt{i}p}},\Im{\check{\mathsf{g}}_{\mathtt{i}q}},r_{-\infty});$}
    \COMMENT{multiply and round down}
    \STATE{$\mathsf{c}'=\mathop{\mathtt{fmsub}}(\Re{\check{\mathsf{g}}_{\mathtt{i}p}},\Im{\check{\mathsf{g}}_{\mathtt{i}q}},\mathsf{c});$}
    \COMMENT{$\mathsf{c}+\mathsf{c}'\approx\Re{\check{\mathsf{g}}_{\mathtt{i}p}}\cdot\Im{\check{\mathsf{g}}_{\mathtt{i}q}}$, lane-wise, as per G.}
    \STATE{$(\Im{\mathsf{s}},\Im{\mathsf{t}})=\mathop{\mathtt{d2sum}}(\Im{\mathsf{s}},\mathop{\mathtt{add}}(\mathsf{c},\Im{\mathsf{t}}));$}
    \COMMENT{K.~\&~\cref{e:d2sum} on the $\Im$-stream}
    \STATE{$(\Im{\mathsf{s}}',\Im{\mathsf{t}}')=\mathop{\mathtt{d2sum}}(\Im{\mathsf{s}}',\mathop{\mathtt{add}}(\mathsf{c}',\Im{\mathsf{t}}'));$}
    \COMMENT{K.~\&~\cref{e:d2sum} on the $\Im'$-stream\\}
    \COMMENT{\cref{e:CDP}, update $-(\Im{\mathbf{w}})^T\Re{\mathbf{z}}$\hfill}
    \STATE{$-\Im{\check{g}_{\mathtt{i}p}}=\mathop{\mathtt{xor}}(\Im{\check{g}_{\mathtt{i}p}},-\mathsf{0});$}
    \COMMENT{flip the sign bit}
    \STATE{$\mathsf{c}=\mathop{\mathtt{mul\_round}}(-\Im{\check{\mathsf{g}}_{\mathtt{i}p}},\Re{\check{\mathsf{g}}_{\mathtt{i}q}},r_{-\infty});$}
    \COMMENT{multiply and round down}
    \STATE{$\mathsf{c}'=\mathop{\mathtt{fmsub}}(-\Im{\check{\mathsf{g}}_{\mathtt{i}p}},\Re{\check{\mathsf{g}}_{\mathtt{i}q}},\mathsf{c});$}
    \COMMENT{$\mathsf{c}+\mathsf{c}'\approx-\Im{\check{\mathsf{g}}_{\mathtt{i}p}}\cdot\Re{\check{\mathsf{g}}_{\mathtt{i}q}}$, lane-wise, as per G.}
    \STATE{$(\Im{\mathsf{s}},\Im{\mathsf{t}})=\mathop{\mathtt{d2sum}}(\Im{\mathsf{s}},\mathop{\mathtt{add}}(\mathsf{c},\Im{\mathsf{t}}));$}
    \COMMENT{K.~\&~\cref{e:d2sum} on the $\Im$-stream}
    \STATE{$(\Im{\mathsf{s}}',\Im{\mathsf{t}}')=\mathop{\mathtt{d2sum}}(\Im{\mathsf{s}}',\mathop{\mathtt{add}}(\mathsf{c}',\Im{\mathsf{t}}'));$}
    \COMMENT{K.~\&~\cref{e:d2sum} on the $\Im'$-stream}
    \ENDFOR\COMMENT{$g_p$ divided by $2^{e_p}$, $g_q$ by $2^{e_q}$}
    \STATE{$\Re{\hat{z}}=\mathop{\mathtt{reduce\_add}}(\Re{\mathsf{s}}')+\mathop{\mathtt{reduce\_add}}(\Re{\mathsf{t}}');$}
    \COMMENT{reduce the $\Re^{(\prime)}$-partial sums\ldots}
    \STATE{$\Re{\hat{z}}=\Re{\hat{z}}+\mathop{\mathtt{reduce\_add}}(\Re{\mathsf{t}});\quad\Re{\hat{z}}=\Re{\hat{z}}+\mathop{\mathtt{reduce\_add}}(\Re{\mathsf{s}});$}
    \STATE{$\Im{\hat{z}}=\mathop{\mathtt{reduce\_add}}(\Im{\mathsf{s}}')+\mathop{\mathtt{reduce\_add}}(\Im{\mathsf{t}}');$}
    \COMMENT{reduce the $\Im^{(\prime)}$-partial sums\ldots}
    \STATE{$\Im{\hat{z}}=\Im{\hat{z}}+\mathop{\mathtt{reduce\_add}}(\Im{\mathsf{t}});\quad\Im{\hat{z}}=\Im{\hat{z}}+\mathop{\mathtt{reduce\_add}}(\Im{\mathsf{s}});$}
    \STATE{$\mskip-7mu\mathop{\mathtt{\_mm\_store\_pd}}(\&z,\!\mathop{\mathtt{\_mm\_div\_pd}}(\mathop{\mathtt{\_mm\_set\_pd}}(\Im{\hat{z}},\!\Re{\hat{z}}),\!\mathop{\mathtt{\_mm\_set1\_pd}}(f_q\!\cdot\!f_p)\mskip-2mu)\mskip-2mu);$}
    \RETURN{$z;$}
    \COMMENT{$\hat{z}$ divided by the product of the norms' ``significands''}
  \end{algorithmic}
\end{algorithm}

\looseness=-1
\Cref{a:ddpsclcs} is the real variant of \cref{a:zdpsclcs}.  Both
algorithms require significantly more operations per iteration than
the simplest scaled dot-product (with a single $\mathtt{fma}$ and two
vector scalings in the real case), since each inlined $\mathtt{d2sum}$
computation requires six vector operations.  In many testing instances
the number of sweeps fell by one or two when these implementations
were employed instead of the simplest ones.  Accuracy of the results
was slightly improved, but with a degraded performance.

\begin{algorithm}[hbtp]
  \caption{$\mathtt{ddpscl}'$: an enhanced vectorized real scaled dot-product.}
  \label{a:ddpsclcs}
  \begin{algorithmic}[1]
    \REQUIRE{$g_q,\ 0<\|g_q\|_F=(e_q,f_q);\quad g_p,\ 0<\|g_p\|_F=(e_p,f_p)$.}
    \ENSURE{$d=\mathop{\mathrm{fl}}(\check{g}_q^T\check{g}_p^{}=g_q^Tg_p^{}/(\|g_q^{}\|_F^{}\|g_p^{}\|_F^{}))$.}
    \STATE{$-\mathsf{e}_p=\mathop{\mathtt{set1}}(-e_p);\quad-\mathsf{e}_q=\mathop{\mathtt{set1}}(-e_q);$}
    \STATE{$\mathsf{s}=\mathop{\mathtt{setzero}}();\quad\mathsf{t}=\mathop{\mathtt{setzero}}();\qquad\mathsf{s}'=\mathop{\mathtt{setzero}}();\quad\mathsf{t}'=\mathop{\mathtt{setzero}}();$}
    \FOR[sequentially]{$\mathtt{i}=0$ \TO $\tilde{m}-1$ \textbf{step} $\mathtt{s}$}
    \STATE{$\mathsf{g}_{\mathtt{i}p}=\mathop{\mathtt{load}}(g_p+\mathtt{i});\quad\mathsf{g}_{\mathtt{i}q}=\mathop{\mathtt{load}}(g_q+\mathtt{i});$}
    \STATE{$\check{\mathsf{g}}_{\mathtt{i}p}=\mathop{\mathtt{scalef}}(\mathsf{g}_{\mathtt{i}p},-\mathsf{e}_p);\quad\check{\mathsf{g}}_{\mathtt{i}q}=\mathop{\mathtt{scalef}}(\mathsf{g}_{\mathtt{i}q},-\mathsf{e}_q);$}
    \COMMENT{division by $\mathsf{2}^{\mathsf{e}_p}$ \&\ $\mathsf{2}^{\mathsf{e}_q}$}
    \STATE{$\mathsf{c}=\mathop{\mathtt{mul\_round}}(\check{\mathsf{g}}_{\mathtt{i}p},\check{\mathsf{g}}_{\mathtt{i}q},\mathtt{\_MM\_FROUND\_TO\_NEG\_INF}|\mathtt{\_MM\_FROUND\_NO\_EXC});$}
    \STATE{$\mathsf{c}'=\mathop{\mathtt{fmsub}}(\check{\mathsf{g}}_{\mathtt{i}p},\check{\mathsf{g}}_{\mathtt{i}q},\mathsf{c});$}
    \COMMENT{$\mathsf{c}+\mathsf{c}'\approx\check{\mathsf{g}}_{\mathtt{i}p}\cdot\check{\mathsf{g}}_{\mathtt{i}q}$, lane-wise, as per G.}
    \STATE{$(\mathsf{s},\mathsf{t})=\mathop{\mathtt{d2sum}}(\mathsf{s},\mathop{\mathtt{add}}(\mathsf{c},\mathsf{t}));$}
    \COMMENT{K.~\&~\cref{e:d2sum} on the value stream}
    \STATE{$(\mathsf{s}',\mathsf{t}')=\mathop{\mathtt{d2sum}}(\mathsf{s}',\mathop{\mathtt{add}}(\mathsf{c}',\mathsf{t}'));$}
    \COMMENT{K.~\&~\cref{e:d2sum} on the error stream}
    \ENDFOR\COMMENT{$g_p$ divided by $2^{e_p}$, $g_q$ by $2^{e_q}$}
    \STATE{$\hat{d}=\mathop{\mathtt{reduce\_add}}(\mathsf{s})+(\mathop{\mathtt{reduce\_add}}(\mathsf{t})+(\mathop{\mathtt{reduce\_add}}(\mathsf{s}')+\mathop{\mathtt{reduce\_add}}(\mathsf{t}')));$}
    \RETURN{$d=\hat{d}/(f_p\cdot f_q);$}
    \COMMENT{division by the product of the norms' ``significands''}
  \end{algorithmic}
\end{algorithm}
\section{Overflow conditions of a real and a complex dot-product}\label{s:SM3}
A real dot-product, $\mathbf{x}^T\mathbf{y}$, of columns $\mathbf{x}$
and $\mathbf{y}$ of length $m$, can be bounded in magnitude using the
triangle inequality as
\begin{equation}
  |\mathbf{x}^T\mathbf{y}|=\left|\sum_{i=1}^m x_i\cdot y_i\right|\le\sum_{i=1}^m|x_i\cdot y_i|=\sum_{i=1}^m|x_i||y_i|=|\mathbf{x}|^T|\mathbf{y}|\le m M^2,
  \label{e:RDP}
\end{equation}
where $\displaystyle M=\max_{1\le i\le m}M_i$ and
$M_i=\max\{|x_i|,|y_i|\}$.  From~\cref{e:RDP} and
\begin{equation}
  |\mathop{\mathrm{fl}}(\mathbf{x}^T\mathbf{y})-\mathbf{x}^T\mathbf{y}|\le m\varepsilon|\mathbf{x}|^T|\mathbf{y}|,
  \label{e:rump}
\end{equation}
assuming no particular order of evaluation, but also no overflow or
underflow when computing
$\mathop{\mathrm{fl}}(\mathbf{x}^T\mathbf{y})$
(see~\cite[Theorem~4.2]{Jeannerod-Rump-13}), it follows
\begin{equation}
  \begin{aligned}
    |\mathop{\mathrm{fl}}(\mathbf{x}^T\mathbf{y})|&=|\mathbf{x}^T\mathbf{y}+(\mathop{\mathrm{fl}}(\mathbf{x}^T\mathbf{y})-\mathbf{x}^T\mathbf{y})|\le|\mathbf{x}^T\mathbf{y}|+|\mathop{\mathrm{fl}}(\mathbf{x}^T\mathbf{y})-\mathbf{x}^T\mathbf{y}|\\
    &\le|\mathbf{x}|^T|\mathbf{y}|(1+m\varepsilon)\le m M^2(1+m\varepsilon).
  \end{aligned}
  \label{e:afldp}
\end{equation}

Formation of the Grammian matrices of order two in the real case,
if~\cref{e:scla} is to be achieved without any scaling, requires
$|\mathop{\mathrm{fl}}(\mathbf{x}^T\mathbf{y})|\le\tilde{\nu}$, so
it suffices to hold
\begin{equation}
  M\le\sqrt{\nu/(4\sqrt{2}m(1+m\varepsilon))}=\tau_m.
  \label{e:RM}
\end{equation}

A complex dot-product, $\mathbf{w}^{\ast}\mathbf{z}$, can be
decomposed into a sum of two real ones, each twice longer than the
column length, as
\begin{equation}
  \begin{aligned}
    \mathbf{w}^{\ast}\mathbf{z}&=\Re(\mathbf{w}^{\ast}\mathbf{z})+\mathrm{i}\Im(\mathbf{w}^{\ast}\mathbf{z})\\
    &=(\Re{\mathbf{w}})^T(\Re{\mathbf{z}})+(\Im{\mathbf{w}})^T(\Im{\mathbf{z}})+\mathrm{i}((\Re{\mathbf{w}})^T(\Im{\mathbf{z}})-(\Im{\mathbf{w}})^T(\Re{\mathbf{z}}))\\
    &=\begin{bmatrix}(\Re{\mathbf{w}})^T&(\Im{\mathbf{w}})^T\end{bmatrix}\begin{bmatrix}\Re{\mathbf{z}}\\\Im{\mathbf{z}}\end{bmatrix}+\mathrm{i}\begin{bmatrix}(\Re{\mathbf{w}})^T&-(\Im{\mathbf{w}})^T\end{bmatrix}\begin{bmatrix}\Im{\mathbf{z}}\\\Re{\mathbf{z}}\end{bmatrix}\\
    &=\mathbf{a}^T\mathbf{b}+\mathrm{i}\,\mathbf{c}^T\mathbf{d},
  \end{aligned}
  \label{e:CDP}
\end{equation}
where $\mathbf{a}$, $\mathbf{b}$, $\mathbf{c}$, and $\mathbf{d}$ stand
for the corresponding real arrays of length $\hat{m}=2 m$.

Let
$\widehat{M}_i=\max\{|\Re{w_i}|,|\Im{w_i}|,|\Re{z_i}|,|\Im{z_i}|\}$
and $\displaystyle\widehat{M}=\max_{1\le i\le m}\widehat{M}_i$.
Similarly to the real case, from~\cref{e:CDP,e:rump} it then follows
\begin{equation}
  \begin{aligned}
    |\mathop{\mathrm{fl}}(\mathbf{a}^T\mathbf{b})|&=|\mathbf{a}^T\mathbf{b}+(\mathop{\mathrm{fl}}(\mathbf{a}^T\mathbf{b})-\mathbf{a}^T\mathbf{b})|\le|\mathbf{a}^T\mathbf{b}|+|\mathop{\mathrm{fl}}(\mathbf{a}^T\mathbf{b})-\mathbf{a}^T\mathbf{b}|\\
    &\le|\mathbf{a}|^T|\mathbf{b}|(1+\hat{m}\varepsilon)\le\hat{m}\widehat{M}^2(1+\hat{m}\varepsilon).
  \end{aligned}
  \label{e:cfldp}
\end{equation}
The same relation holds if $\mathbf{a}$ and $\mathbf{b}$ are fully
replaced by $\mathbf{c}$ and $\mathbf{d}$, respectively.  Since
$\mathop{\mathrm{fl}}(\mathbf{w}^{\ast}\mathbf{z})=\mathop{\mathrm{fl}}(\mathbf{a}^T\mathbf{b})+\mathrm{i}\mathop{\mathrm{fl}}(\mathbf{c}^T\mathbf{d})$,
its magnitude can be bounded above, using~\cref{e:cfldp}, by
\begin{equation}
  |\mathop{\mathrm{fl}}(\mathbf{w}^{\ast}\mathbf{z})|=\sqrt{|\mathop{\mathrm{fl}}(\mathbf{a}^T\mathbf{b})|^2+|\mathop{\mathrm{fl}}(\mathbf{c}^T\mathbf{d})|^2}\le\sqrt{2}\hat{m}\widehat{M}^2(1+\hat{m}\varepsilon).
  \label{e:Cfldp}
\end{equation}
If
$|\mathop{\mathrm{fl}}(\mathbf{w}^{\ast}\mathbf{z})|\le\tilde{\nu}$
is required, then due to~\cref{e:Cfldp} it suffices to hold
\begin{equation}
  \widehat{M}\le\sqrt{\nu/(8\hat{m}(1+\hat{m}\varepsilon))}=\sqrt{\nu/(16m(1+2m\varepsilon))}=\hat{\tau}_m.
  \label{e:CM}
\end{equation}

\looseness=-1
Issues with underflow are ignored for the na\"{\i}ve Jacobi SVD\@.
Let $\tilde{\tau}_m=\tau_m$ in the real, and
$\tilde{\tau}_m=\hat{\tau}_m$ in the complex case.  The
constraints~\cref{e:RM} and~\cref{e:CM} on the magnitudes of (the
components of) the elements of the iteration matrix $G_k$ become
\begin{equation}
  \max\{\|\Re(G_k)\|_{\max},\|\Im(G_k)\|_{\max}\}\le\tilde{\tau}_m.
  \label{e:M}
\end{equation}
This equation establishes a ``safe region'' for the magnitudes of (the
components of) the elements of the iteration matrix, within which it
is guaranteed that both the formation of the Grammians and the
calculation of the Jacobi rotations without the prescaling from
\cref{ss:2.3} will succeed.  Not only that this region is relatively
narrow, but the transformations of the pivot column pairs could cause
the iteration matrix to fall outside it at the beginning of the
following iteration, as shown in \cref{ss:3.1}.

If the iteration matrix $G_k^{}$ is scaled as in~\cref{e:M}, no
element of $\mathop{\mathrm{fl}}(G_{k+1}^{})$ can overflow, due to
\cref{p:t}.  The constraint~\cref{e:M}, if re-evaluated by examining
the magnitudes of the affected (components of) elements, could become
violated, which is acceptable as long as all dot-products remain below
the limit~\cref{e:scla} by magnitude and the assumption of \cref{p:t}
holds.  If it does not hold, or if, once all dot-products for the
current step are obtained, at least one lands above~\cref{e:scla}, the
iteration matrix has to be rescaled according to~\cref{e:M}.  These
observations suffice for a fast implementation of the pointwise
Jacobi-type SVD method, applicable when the input matrix so permits.
\section{Proofs of \cref{l:rt,l:ct}}\label{s:SM4}
See their statements in the main paper.

\begin{proof}[Proof of \cref{l:rt}]
  From
  $g_{ip}'=(g_{ip}^{}\pm g_{iq}^{}\mathop{\mathrm{fl}}(\tan\varphi))\mathop{\mathrm{fl}}(\cos\varphi)$
  it follows
  \begin{equation}
    \mathop{\mathrm{fl}}(g_{ip}')=((g_{ip}^{}\pm g_{iq}^{}\mathop{\mathrm{fl}}(\tan\varphi))(1+\varepsilon_1^{})\mathop{\mathrm{fl}}(\cos\varphi))(1+\varepsilon_2^{}),
    \label{e:rfl}
  \end{equation}
  due to the final and only rounding performed when evaluating the
  fused multiply-add expression, with the relative error
  $\varepsilon_1$, $|\varepsilon_1|\le\varepsilon$.  It can now be
  obtained, by rearranging the above terms and noting that, for the
  multiplication by the cosine, $|\varepsilon_2|\le\varepsilon$,
  \begin{displaymath}
    \begin{aligned}
    \mathop{\mathrm{fl}}(g_{ip}')&=((g_{ip}^{}\pm g_{iq}^{}\mathop{\mathrm{fl}}(\tan\varphi))\mathop{\mathrm{fl}}(\cos\varphi))((1+\varepsilon_1^{})(1+\varepsilon_2^{}))\\
    &=g_{ip}'((1+\varepsilon_1^{})(1+\varepsilon_2^{}))=g_{ip}'(1+\varepsilon_p'),
    \end{aligned}
  \end{displaymath}
  and thus $(1-\varepsilon)^2\le 1+\varepsilon_p'\le(1+\varepsilon)^2$,
  what proves the first part of the first statement of the Lemma.  The
  second part,
  $(1-\varepsilon)^2\le 1+\varepsilon_q'\le(1+\varepsilon)^2$, is
  shown similarly.

  From~\cref{e:rfl}, $|\mathop{\mathrm{fl}}(\tan\varphi)|\le 1$, the
  assumption that $\max\{|g_{ip}^{}|,|g_{iq}^{}|\}\le\nu/2$, and
  the fact that $\nu/2$ is exactly representable in floating-point,
  so $\mathop{\mathrm{fl}}(\nu/2+\nu/2)=\nu$, it follows that
  $|d|(1+\varepsilon_1^{})\le\nu$, for
  $d=g_{ip}^{}\pm g_{iq}^{}\mathop{\mathrm{fl}}(\tan\varphi)$ (or for
  $d=g_{iq}^{}\mp g_{ip}^{}\mathop{\mathrm{fl}}(\tan\varphi)$).  Here,
  monotonicity of the inner addition, the inner multiplication, and
  the (outer) rounding of the $\mathrm{fma}$ operation are relied
  upon.  Multiplying $d(1+\varepsilon_1^{})$ by
  $\mathop{\mathrm{fl}}(\cos\varphi)\le 1$ and the subsequent
  (monotonous) rounding cannot yield a result of a magnitude strictly
  greater than $|d(1+\varepsilon_1^{})|\le\nu$, what proves the
  last statement of the Lemma.
\end{proof}

\begin{proof}[Proof of \cref{l:ct}]
  Assume that $|g_{ip}'|>0$ and
  $|g_{ip}'|/\mathop{\mathrm{fl}}(\cos\varphi)\ge|\Re{g_{iq}^{}}|$
  hold.  Let, from the first equation in~\cref{e:transf},
  $d=g_{ip}'/\mathop{\mathrm{fl}}(\cos\varphi)=a\cdot b+c$, where
  $a=g_{iq}^{}$,
  \begin{displaymath}
    \Re{b}=\mathop{\mathrm{fl}}(\cos\alpha)\mathop{\mathrm{fl}}(\tan\varphi)(1+\varepsilon_7^{}),\quad
    \Im{b}=\mathop{\mathrm{fl}}(\sin\alpha)\mathop{\mathrm{fl}}(\tan\varphi)(1+\varepsilon_8^{}),\quad
    \max\{|\varepsilon_7^{}|,|\varepsilon_8^{}|\}\le\varepsilon,
  \end{displaymath}
  $c=g_{ip}^{}$, and
  $\tilde{d}=\mathop{\mathrm{fma}}(a,b,c)=\mathop{\mathrm{RN}}(d)$, in
  the context of~\cref{e:zfma}.  In this notation,
  $0\ne|d|\ge|\Re{a}|$.  Let
  $\max\{|\varepsilon_1^{}|,|\varepsilon_2^{}|,|\varepsilon_3^{}|,|\varepsilon_4^{}|,|\varepsilon_5^{}|,|\varepsilon_6^{}|\}\le\varepsilon$.
  Then,
  \begin{displaymath}
    \begin{aligned}
      \Re{\tilde{d}}&=(\Re{a}\cdot\Re{b}+(\Re{c}-\Im{a}\cdot\Im{b})(1+\varepsilon_1))(1+\varepsilon_2)=(\Re{d}+(\Re{c}-\Im{a}\cdot\Im{b})\varepsilon_1)(1+\varepsilon_2),\\
      \Im{\tilde{d}}&=(\Re{a}\cdot\Im{b}+(\Im{c}+\Im{a}\cdot\Re{b})(1+\varepsilon_3))(1+\varepsilon_4)=(\Im{d}+(\Im{c}+\Im{a}\cdot\Re{b})\varepsilon_3)(1+\varepsilon_4).
    \end{aligned}
  \end{displaymath}
  After subtracting $\Re{d}$ ($\Im{d}$), rearranging the terms, and
  adding a zero, it follows
  \begin{displaymath}
    \begin{aligned}
      \Re{\tilde{d}}-\Re{d}&=\Re{d}\cdot\varepsilon_2+(\Re{c}-\Im{a}\cdot\Im{b}+(\Re{a}\cdot\Re{b}-\Re{a}\cdot\Re{b}))\varepsilon_1(1+\varepsilon_2),\\
      \Im{\tilde{d}}-\Im{d}&=\Im{d}\cdot\varepsilon_4+(\Im{c}+\Im{a}\cdot\Re{b}+(\Re{a}\cdot\Im{b}-\Re{a}\cdot\Im{b}))\varepsilon_3(1+\varepsilon_4),
    \end{aligned}
  \end{displaymath}
  i.e., after regrouping the terms and extracting $\Re{d}$ ($\Im{d}$)
  from the right hand sides,
  \begin{displaymath}
    \begin{aligned}
      \Re{\tilde{d}}-\Re{d}&=\Re{d}(\varepsilon_1+\varepsilon_2+\varepsilon_1\varepsilon_2)-\Re{a}\cdot\Re{b}\cdot\varepsilon_1(1+\varepsilon_2),\\
      \Im{\tilde{d}}-\Im{d}&=\Im{d}(\varepsilon_3+\varepsilon_4+\varepsilon_3\varepsilon_4)-\Re{a}\cdot\Im{b}\cdot\varepsilon_3(1+\varepsilon_4).
    \end{aligned}
  \end{displaymath}

  Note that $|\mathop{\mathrm{fl}}(\tan\varphi)|\le 1$, so an argument
  similar to the one used in the proof of \cref{p:s2} ensures that
  $\varepsilon_7$ and $\varepsilon_8$ can be ignored (considered to be
  zero) when taking $\max\{|\Re{b}|,|\Im{b}|\}$ with
  $|\mathop{\mathrm{fl}}(\tan\varphi)|=1$.  Due to~\cref{e:prop},
  $\max\{|\Re{b}|,|\Im{b}|\}\le\delta_{\alpha}^+$, regardless of a
  possible inaccuracy of $\mathrm{e}'$ indicated by
  \cref{r:ufl,r:hypot}, since even the components of such a
  pathological $\mathrm{e}'$ are at most unity by magnitude.

  Taking the absolute values of the two previous equations, it follows
  \begin{equation}
    \begin{aligned}
      |\Re{\tilde{d}}-\Re{d}|&\le|\Re{d}||\varepsilon_1^{}+\varepsilon_2^{}+\varepsilon_1^{}\varepsilon_2^{}|+|\Re{a}|\delta_{\alpha}^+|\varepsilon_1^{}|(1+\varepsilon_2^{}),\\
      |\Im{\tilde{d}}-\Im{d}|&\le|\Im{d}||\varepsilon_3^{}+\varepsilon_4^{}+\varepsilon_3^{}\varepsilon_4^{}|+|\Re{a}|\delta_{\alpha}^+|\varepsilon_3^{}|(1+\varepsilon_4^{}).
    \end{aligned}
    \label{e:Ra}
  \end{equation}
  Dividing by $|d|$ and applying again the triangle
  inequality simplifies~\cref{e:Ra} to
  \begin{displaymath}
    \begin{aligned}
      |\Re{\tilde{d}}-\Re{d}|/|d|&\le|\varepsilon_1^{}|+|\varepsilon_2^{}|+|\varepsilon_1^{}||\varepsilon_2^{}|+\delta_{\alpha}^+(|\varepsilon_1^{}|+|\varepsilon_1^{}||\varepsilon_2^{}|),\\
      |\Im{\tilde{d}}-\Im{d}|/|d|&\le|\varepsilon_3^{}|+|\varepsilon_4^{}|+|\varepsilon_3^{}||\varepsilon_4^{}|+\delta_{\alpha}^+(|\varepsilon_3^{}|+|\varepsilon_3^{}||\varepsilon_4^{}|),
    \end{aligned}
  \end{displaymath}
  or, denoting by $\mathfrak{C}$ either $\Re$ or $\Im$ and
  bounding each $|\varepsilon_{\imath}|$, $1\le\imath\le 4$, by
  $|\varepsilon|$,
  \begin{equation}
    \frac{|\mathop{\mathfrak{C}}{\tilde{d}}-\mathop{\mathfrak{C}}{d}|}{|d|}\le\tilde{\epsilon},\quad
    3.000000\,\varepsilon<\tilde{\epsilon}<3.000001\,\varepsilon,
    \label{e:zeps}
  \end{equation}
  where the approximate multiples of $\varepsilon$ are computed by the
  script \cref{f:SM1.2} as \texttt{eps}.  From~\cref{e:zeps} magnitude
  of each component of $\tilde{d}$ can be bound relative to $|d|$ as
  \begin{equation}
    \frac{|\mathop{\mathfrak{C}}\tilde{d}|}{|d|}=\frac{|\mathop{\mathfrak{C}}\tilde{d}-\mathop{\mathfrak{C}}d+\mathop{\mathfrak{C}}d|}{|d|}\le\frac{|\mathop{\mathfrak{C}}\tilde{d}-\mathop{\mathfrak{C}}d|}{|d|}+\frac{|\mathop{\mathfrak{C}}d|}{|d|}\le\tilde{\epsilon}+1.
    \label{e:zcomp}
  \end{equation}
  Since $g_{ip}'=d\mathop{\mathrm{fl}}(\cos\varphi)$, it holds
  \begin{displaymath}
    \mathop{\mathrm{fl}}(\Re{g_{ip}'})=\Re{\tilde{d}}\mathop{\mathrm{fl}}(\cos\varphi)(1+\varepsilon_5^{}),\qquad
    \mathop{\mathrm{fl}}(\Im{g_{ip}'})=\Im{\tilde{d}}\mathop{\mathrm{fl}}(\cos\varphi)(1+\varepsilon_6^{}).
  \end{displaymath}
  Thus, from $\mathop{\mathrm{fl}}(\cos\varphi)>0$
  and~\cref{e:zeps,e:zcomp}, it follows
  \begin{displaymath}
    \begin{aligned}
      |\mathop{\mathrm{fl}}(\Re{g_{ip}'})&-\Re{g_{ip}'}|/|g_{ip}'|=|\Re{\tilde{d}}\mathop{\mathrm{fl}}(\cos\varphi)(1+\varepsilon_5^{})-\Re{d}\mathop{\mathrm{fl}}(\cos\varphi)|/(|d|\mathop{\mathrm{fl}}(\cos\varphi))\\
      &\le(|\Re{\tilde{d}}-\Re{d}|/|d|+|\Re{\tilde{d}}\cdot\varepsilon_5^{}|/|d|)\le\tilde{\epsilon}+(1+\tilde{\epsilon})|\varepsilon_5^{}|\le\tilde{\epsilon}+\varepsilon+\tilde{\epsilon}\varepsilon,\\
      |\mathop{\mathrm{fl}}(\Im{g_{ip}'})&-\Im{g_{ip}'}|/|g_{ip}'|=|\Im{\tilde{d}}\mathop{\mathrm{fl}}(\cos\varphi)(1+\varepsilon_6^{})-\Im{d}\mathop{\mathrm{fl}}(\cos\varphi)|/(|d|\mathop{\mathrm{fl}}(\cos\varphi))\\
      &\le(|\Im{\tilde{d}}-\Im{d}|/|d|+|\Im{\tilde{d}}\cdot\varepsilon_5^{}|/|d|)\le\tilde{\epsilon}+(1+\tilde{\epsilon})|\varepsilon_6^{}|\le\tilde{\epsilon}+\varepsilon+\tilde{\epsilon}\varepsilon,
    \end{aligned}
  \end{displaymath}
  or, with
  $\displaystyle\tilde{\epsilon}'=\tilde{\epsilon}+\varepsilon+\tilde{\epsilon}\varepsilon$,
  $\displaystyle\frac{|\mathop{\mathrm{fl}}(\mathop{\mathfrak{C}}{g_{ip}'})-\mathop{\mathfrak{C}}{g_{ip}'}|}{|g_{ip}'|}\le\tilde{\epsilon}'$.
  Finally, with $\tilde{\epsilon}''=\sqrt{2}\tilde{\epsilon}'$,
  \begin{displaymath}
    \begin{gathered}
      |\mathop{\mathrm{fl}}(g_{ip}')-g_{ip}'|/|g_{ip}'|=\sqrt{(|\mathop{\mathrm{fl}}(\Re{g_{ip}'})-\Re{g_{ip}'}|/|g_{ip}'|)^2+(|\mathop{\mathrm{fl}}(\Im{g_{ip}'})-\Im{g_{ip}'}|/|g_{ip}'|)^2}\le\tilde{\epsilon}'',\\
      5.656854\,\varepsilon<\tilde{\epsilon}''<5.656856\,\varepsilon,
    \end{gathered}
  \end{displaymath}
  where the approximate multiples of $\varepsilon$ are computed by the
  script \cref{f:SM1.2} as \texttt{epp}.  A similar proof is valid for
  $|\mathop{\mathrm{fl}}(g_{iq}')-g_{iq}'|/|g_{iq}'|\le\tilde{\epsilon}''$
  when
  $0\ne|g_{iq}'|/\mathop{\mathrm{fl}}(\cos\varphi)\ge|\Re{g_{ip}^{}}|$.

  \looseness=-1
  Take a look what happens with the ``else'' parts of the claims~1
  and~2 of \cref{l:ct}.  If
  $|g_{ip}'|/\mathop{\mathrm{fl}}(\cos\varphi)<|\Re{g_{iq}^{}}|$,
  then $|g_{ip}'|<|g_{iq}^{}|$, and if
  $|g_{iq}'|/\mathop{\mathrm{fl}}(\cos\varphi)<|\Re{g_{ip}^{}}|$, then
  $|g_{iq}'|<|g_{ip}^{}|$.  In exact arithmetic, the max-norm of the
  iteration matrix could not have been affected by either
  transformation.  When computing in floating-point, from
  $|\Re{a}|/|d|>1$ and~\cref{e:Ra} follows that the relative error
  $|\mathop{\mathfrak{C}}{\tilde{d}}-\mathop{\mathfrak{C}}d|/|d|$ is
  no longer bounded above by a constant expression in $\varepsilon$.
  However, if~\cref{e:Ra} is divided by $|a|$ (not by $|d|$), since
  $\max\{|\Re{d}|,|\Im{d}|\}\le|d|<|\Re{a}|\le|a|$, it follows
  \begin{displaymath}
    |\Re{\tilde{d}}-\Re{d}|/|a|<\tilde{\epsilon},\quad|\Im{\tilde{d}}-\Im{d}|/|a|<\tilde{\epsilon},\quad|\tilde{d}-d|/|a|<\sqrt{2}\tilde{\epsilon},
  \end{displaymath}
  and thus $|\tilde{d}-d|<\sqrt{2}|a|\tilde{\epsilon}$, so $\tilde{d}$
  lies in a disk with the center $d$ and the radius
  $\sqrt{2}|a|\tilde{\epsilon}$.  If
  $|a|(1+\sqrt{2}\tilde{\epsilon})\le\nu$, then
  $|\tilde{d}|<\nu$ since $|d|<|a|$.  The final transformed element
  is $\mathop{\mathrm{fl}}(\tilde{d}\mathop{\mathrm{fl}}(\cos\varphi))$,
  with $\mathop{\mathrm{fl}}(\cos\varphi)\le 1$, so
  $|\mathop{\mathrm{fl}}(\tilde{d}\mathop{\mathrm{fl}}(\cos\varphi))|<\nu$
  (due to the monotonicity argument, applied here component-wise),
  what had to be proven.

  Regarding the last statement of the Lemma, note that from
  $\max\{|\Re{b}|,|\Im{b}|\}\le\delta_{\alpha}^+$ and
  $\max\{|\Re{a}|,|\Im{a}|,|\Re{c}|,|\Im{c}|\}\le\nu/4$ follows
  that neither the inner nor the outer real $\mathrm{fma}$ operations
  in~\cref{e:zfma} can overflow, since the magnitude of the result of
  the inner one cannot be greater than
  $\nu(1+\delta_{\alpha}^+)(1+\varepsilon)/4$, so the result of the
  outer one cannot exceed $\nu h/4$ in magnitude, where
  \begin{displaymath}
    h=\mathop{h}(\varepsilon)=(\delta_{\alpha}^++(1+\delta_{\alpha}^+)(1+\varepsilon))(1+\varepsilon).
  \end{displaymath}
  Solving numerically for $\varepsilon$ such that
  $4-\mathop{h}(\varepsilon)=0$ gives two real roots, but only one,
  $\varepsilon_{\mathfrak{2}}$,
  $6.520087\cdot 10^{-2}<\varepsilon_{\mathfrak{2}}^{}<6.520088\cdot 10^{-2}$,
  in $[0,1/4]$.  For all $\varepsilon$ such that
  $0\le\varepsilon\le\varepsilon_{\mathfrak{2}}^{}$, what includes
  $\varepsilon$ of all standard floating-point datatypes, it holds
  $\mathop{h}(\varepsilon)\le 4$, and thus $\nu h/4\le\nu$.
  Having obtained $\tilde{d}$, multiplying it by
  $\mathop{\mathrm{fl}}(\cos\varphi)\le 1$ cannot increase its
  components' magnitudes, and thus no overflow occurs in any rounding.

  Assume that an underflow occurs in~\cref{e:zfma}, when rounding the
  result of an inner or of an outer real $\mathrm{fma}$.  If it is an
  outer $\mathrm{fma}$, such a computed component (be it real or
  imaginary) of $\tilde{d}$ is subnormal.  If it is an inner
  $\mathrm{fma}$, when the affected component of $\tilde{d}$ is
  computed, it cannot be greater than
  $\nu\delta_{\alpha}^+(1+\varepsilon)/4$ in magnitude.  Also, if
  an underflow occurs when forming either component of
  $\mathop{\mathrm{fl}}(\tilde{d}\mathop{\mathrm{fl}}(\cos\varphi))$,
  each component stays below $\nu$ in magnitude.  It may therefore,
  for the purposes of this part of the proof, be assumed in the
  following that no underflow occurs in any rounding
  of~\cref{e:transf} and thus the conditions of the first, already
  proven part of this Lemma are met.

  From~\cref{e:transf}, \cref{r:ufl,r:hypot}, and
  $\max\{|g_{ip}|,|g_{iq}|\}\le\nu/4$ it follows
  \begin{displaymath}
    |g_{ip}'|=|(g_{ip}^{}+g_{iq}^{}\mathrm{e}')\mathop{\mathrm{fl}}(\cos\varphi)|\le|g_{ip}^{}|+\sqrt{2}|g_{iq}^{}|\le\nu(1+\sqrt{2})/4.
  \end{displaymath}
  If $0\ne|d|\ge|\Re{a}|$, then
  $|\mathop{\mathrm{fl}}(g_{ip}')-g_{ip}'|\le|g_{ip}'|\tilde{\epsilon}''$,
  what implies
  \begin{displaymath}
    |\mathop{\mathrm{fl}}(g_{ip}')|\le|g_{ip}'|(1+\tilde{\epsilon}'')\le\nu(1+\sqrt{2})(1+\tilde{\epsilon}'')/4\le\nu,
  \end{displaymath}
  and similarly for $|\mathop{\mathrm{fl}}(g_{iq}')|$, if applicable.
  Else, $|g_{iq}^{}|\le\nu/(1+\sqrt{2}\tilde{\epsilon})$
  since $\sqrt{2}\tilde{\epsilon}\le 3$.  Therefore,
  $|\mathop{\mathrm{fl}}(g_{ip}')|<\nu$ (and similarly for
  $|\mathop{\mathrm{fl}}(g_{iq}')|$, if applicable), what concludes
  the proof.
\end{proof}
\section{Converting the complex values to and back from the split form}\label{s:SM5}
\Cref{a:zsplit} shows how to transform the customary representation of
complex numbers into the split form.  This processing is sequential
(could also be parallel) on each column of $G$, and the columns are
transformed concurrently.  Repacking of the columns of $\Re{X}$ and
$\Im{X}$ back to $X$, for $X\in\{G,V\}$, as in \cref{a:zmerge}, is
performed in parallel over the columns, and sequentially within each
column.

\begin{algorithm}[hbtp]
  \caption{Splitting a vector $\mathsf{ig}$ of four double complex
  numbers from $g_j$, with their real and imaginary components
  customary interleaved, into two vectors of the packed real and
  imaginary parts, respectively, by a lane permutation $\mathsf{ps}$.}
  \label{a:zsplit}
  \begin{algorithmic}[1]
    \STATE{$\mathtt{\_\_m512i}\ \mathsf{ps}=\mathop{\mathtt{\_mm512\_set\_epi64}}(7,5,3,1,6,4,2,0);$}
    \FOR[sequentially, but could be done in parallel]{$i=1$ \TO $m$ \textbf{step} $4$}
    \STATE{$\mathsf{ig}=\mathop{\mathtt{load}}((G)_{ij});$}
    \COMMENT{$\left[(\Re{g_{rj}},\Im{g_{rj}})_{r=i}^{i+3}\right]$, i.e., interleaved}
    \STATE{$\mathsf{pg}=\mathop{\mathtt{permutexvar}}(\mathsf{ps},\mathsf{ig});$}
    \COMMENT{$\left[(\Re{g_{rj}})_{r=i}^{i+3},(\Im{g_{rj}})_{r=i}^{i+3}\right]$, i.e., split}
    \STATE{$\mathop{\mathtt{\_mm256\_store\_pd}}((\Re{G})_{ij},\mathop{\mathtt{extractf64x4}}(\mathsf{pg},0));$}
    \COMMENT{$\left[(\Re{g_{rj}})_{r=i}^{i+3}\right]$}
    \STATE{$\mathop{\mathtt{\_mm256\_store\_pd}}((\Im{G})_{ij},\mathop{\mathtt{extractf64x4}}(\mathsf{pg},1));$}
    \COMMENT{$\left[(\Im{g_{rj}})_{r=i}^{i+3}\right]$}
    \ENDFOR\COMMENT{output: $\Re{g_j}$ and $\Im{g_j}$}
  \end{algorithmic}
\end{algorithm}

\begin{algorithm}[hbtp]
  \caption{Merging, by a lane permutation $\mathsf{pm}$, of two
    vectors, $\mathsf{r}$ from $\Re{X}$ and $\mathsf{i}$ from
    $\Im{X}$, of the real and the imaginary parts, respectively, of
    four double complex numbers into a vector from $X$ with the
    numbers' parts customary interleaved.}
  \label{a:zmerge}
  \begin{algorithmic}[1]
    \STATE{$\mathtt{\_\_m512i}\ \mathsf{pm}=\mathop{\mathtt{\_mm512\_set\_epi64}}(7,3,6,2,5,1,4,0);$}
    \FOR[sequentially, but could be done in parallel]{$i=1$ \TO $\mathop{\mathrm{\#rows}}(X)$ \textbf{step} $4$}
    \STATE{$\mathtt{\_\_m256d}\ \mathsf{r}=\mathop{\mathtt{load}}((\Re{X})_{ij}),\ \mathsf{i}=\mathop{\mathtt{load}}((\Im{X})_{ij});$}
    \COMMENT{$\left[(\Re{x_{rj}})_{r=i}^{i+3}\right],\ \left[(\Im{x_{rj}})_{r=i}^{i+3}\right]$}
    \STATE{$\mathsf{ri}=\mathop{\mathtt{\_mm512\_insertf64x4}}(\mathop{\mathtt{\_mm512\_zextpd256\_pd512}}(\mathsf{r}),\mathsf{i},1);$}
    \COMMENT{$\left[\mathsf{r},\mathsf{i}\right]$}
    \STATE{$\mathsf{il}=\mathop{\mathtt{permutexvar}}(\mathsf{pm},\mathsf{ri});$}
    \COMMENT{$\left[(\Re{x_{rj}},\Im{x_{rj}})_{r=i}^{i+3}\right]$}
    \STATE{$\mathop{\mathtt{store}}((X)_{ij},\mathsf{il});$}
    \COMMENT{store the values with their parts interleaved}
    \ENDFOR\COMMENT{output: $x_j$}
  \end{algorithmic}
\end{algorithm}
\section{Vectorized non-overflowing computation of the Frobenius norm}\label{s:SM6}
Here, a method is proposed in which the exponent range of all partial
sums of the squares of the input array's elements, as well as of the
final result, is sufficiently widened to avoid obtaining an infinite
value for any expected array length, but the number of significant
digits is unaltered from that of the input's datatype
(\texttt{double}).

The method's operation is conceptually equivalent to that of a
vectorized $\mathbf{x}^T\mathbf{x}$ dot-product, shown in
\cref{a:vdp}, and thus their outputs are generally identical when the
number of lanes $\mathtt{s}$ is same for both, except in the cases of
overflowing (or extreme underflowing to zero) of the results of the
\cref{a:vdp}, while the proposed method returns a finite (or non-zero)
representation, respectively, by design.

\begin{algorithm}[hbtp]
  \caption{A vectorized implementation of a dot-product $\mathbf{x}^T\mathbf{x}$ using $\mathrm{fma}$.}
  \label{a:vdp}
  \begin{algorithmic}[1]
    \STATE{$\mathsf{ssq}=\mathop{\mathtt{setzero}}();$}
    \COMMENT{initially, the partial sums of squares is a vector of zeroes}
    \FOR[update $\mathsf{ssq}$]{$i=1$ \TO $m$ \textbf{step} $\mathtt{s}$}
    \STATE{$\mathsf{x}=\mathop{\mathtt{load}}(x_i);$}
    \COMMENT{$\displaystyle\mathsf{x}=\begin{bmatrix}x_i\!&\!x_{i+1}\!&\!\cdots\!&x_{i+\mathtt{s}-1}\end{bmatrix}$}
    \STATE{$\mathsf{ssq}=\mathop{\mathtt{fmadd}}(\mathsf{x},\mathsf{x},\mathsf{ssq});$}
    \COMMENT{$\mathsf{ssq}=\mathsf{x}^T\mathsf{x}+\mathsf{ssq}$}
    \ENDFOR\COMMENT{output: $\mathsf{ssq}$}
    \STATE{sum-reduce $\mathsf{ssq}$ and return the result;}
    \COMMENT{optionally, pre-sort $\mathsf{ssq}$}
  \end{algorithmic}
\end{algorithm}

The final $\mathsf{ssq}$ could be sorted by the algorithm
from~\cite[Appendix~B]{Bramas-17}, and the sum-reduction order in
\cref{a:vdp} could be taken from \cref{a:rnF} for comparison with the
results of the latter, or the reduction could proceed sequentially.
\subsection{Input, output, constraints, and a data representation}\label{ss:SM6.1}
Let $\mathbf{x}$ be an input array, with all its elements finite.  A
partial sum of their squares $r$ (including the resulting
$\|\mathbf{x}\|_F^2$) is represented as $(e,f)$, where $1\le f<2$ is
the ``fractional'' part of $r$, while
$e=\left\lfloor\lg r\right\rfloor$ is the exponent of its power-of-two
scaling factor, such that $r=2^e f$.  Both $e$ and $f$ are
floating-point quantities, and $0=(-\infty,1)$.  For $r>0$, $e$ is a
finite integral value.  For example, $12=2^3\cdot 1.5=(3,1.5)$ and
$80=2^6\cdot 1.25=(6,1.25)$.

To get
$\displaystyle\|\mathbf{x}\|_F^{}=\sqrt{\|\mathbf{x}\|_F^2}=\sqrt{(e,f)}$
when $e$ is finite and odd, take $f'=2f$, $e'=e-1$, and compute
$\|\mathbf{x}\|_F^{}=(e'/2,\mathop{\mathrm{fl}}(\sqrt{f'}))$,
since
$\mathop{\mathrm{fl}}(\sqrt{2})\le\mathop{\mathrm{fl}}(\sqrt{f'})<2$
and $e'$ is even.  For example,
$\sqrt{12}=\sqrt{(3,1.5)}=((3-1)/2,\mathop{\mathrm{fl}}(\sqrt{2\cdot 1.5}))=(1,\mathop{\mathrm{fl}}(\sqrt{3}))$.
If $e$ is infinite or even, set $f'=f$ and $e'=e$; e.g.,
$\sqrt{80}=\sqrt{(6,1.25)}=(3,\mathop{\mathrm{fl}}(\sqrt{1.25}))$.
Let $e''=e'/2$ and $f''=\mathop{\mathrm{fl}}(\sqrt{f'})$.  The exact
value of $2^{e''}\!f''$ could be greater than $\nu$, but
$\|\mathbf{x}\|_F^{}=(e'',f'')$, for $\mathbf{x}\ne\mathbf{0}$,
remains representable by two finite quantities of the input's
datatype.

\looseness=-1
For example, let
$\mathbf{x}=\nu\mathbf{1}=\begin{bmatrix}\nu&\cdots&\nu\end{bmatrix}^T$
of length $m\ge 2$.  Then,
$\|\mathbf{x}\|_F^{}=\nu\sqrt{m}$ (mathematically), so
$\mathop{\mathrm{fl}}(\|\mathbf{x}\|_F^{})=\infty$, but $e''$ is
finite (as well as $f''$) and constrained only by the requirement that
all operations with finite exponents as arguments have to produce the
exact result, as if performed in integer arithmetic, without any
rounding.  All finite exponents throughout the computation should thus
be at most $\check{\nu}$ in magnitude, where, in single precision,
$\check{\nu}=2^{24}$, and in double
precision\footnote{$\check{\nu}=2^{31}-1$ in double precision with
the AVX512F instruction subset only}, $\check{\nu}=2^{53}$.  For
every partial sum $r$ it has to hold $r<2^{\check{\nu}+1}$, a limit
that could hardly ever be reached.
\subsection{Vectorized iterative computation of the partial sums}\label{ss:SM6.2}
Given $\mathbf{x}$ of length $m\ge 1$ (zero-padded to $\tilde{m}$ as
in~\cref{e:pad0}), let
$|\mathsf{x}_{\mathtt{i}}|=(\mathop{\mathtt{E}}(\mathsf{x}_{\mathtt{i}}),\mathop{\mathtt{F}}(\mathsf{x}_{\mathtt{i}}))=(\mathsf{e}_{\mathtt{i}},\mathsf{f}_{\mathtt{i}})$,
where $\mathtt{i}\ge 0$,
$\mathop{\mathtt{E}}(\mathsf{x})=\mathop{\mathtt{getexp}}(\mathsf{x})$,
$\mathop{\mathtt{F}}(\mathsf{x})$ is given by~\cref{e:F}, and
$\mathsf{x}_{\mathtt{i}}=\begin{bmatrix}x_i\!&\!x_{i+1}\!&\!\cdots\!&\!x_{i+\mathtt{s}-1}\end{bmatrix}^T$
is a vector with $\mathtt{s}$ lanes, loaded from a contiguous subarray
of $\mathbf{x}$ with the one-based indices $i$ such that
$\mathtt{i}\mathtt{s}+1\le i\le(\mathtt{i}+1)\mathtt{s}$.

Let, for $\mathtt{i}\ge 0$,
$\mathsf{r}_{\mathtt{i}}=(\mathsf{e}^{(\mathtt{i})},\mathsf{f}^{(\mathtt{i})})$,
be a pair of vectors representing the
$\mathtt{i}$th partial sums, and let $\mathsf{r}_{\mathtt{0}}$
represent $\mathtt{s}$ zeros as
$(\mathsf{e}^{(\mathtt{0})},\mathsf{f}^{(\mathtt{0})})$.
Conceptually, the following operation has to be defined,
\begin{displaymath}
  \mathsf{r}_{\mathtt{i}+1}=\mathop{\mathrm{fma}}(|\mathsf{x}_{\mathtt{i}}|,|\mathsf{x}_{\mathtt{i}}|,\mathsf{r}_{\mathtt{i}}),
\end{displaymath}
i.e.,
$(\mathsf{e}^{(\mathtt{i}+1)},\mathsf{f}^{(\mathtt{i}+1)})=\mathop{\mathrm{fma}}((\mathsf{e}_{\mathtt{i}},\mathsf{f}_{\mathtt{i}}),(\mathsf{e}_{\mathtt{i}},\mathsf{f}_{\mathtt{i}}),(\mathsf{e}^{(\mathtt{i})},\mathsf{f}^{(\mathtt{i})}))$,
where such a function is implemented in terms of SIMD instructions as
follows.

First, define
$\mathsf{r}_{\mathtt{i}}'=(\mathsf{e}^{\prime(\mathtt{i})},\mathsf{f}^{\prime(\mathtt{i})})$
as a ``non-normalized'' representation of the same values as
$\mathsf{r}_{\mathtt{i}}^{}$, but with each exponent even or
infinite.  To obtain $\mathsf{r}_{\mathtt{i}}'$ from
$\mathsf{r}_{\mathtt{i}}^{}$, $\mathsf{e}^{(\mathtt{i})}$ has to be
converted to a vector $\check{\mathsf{e}}^{(\mathtt{i})}$ of signed
integers and then their least significant bits have to be extracted.
In the AVX512F instruction subset, only a conversion from
floating-point values to 32-bit integers is natively
supported\footnote{With AVX512DQ instructions a double precision
$\mathsf{e}^{(\mathtt{i})}$ can be converted to 64-bit integers.}, so
the integral exponents are obtained as
$\check{\mathsf{e}}_{\mathrm{all}}^{(\mathtt{i})}=\mathop{\mathtt{\_mm512\_cvtpX\_epi32}}(\mathsf{e}^{(\mathtt{i})})$,
$\mathtt{X}\in\{\mathtt{s},\mathtt{d}\}$, where $-\infty$ gets
converted to $\mathtt{INT\_MIN}=-2^{31}$.  The least significant bit
of each exponent is extracted by the bitwise-and operation as
$\check{\mathsf{e}}_{\mathrm{lsb}}^{(\mathtt{i})}=\mathop{\mathtt{\_mm256\_and\_si256}}(\check{\mathsf{e}}_{\mathrm{all}}^{(\mathtt{i})},\mathop{\mathtt{\_mm256\_set1\_epi32}}(1))$,
where a lane of the result is one if and only if the corresponding
exponent is odd, and zero otherwise.  Then,
$\check{\mathsf{e}}_{\mathrm{lsb}}^{(\mathtt{i})}$ is converted to a
floating-point vector
$\hat{\mathsf{e}}^{(\mathtt{i})}=\mathop{\mathtt{cvtepi32}}(\check{\mathsf{e}}_{\mathrm{lsb}}^{(\mathtt{i})})$.
It can now be defined that
$\mathsf{e}^{\prime(\mathtt{i})}=\mathop{\mathtt{sub}}(\mathsf{e}^{(\mathtt{i})},\hat{\mathsf{e}}^{(\mathtt{i})})$
and
$\mathsf{f}^{\prime(\mathtt{i})}=\mathop{\mathtt{scalef}}(\mathsf{f}^{(\mathtt{i})},\hat{\mathsf{e}}^{(\mathtt{i})})$.
This process of computing
$(\mathsf{e}^{\prime(\mathtt{i})},\mathsf{f}^{\prime(\mathtt{i})})$
from $(\mathsf{e}^{(\mathtt{i})},\mathsf{f}^{(\mathtt{i})})$ is a
vectorization of the scalar computation of $(e',f')$ from $(e,f)$,
given in \cref{ss:SM6.1}, but with a different purpose.

With $\mathsf{1}$ being a vector of all ones, and
$\hat{\mathsf{e}}_{\mathtt{i}}=\mathop{\mathtt{scalef}}(\mathsf{e}_{\mathtt{i}},\mathsf{1})$,
observe that the ``normalized'' exponents of
$|\mathsf{x}_{\mathtt{i}}|^{\mathfrak{2}}$, lane-wise, should be
either those from $\hat{\mathsf{e}}_{\mathtt{i}}$ (all even or
infinite) or greater by one, since
$\mathsf{1}\le\mathsf{f}_{\mathtt{i}}^{\mathfrak{2}}<\mathsf{4}$.
Therefore, a provisory, non-normalized exponent vector of
$|\mathsf{x}_{\mathtt{i}}^{}|^{\mathfrak{2}}+\mathsf{r}_{\mathtt{i}}'$
can be taken as
$\mathsf{e}_{\max}=\mathop{\mathtt{max}}(\hat{\mathsf{e}}_{\mathtt{i}},\mathsf{e}^{\prime(\mathtt{i})})$,
in which all exponents are even or infinite since the same holds for
both $\hat{\mathsf{e}}_{\mathtt{i}}$ and
$\mathsf{e}^{\prime(\mathtt{i})}$ by design.

In general, $\hat{\mathsf{e}}_{\mathtt{i}}^{}$ and
$\mathsf{e}^{\prime(\mathtt{i})}$ are not equal.  To compute
$|\mathsf{x}_{\mathtt{i}}^{}|^{\mathfrak{2}}+\mathsf{r}_{\mathtt{i}}'$,
the exponent of each scalar addend has to be equalized such that the
fractional part of the addend having the smaller exponent
$e_{\min}^{}$ is scaled by two to the power of the difference of
$e_{\min}^{}$ and the larger exponent $e_{\max}^{}$.  Having the
exponents equalized to $\mathsf{e}_{\max}$, the computation can
proceed with the scaled fractional parts and the fused multiply-add
operation.

Let
$\hat{\mathsf{e}}_{\mathtt{i}}'=\mathop{\mathtt{max}}(\mathop{\mathtt{sub}}(\hat{\mathsf{e}}_{\mathtt{i}}^{},\mathsf{e}_{\max}^{}),-\mathsf{\infty})$,
with the maximum taken to filter out a possible $\mathtt{NaN}$ as the
result of $-\infty-(-\infty)$, when instead $-\infty$ is desired, and
similarly let
$\mathsf{e}^{\prime\prime(\mathtt{i})}=\mathop{\mathtt{max}}(\mathop{\mathtt{sub}}(\mathsf{e}^{\prime(\mathtt{i})},\mathsf{e}_{\max}^{}),-\mathsf{\infty})$.
Since all lanes of $\hat{\mathsf{e}}_{\mathtt{i}}'$ are even or
infinite, they can be divided by two as
$\hat{\mathsf{e}}_{\mathtt{i}}''=\mathop{\mathtt{scalef}}(\hat{\mathsf{e}}_{\mathtt{i}}',-\mathsf{1})$.
The scaling of $\mathsf{f}_{\mathtt{i}}^{\mathfrak{2}}$ by
$\mathsf{2}^{\hat{\mathsf{e}}_{\mathtt{i}}'}$ is mathematically
equivalent to the scaling of $\mathsf{f}_{\mathtt{i}}^{}$ by
$\mathsf{2}^{\hat{\mathsf{e}}_{\mathtt{i}}''}$ before squaring the
result, but the latter does not require explicit computation of
$\mathsf{f}_{\mathtt{i}}^{\mathfrak{2}}\mathsf{2}^{\hat{\mathsf{e}}_{\mathtt{i}}'}$
and thus has a lower potential for underflow, since the squaring
happens as a part of an $\mathrm{fma}$ computation, and such an
intermediate result is not rounded.  Let
$\hat{\mathsf{f}}_{\mathtt{i}}^{}=\mathop{\mathtt{scalef}}(\mathsf{f}_{\mathtt{i}}^{},\hat{\mathsf{e}}_{\mathtt{i}}'')$,
$\hat{\mathsf{f}}^{\prime(\mathtt{i})}=\mathop{\mathtt{scalef}}(\mathsf{f}^{\prime(\mathtt{i})},\mathsf{e}^{\prime\prime(\mathtt{i})})$,
and define
\begin{displaymath}
  \mathsf{r}_{\mathtt{i}+1}'=(\mathsf{e}^{\prime(\mathtt{i}+1)},\mathsf{f}^{\prime(\mathtt{i}+1)})=(\mathsf{e}_{\max}^{},\mathop{\mathtt{fmadd}}(\hat{\mathsf{f}}_{\mathtt{i}}^{},\hat{\mathsf{f}}_{\mathtt{i}}^{},\hat{\mathsf{f}}^{\prime(\mathtt{i})}))
\end{displaymath}
as a non-normalized representation of the value of $\mathsf{r}_{\mathtt{i}+1}^{}$.
To get its normalized form, let
$\mathsf{f}^{(\mathtt{i}+1)}=\mathop{\mathtt{F}}(\mathsf{f}^{\prime(\mathtt{i}+1)})$,
$\mathsf{e}^{(\mathtt{i}+1)}=\mathop{\mathtt{add}}(\mathsf{e}^{\prime(\mathtt{i}+1)},\mathop{\mathtt{E}}(\mathsf{f}^{\prime(\mathtt{i}+1)}))$,
and
$\mathsf{r}_{\mathtt{i}+1}^{}=(\mathsf{e}^{(\mathtt{i}+1)},\mathsf{f}^{(\mathtt{i}+1)})$.

In the case of a complex input in the form~\cref{e:layout1}, in
each iteration the first $\mathsf{x}_{\mathtt{i}}$ is taken from
$\Re{\tilde{\mathbf{z}}}$, and the second one from
$\Im{\tilde{\mathbf{z}}}$.  After $\mathtt{m}=\tilde{m}/\mathtt{s}$
iterations, this main part of the method terminates, with the final
partial sums $\mathsf{r}_{\mathtt{m}}$ in the normalized form.

\paragraph{Prefetching $\mathsf{x}_{\mathtt{i}}$}
Before the main loop, $\mathsf{x}_{\mathtt{0}}$ is prefetched to the
L1 cache, as well as each $\mathsf{x}_{\mathtt{i}+1}$ at the start of
the $\mathtt{i}$th iteration, using
$\mathop{\mathtt{\_mm\_prefetch}}(\cdots,\mathtt{\_MM\_HINT\_T0})$.
\subsection{Horizontal reduction of the final partial sums}\label{ss:SM6.3}
Now $\mathsf{r}_{\mathtt{m}}$ has to be sum-reduced horizontally.
If $l$ indexes the SIMD lanes of $\mathsf{e}^{(\mathtt{m})}$ and
$\mathsf{f}^{(\mathtt{m})}$, then
\begin{equation}
  \|\mathbf{x}\|_F^2=\sum_{l=1}^{\mathtt{s}}(\mathsf{e}_l^{(\mathtt{m})},\mathsf{f}_l^{(\mathtt{m})})=\sum_{l=1}^{\mathtt{s}}(e_l^{},f_l^{})=(e,f),
  \label{e:nred}
\end{equation}
\subsubsection{Vectorized sorting of the final partial sums}\label{sss:SM6.3.1}
Accuracy of the result $(e,f)$ might be improved by sorting the
$(e_l,f_l)$ pairs lexicographically (i.e., non-decreasingly by the
values they represent) before the summation, such that
\begin{equation}
  (e_i,f_i)\le(e_j,f_j)\iff(e_i<e_j)\vee((e_i=e_j)\wedge(f_i\le f_j)).
  \label{e:vsort}
\end{equation}

If~\cref{e:nred} is evaluated such that the values closest by
magnitude are added together, there is less chance that a relatively
small value gets ignored, i.e., does not affect the result.  A
vectorized method that establishes the ordering~\cref{e:vsort} was
designed as an extension of the AVX512F-vectorized~\cite{Bramas-17}
Batcher's bitonic sort~\cite{Batcher-68} of a double precision vector
to the pair $(\mathsf{e}^{(\mathtt{m})},\mathsf{f}^{(\mathtt{m})})$ of
vectors with the comparison operator of the $i$th and the $j$th lane
given by~\cref{e:vsort}.

In each sorting stage $\ell$, $1\le\ell\le 6$, the code from
\cref{a:bitonic} is executed, where
$\mathsf{e}_{[1]}^{}=\mathsf{e}^{(\mathtt{m})}$ and
$\mathsf{f}_{[1]}^{}=\mathsf{f}^{(\mathtt{m})}$, respectively,
$\mathsf{p}_{\ell}^{}$ is a permutation vector, and
$\mathfrak{b}_{\ell}^{}$ is a bitmask, as given in \cref{t:bitonic}
and defined in~\cite{Bramas-17}.   The vectors
$\mathsf{e}_{[\ell]}^{}$ and $\mathsf{f}_{[\ell]}^{}$, and their
permutations $\mathsf{e}_{[\ell]}'$ and $\mathsf{f}_{[\ell]}'$, are
compared according to~\cref{e:vsort}, resulting in a bitmask
$\mathfrak{m}_{[\ell]}^{\le}$, using which the lane-wise minimums
($\mathsf{e}_{[\ell]}^{\le}$ and $\mathsf{f}_{[\ell]}^{\le}$) and the
maximums ($\mathsf{e}_{[\ell]}^>$ and $\mathsf{f}_{[\ell]}^>$) are
extracted.  Then, the inter-lane exchanges of values in
$\mathsf{e}_{[\ell]}^{}$ as well as in $\mathsf{f}_{[\ell]}^{}$,
according to $\mathfrak{b}_{\ell}^{}$, form the new sequences
$\mathsf{e}_{[\ell+1]}^{}$ and $\mathsf{f}_{[\ell+1]}^{}$ for the next
stage.

\begin{algorithm}[hbtp]
  \caption{The AVX512F compare-and-exchange operation
    from~\cite[Appendix~B]{Bramas-17} extended to a pair
    $(\mathsf{e}_{[\ell]},\mathsf{f}_{[\ell]})$ of double precision
    vectors according to~\cref{e:vsort}.}
  \label{a:bitonic}
  \begin{algorithmic}[1]
    \FOR[with $\mathtt{s}=8$ double precision lanes per vector]{$\ell=1$ \TO $6$}
    \STATE{$\mathsf{e}_{[\ell]}'=\mathop{\mathtt{permutexvar}}(\mathsf{p}_{\ell}^{},\mathsf{e}_{[\ell]}^{});$}
    \COMMENT{permute $\mathsf{e}_{[\ell]}^{}$ w.r.t.\ $\mathsf{p}_{\ell}^{}$}
    \STATE{$\mathsf{f}_{[\ell]}'=\mathop{\mathtt{permutexvar}}(\mathsf{p}_{\ell}^{},\mathsf{f}_{[\ell]}^{});$}
    \COMMENT{permute $\mathsf{f}_{[\ell]}^{}$ w.r.t.\ $\mathsf{p}_{\ell}^{}$}
    \STATE{$\mathfrak{m}_{\ell}^==\mathop{\mathtt{\_mm512\_cmpeq\_pd\_mask}}(\mathsf{e}_{[\ell]}^{},\mathsf{e}_{[\ell]}');$}
    \COMMENT{where is $\mathsf{e}_{[\ell]}^{}=\mathsf{e}_{[\ell]}'$ \ldots}
    \STATE{$\mathfrak{m}_{\ell}^{\mathsf{f}}=\mathop{\mathtt{\_mm512\_mask\_cmple\_pd\_mask}}(\mathfrak{m}_{\ell}^=,\mathsf{f}_{[\ell]}^{},\mathsf{f}_{[\ell]}');$}
    \COMMENT{\ldots\ and $\mathsf{f}_{[\ell]}^{}\le\mathsf{f}_{[\ell]}'$}
    \STATE{$\mathfrak{m}_{\ell}^{\mathsf{e}}=\mathop{\mathtt{\_mm512\_cmplt\_pd\_mask}}(\mathsf{e}_{[\ell]}^{},\mathsf{e}_{[\ell]}');$}
    \COMMENT{where is $\mathsf{e}_{[\ell]}^{}<\mathsf{e}_{[\ell]}'$}
    \STATE{$\mathfrak{m}_{\ell}^{\le}=(\mathtt{\_\_mmask8})\mathop{\mathtt{\_kor\_mask16}}(\mathfrak{m}_{\ell}^{\mathsf{e}},\mathfrak{m}_{\ell}^{\mathsf{f}});$}
    \COMMENT{bitmask of~\cref{e:vsort}}
    \STATE{$\mathsf{e}_{[\ell]}^{\le}=\mathop{\mathtt{mask\_blend}}(\mathfrak{m}_{\ell}^{\le},\mathsf{e}_{[\ell]}',\mathsf{e}_{[\ell]}^{});$}
    \COMMENT{$e$-s of smaller values}
    \STATE{$\mathsf{e}_{[\ell]}^>=\mathop{\mathtt{mask\_blend}}(\mathfrak{m}_{\ell}^{\le},\mathsf{e}_{[\ell]}^{},\mathsf{e}_{[\ell]}');$}
    \COMMENT{$e$-s of larger values}
    \STATE{$\mathsf{f}_{[\ell]}^{\le}=\mathop{\mathtt{mask\_blend}}(\mathfrak{m}_{\ell}^{\le},\mathsf{f}_{[\ell]}',\mathsf{f}_{[\ell]}^{});$}
    \COMMENT{$f$-s of smaller values}
    \STATE{$\mathsf{f}_{[\ell]}^>=\mathop{\mathtt{mask\_blend}}(\mathfrak{m}_{\ell}^{\le},\mathsf{f}_{[\ell]}^{},\mathsf{f}_{[\ell]}');$}
    \COMMENT{$f$-s of larger values}
    \STATE{$\mathsf{e}_{[\ell+1]}^{}=\mathop{\mathtt{mask\_mov}}(\mathsf{e}_{[\ell]}^{\le},\mathfrak{b}_{\ell}^{},\mathsf{e}_{[\ell]}^>);$}
    \COMMENT{lane exchanges $\mathsf{e}_{[\ell]}^{}\to\mathsf{e}_{[\ell+1]}^{}$}
    \STATE{$\mathsf{f}_{[\ell+1]}^{}=\mathop{\mathtt{mask\_mov}}(\mathsf{f}_{[\ell]}^{\le},\mathfrak{b}_{\ell}^{},\mathsf{f}_{[\ell]}^>);$}
    \COMMENT{lane exchanges $\mathsf{f}_{[\ell]}^{}\to\mathsf{f}_{[\ell+1]}^{}$}
    \ENDFOR\COMMENT{output: $(\mathsf{e}_{[7]}^{},\mathsf{f}_{[7]}^{})$}
  \end{algorithmic}
\end{algorithm}

\begin{table}[hbtp]
\begin{displaymath}
  \begin{array}{ccc|ccc}
    \ell&\mathsf{p}_{\ell}\text{ (high-to-low)}&\mathfrak{b}_{\ell}&\ell&\mathsf{p}_{\ell}\text{ (high-to-low)}&\mathfrak{b}_{\ell}\\\hline
    1&[6,7,4,5,2,3,0,1]&(\mathtt{AA})_{16}&4&[0,1,2,3,4,5,6,7]&(\mathtt{F0})_{16}\\
    2&[4,5,6,7,0,1,2,3]&(\mathtt{CC})_{16}&5&[5,4,7,6,1,0,3,2]&(\mathtt{CC})_{16}\\
    3&[6,7,4,5,2,3,0,1]&(\mathtt{AA})_{16}&6&[6,7,4,5,2,3,0,1]&(\mathtt{AA})_{16}
  \end{array}
\end{displaymath}
\caption{The values of the permutation index vector
  $\mathsf{p}_{\ell}$ (from the highest to the lowest lane, i.e., in
  the order of the arguments of
  $\mathop{\mathtt{\_mm512\_set\_epi64}}$) and of the bitmask
  $\mathfrak{b}_{\ell}$ for each stage $\ell$ of the bitonic sort from
  \cref{a:bitonic}, as defined in~\cite{Bramas-17}.}
\label{t:bitonic}
\end{table}

After the sorting has taken place,
$(\mathsf{e}^{(\mathtt{m})},\mathsf{f}^{(\mathtt{m})})$ is redefined
as $(\mathsf{e}_{[7]}^{},\mathsf{f}_{[7]}^{})$, since the represented
values have not changed, but are now in a possibly different order.
In single precision, and generally with a different power-of-two
number of lanes $\mathtt{s}$, \cref{a:bitonic} and \cref{t:bitonic}
have to be reimplemented for the new sorting network.

\paragraph{Sorting $\mathsf{x}_{\mathtt{i}}$}
In \cref{a:rnF}, after $\mathsf{x}_{\mathtt{i}}$ had been loaded, its
absolute values could have been taken and sorted as a double precision
vector by the procedure from~\cite[Appendix~B]{Bramas-17}, before
forming $|\mathsf{x}_{\mathtt{i}}|$.  This way the post-iteration
sorting could have been rendered redundant, at the expense of more
work in the main loop.
\subsubsection{Vectorized pairwise reduction of the final partial sums}\label{sss:SM6.3.2}
For reproducibility of \cref{a:rnF} to depend on $\mathtt{s}$ only,
the summation~\cref{e:nred} should always be performed in the same
fashion.  Let $(e_i^{[0]},f_i^{[0]})=(e_i^{},f_i^{})$, and define
\begin{equation}
  (e_i^{[j]},f_i^{[j]})=(e_{2i-1}^{[j-1]},f_{2i-1}^{[j-1]})+(e_{2i}^{[j-1]},f_{2i}^{[j-1]})
  \label{e:1red}
\end{equation}
for $1\le j\le\lg\mathtt{s}$ and $1\le i\le\mathtt{s}/2^j$.  Then,
$(e,f)=(e_1^{[\lg\mathtt{s}]},f_1^{[\lg\mathtt{s}]})$.

Assuming that $(\mathsf{e}^{(\mathtt{m})},\mathsf{f}^{(\mathtt{m})})$
have been sorted as in~\cref{e:vsort}, for all $j\ge 0$ and $i$, where
$1\le i\le\mathtt{s}/2^j$, it holds
$(e_i^{[j]},f_i^{[j]})\le(e_{i+1}^{[j]},f_{i+1}^{[j]})$, what can be
proven by induction on $j$.\\For $j=0$, the statement is a direct
consequence of the sorting, and for $j>0$, it follows
from~\cref{e:1red}.  Consequently, $e_i^{[j]}\le e_{i+1}^{[j]}$ for
all $j$ and odd $i$, what simplifies the addition of the adjacent
pairs from~\cref{e:1red} to
\begin{equation}
  (e_{2i-1}^{[j-1]},f_{2i-1}^{[j-1]})+(e_{2i}^{[j-1]},f_{2i}^{[j-1]})=(e_{2i}^{[j-1]},2^{e_{2i-1}^{[j-1]}-e_{2i}^{[j-1]}}f_{2i-1}^{[j-1]}+f_{2i}^{[j-1]}),
  \label{e:redj}
\end{equation}
followed by the normalization of the result
$(e_i^{\prime[j]},f_i^{\prime[j]})$.  The reduction~\cref{e:nred}
with this addition operator is vectorized as follows.

Define $\mathsf{c}^{[0]}=\mathsf{c}^{(\mathtt{m})}$, where
$\mathsf{c}\in\{\mathsf{e},\mathsf{f}\}$.  Let
$\mathfrak{m}^{[0]}=(01010101)_2$ be a lane bit-mask with every other
bit set, $\tilde{\mathfrak{m}}^{[0]}=(10101010)_2$ its bitwise
complement, and $\mathfrak{m}^{[j]}$ and $\tilde{\mathfrak{m}}^{[j]}$
bit-masks with the lowest $\mathtt{s}/2^j$ bits taken from
$\mathfrak{m}^{[0]}$ and $\tilde{\mathfrak{m}}^{[0]}$, respectively,
while the higher bits are zero.  Extend the masks analogously if
$\mathtt{s}=16$.

For each $j\ge 1$, extract the odd-indexed and the even-indexed parts
of $\mathsf{c}^{[j-1]}$, as in the right hand side of~\cref{e:1red},
into the $\mathtt{s}/2^j$ contiguous lower lanes of
$\mathsf{a}_{\mathsf{c}}^{[j-1]}$ and
$\mathsf{b}_{\mathsf{c}}^{[j-1]}$, respectively, where
\begin{align*}
  \mathsf{a}_{\mathsf{c}}^{[j-1]}&=\mathop{\mathtt{maskz\_compress}}(\mathfrak{m}^{[j-1]},\mathsf{c}^{[j-1]})=(0,\ldots,0,(c_{2i-1}^{[j-1]})_i),\\
  \mathsf{b}_{\mathsf{c}}^{[j-1]}&=\mathop{\mathtt{maskz\_compress}}(\tilde{\mathfrak{m}}^{[j-1]},\mathsf{c}^{[j-1]})=(0,\ldots,0,(c_{2i}^{[j-1]})_i),
\end{align*}
and $c\in\{e,f\}$.  Form the non-normalized $\mathsf{e}^{\prime[j]}$
as $\mathsf{b}_{\mathsf{e}}^{[j-1]}$ and $\mathsf{f}^{\prime[j]}$,
from~\cref{e:redj}, as
\begin{displaymath}
  \mathop{\mathtt{fmadd}}(\mathop{\mathtt{scalef}}(\mathsf{1},\mathop{\mathtt{max}}(\mathop{\mathtt{sub}}(\mathsf{a}_{\mathsf{e}}^{[j-1]},\mathsf{b}_{\mathsf{e}}^{[j-1]}),-\mathsf{\infty}),\mathsf{a}_{\mathsf{f}}^{[j-1]},\mathsf{b}_{\mathsf{f}}^{[j-1]}).
\end{displaymath}
Normalize the result as
$\mathsf{f}^{[j]}=\mathop{\mathtt{F}}(\mathsf{f}^{\prime[j]})$ and
$\mathsf{e}^{[j]}=\mathop{\mathtt{add}}(\mathsf{e}^{\prime[j]},\mathop{\mathtt{E}}(\mathsf{f}^{\prime[j]}))$.

\looseness=-1
After $\lg\mathtt{s}$ reduction stages the lowest lanes of
$\mathsf{e}^{[\lg\mathtt{s}]}$ and $\mathsf{f}^{[\lg\mathsf{s}]}$
contain $e$ and $f$, respectively.  Finally, $\|\mathbf{x}\|_F^{}$ is
computed from $\|\mathbf{x}\|_F^2$ using scalar operations, as
described in \cref{ss:SM6.1}.  The Frobenius-norm method is
summarized in \cref{a:rnF}.

\begin{algorithm}[hbtp]
  \caption{Vectorized Frobenius norm computation of a one-dimensional
    double precision array, assuming $\mathtt{s}=8$ lanes per vector
    and the AVX-512 instructions.}
  \label{a:rnF}
  \begin{algorithmic}[1]
    \REQUIRE{$\mathbf{x}$ is a properly aligned double precision array with $\tilde{m}$ elements, all finite}
    \ENSURE{a finite approximation of $\|\mathbf{x}\|_F$ is computed, represented as $(e'',f'')$}
    \STATE{let $\mathfrak{1},\mskip-1mu\mathsf{1},\mskip-1mu-\mathsf{1},\mskip-1mu-\mathsf{\infty}$ be vectors of ones (integer, real) and negative ones and infinities;}
    \STATE{let $\mathsf{r}_{\mathtt{0}}^{}=(\mathsf{e}^{(\mathtt{0})},\mathsf{f}^{(\mathtt{0})})=(-\mathsf{\infty},\mathsf{1})$ be a pair of vectors representing $\mathtt{s}$ zeros;}
    \FOR[$\mathsf{r}_{\mathtt{i}+1}=\mathop{\mathrm{fma}}(|\mathsf{x}_{\mathtt{i}}|,|\mathsf{x}_{\mathtt{i}}|,\mathsf{r}_{\mathtt{i}})$ as in \cref{ss:SM6.2}]{$\mathtt{i}=0$ \TO $(\mathtt{m}=\tilde{m}/\mathtt{s})-1$\label{nl:3}}
    \STATE{$\mathsf{x}_{\mathtt{i}}=\begin{bmatrix}x_i\!&\!x_{i+1}\!&\!\cdots\!&\!x_{i+\mathtt{s}-1}\end{bmatrix}^T$, where $\mathtt{i}\mathtt{s}+1\le i\le(\mathtt{i}+1)\mathtt{s}$;}
    \COMMENT{$i$ is one-based}
    \STATE{$\check{\mathsf{e}}_{\mathrm{lsb}}^{(\mathtt{i})}=\mathop{\mathtt{\_mm512\_and\_epi64}}(\mathop{\mathtt{\_mm512\_cvtpd\_epi64}}(\mathsf{e}^{(\mathtt{i})}),\mathfrak{1});$}
    \COMMENT{AVX512DQ}
    \STATE{$\hat{\mathsf{e}}^{(\mathtt{i})}=\mathop{\mathtt{cvtepi64}}(\check{\mathsf{e}}_{\mathrm{lsb}}^{(\mathtt{i})});\quad
      \mathsf{e}^{\prime(\mathtt{i})}=\mathop{\mathtt{sub}}(\mathsf{e}^{(\mathtt{i})},\hat{\mathsf{e}}^{(\mathtt{i})});\quad
      \mathsf{f}^{\prime(\mathtt{i})}=\mathop{\mathtt{scalef}}(\mathsf{f}^{(\mathtt{i})},\hat{\mathsf{e}}^{(\mathtt{i})});$}
    \STATE{$\mathsf{r}_{\mathtt{i}}'=(\mathsf{e}^{\prime(\mathtt{i})},\mathsf{f}^{\prime(\mathtt{i})});$}
    \COMMENT{representation of $\mathsf{r}_{\mathtt{i}}^{}$ where $\mathsf{e}^{\prime(\mathtt{i})}$ has even or infinite values}
    \STATE{$|\mathsf{x}_{\mathtt{i}}|=(\mathop{\mathtt{E}}(\mathsf{x}_{\mathtt{i}}),\mathop{\mathtt{F}}(\mathsf{x}_{\mathtt{i}}))=(\mathsf{e}_{\mathtt{i}},\mathsf{f}_{\mathtt{i}});$}
    \STATE{$\hat{\mathsf{e}}_{\mathtt{i}}=\mathop{\mathtt{scalef}}(\mathsf{e}_{\mathtt{i}},\mathsf{1});\quad
      \mathsf{e}^{\prime(\mathtt{i}+1)}=\mathsf{e}_{\max}=\mathop{\mathtt{max}}(\hat{\mathsf{e}}_{\mathtt{i}},\mathsf{e}^{\prime(\mathtt{i})});$}
    \COMMENT{$\lessapprox$ the exponents of $\mathsf{r}_{\mathtt{i}+1}^{}$}
    \STATE{$\hat{\mathsf{e}}_{\mathtt{i}}'=\mathop{\mathtt{max}}(\mathop{\mathtt{sub}}(\hat{\mathsf{e}}_{\mathtt{i}}^{},\mathsf{e}_{\max}^{}),-\mathsf{\infty});\ \
      \mathsf{e}^{\prime\prime(\mathtt{i})}=\mathop{\mathtt{max}}(\mathop{\mathtt{sub}}(\mathsf{e}^{\prime(\mathtt{i})},\mathsf{e}_{\max}^{}),-\mathsf{\infty});$}
    \COMMENT{differences}
    \STATE{$\hat{\mathsf{e}}_{\mathtt{i}}''=\mathop{\mathtt{scalef}}(\hat{\mathsf{e}}_{\mathtt{i}}',-\mathsf{1});\quad
      \hat{\mathsf{f}}_{\mathtt{i}}^{}=\mathop{\mathtt{scalef}}(\mathsf{f}_{\mathtt{i}}^{},\hat{\mathsf{e}}_{\mathtt{i}}'');\quad
      \hat{\mathsf{f}}^{\prime(\mathtt{i})}=\mathop{\mathtt{scalef}}(\mathsf{f}^{\prime(\mathtt{i})},\mathsf{e}^{\prime\prime(\mathtt{i})});$}
    \STATE{$\mathsf{f}^{\prime(\mathtt{i}+1)}=\mathop{\mathtt{fmadd}}(\hat{\mathsf{f}}_{\mathtt{i}}^{},\hat{\mathsf{f}}_{\mathtt{i}}^{},\hat{\mathsf{f}}^{\prime(\mathtt{i})});\quad
      \mathsf{f}^{(\mathtt{i}+1)}=\mathop{\mathtt{F}}(\mathsf{f}^{\prime(\mathtt{i}+1)});$}
    \COMMENT{$\hat{\mathsf{f}}_{\mathtt{i}}^{\mathfrak{2}}+\hat{\mathsf{f}}^{\prime(\mathtt{i})}$ and normalization}
    \STATE{$\mathsf{e}^{(\mathtt{i}+1)}=\mathop{\mathtt{add}}(\mathsf{e}^{\prime(\mathtt{i}+1)},\mathop{\mathtt{E}}(\mathsf{f}^{\prime(\mathtt{i}+1)}));\quad
      \mathsf{r}_{\mathtt{i}+1}^{}=(\mathsf{e}^{(\mathtt{i}+1)},\mathsf{f}^{(\mathtt{i}+1)});$}
    \COMMENT{exps.\ normalized}
    \ENDFOR\COMMENT{output: $\mathsf{r}_{\mathtt{m}}$\label{nl:14}}
    \STATE{sort $\mathsf{r}_{\mathtt{m}}$ according to~\cref{e:vsort} from \cref{sss:SM6.3.1} by \cref{a:bitonic};}
    \STATE{let $\mathsf{c}^{[0]}=\mathsf{c}^{(\mathtt{m})}$, where $\mathsf{c}\in\{\mathsf{e},\mathsf{f}\}$, and $\mathfrak{m}^{[0]}=(01010101)_2$, $\tilde{\mathfrak{m}}^{[0]}=(10101010)_2$;}
    \FOR[reduction~\cref{e:nred} of $\mathsf{r}_{\mathtt{m}}^{}$ to $(e,f)$ as in \cref{sss:SM6.3.2}]{$j=1$ \TO $\lg{\mathtt{s}}$}
    \STATE{$\mathsf{a}_{\mathsf{c}}^{[j-1]}=\mathop{\mathtt{maskz\_compress}}(\mathfrak{m}^{[j-1]},\mathsf{c}^{[j-1]});$}
    \COMMENT{$\mathfrak{m}^{[j-1]}$-indexed parts of $\mathsf{c}\in\{\mathsf{e},\mathsf{f}\}$}
    \STATE{$\mathsf{b}_{\mathsf{c}}^{[j-1]}=\mathop{\mathtt{maskz\_compress}}(\tilde{\mathfrak{m}}^{[j-1]},\mathsf{c}^{[j-1]});$}
    \COMMENT{$\tilde{\mathfrak{m}}^{[j-1]}$-indexed parts of $\mathsf{c}\in\{\mathsf{e},\mathsf{f}\}$}
    \STATE{$\mathsf{f}^{\prime[j]}=\mathop{\mathtt{fmadd}}(\mathop{\mathtt{scalef}}(\mathsf{1},\mathop{\mathtt{max}}(\mathop{\mathtt{sub}}(\mathsf{a}_{\mathsf{e}}^{[j-1]},\mathsf{b}_{\mathsf{e}}^{[j-1]}),-\mathsf{\infty}),\mathsf{a}_{\mathsf{f}}^{[j-1]},\mathsf{b}_{\mathsf{f}}^{[j-1]});$}
    \COMMENT{\cref{e:redj}}
    \STATE{$\mathsf{f}^{[j]}=\mathop{\mathtt{F}}(\mathsf{f}^{\prime[j]});\quad
      \mathsf{e}^{\prime[j]}=\mathsf{b}_{\mathsf{e}}^{[j-1]};\quad
      \mathsf{e}^{[j]}=\mathop{\mathtt{add}}(\mathsf{e}^{\prime[j]},\mathop{\mathtt{E}}(\mathsf{f}^{\prime[j]}));$}
    \COMMENT{normalization}
    \STATE{shift the bitmasks $\mathfrak{m}^{[j-1]}$ and $\tilde{\mathfrak{m}}^{[j-1]}$ $\mathtt{s}/2^j$ bits to the right to get $\mathfrak{m}^{[j]}$ and $\tilde{\mathfrak{m}}^{[j]}$;}
    \ENDFOR\COMMENT{output: $(\mathsf{e}^{[3]},\mathsf{f}^{[3]})$}
    \STATE{let $(e,f)=(\mathsf{e}_1^{[\lg\mathtt{s}]},\mathsf{f}_1^{[\lg\mathsf{s}]})$ and compute $(e'',f'')$ as in \cref{ss:SM6.1};}
    \COMMENT{scalars\label{nl:24}}
    \RETURN{$(e'',f'')$;}
    \COMMENT{also $(e,f)$ and a double precision value $\mathop{\mathrm{fl}}(2^{e''}\!f'')$ for reference}
  \end{algorithmic}
\end{algorithm}

All operations from \cref{a:rnF}, except at line~\ref{nl:24} and for
the main iteration control, are branch-, division-, and
square-root-free, and of similar complexity.  The main loop at
lines~\ref{nl:3}--\ref{nl:14} requires at minimum $20$ vector
arithmetic operations, compared to only one $\mathtt{fmadd}$ in the
standard dot-product computation, but less than $5\times\mskip-2mu$
and $19\times\mskip-2mu$ slowdown is achieved, as shown in
\cref{f:SM6.3,f:SM6.4}, respectively, for $\tilde{m}=m=2^{30}$.  For
$\mathtt{s}$ fixed, the other loops can be unrolled, and together with
the remaining parts of \cref{a:rnF} require a constant number of
operations.
\subsubsection{Sequential reduction of the final partial sums}\label{sss:SM6.3.3}
Instead of vectorizing it, the horizontal reduction could have been
performed sequentially, from the first to the last partial sum,
assuming $(e_i,f_i)\le(e_{i+1},f_{i+1})$, similarly to~\cref{e:redj}
as
\begin{equation}
  (e_{i+1}',f_{i+1}')=(e_i^{},f_{i}^{})+(e_{i+1}^{},f_{i+1}^{})=(e_{i+1}^{},2^{e_i^{}-e_{i+1}^{}}f_i^{}+f_{i+1}^{}),\quad
  1\le i<\mathtt{s}.
  \label{e:reds}
\end{equation}
If $(e_i^{},f_i^{})>(e_{i+1}^{},f_{i+1}^{})$, these two pairs are
swapped.  Having computed~\cref{e:reds} by scalar operations,
$(e_{i+1}^{},f_{i+1}^{})$ is replaced in-place by
$(e_{i+1}',f_{i+1}')$ and $i$ is incremented.

This reduction variant was used in the numerical testing from
\cref{ss:SM6.4}, since its order of summation had a chance of being
more accurate than that of the vectorized variant from
\cref{sss:SM6.3.2}, with the final partial sums sorted as in
\cref{sss:SM6.3.1}.  Note that the sequential reduction does not
presuppose, but can benefit from, this initial sorting.  Finally, the
reduction's result is $(e_{\mathtt{s}}',f_{\mathtt{s}}')=(e,f)$.
\subsection{Numerical testing}\label{ss:SM6.4}
\looseness=-1
\Cref{s:5} gives a description of the test environment.  The reference
BLAS routine \texttt{DNRM2} was renamed, built for comparison, and
verified by disassembling that it used, when possible, the fused
multiply-add instructions.

For each exponent $\xi\in\{0,1008\}$, 65 double precision test vectors
$\mathbf{x}_{\tau}$, $0\le\tau\le 64$, were generated.  For a fixed
$\xi$, every vector had its elements in the range
$\left[0,2^{\xi}\right]$.  Each element was pseudorandomly generated
by the CPU's $\mathtt{RDRAND}$ facility as a 64-bit quantity, with the
same probability of each bit being a zero or a one.  Any candidate
with the magnitude falling out of the given range was discarded and
generated anew.  Finally, the elements were replaced by their absolute
values to aid the future sorting.

An accurate approximation of $\|\mathbf{x}_{\tau}^{}\|_F^{}$, denoted
by $\|\mathbf{x}_{\tau}'\|_F^{}$, was obtained by computing
$\mathbf{x}_{\tau}^{\prime T}\mathbf{x}_{\tau}'$ in quadruple
(128-bit) precision and rounding the result's square root to double
precision, where $\mathbf{x}_{\tau}'$ has the same elements as
$\mathbf{x}_{\tau}^{}$, but sorted non-decreasingly, to improve
accuracy, and converted to quadruple precision.  Then,
\begin{displaymath}
  \max\left\{\frac{|c_{\tau}^{(\diamond)}-\|\mathbf{x}_{\tau}'\|_F^{}|}{\|\mathbf{x}_{\tau}'\|_F^{}},0\right\}\le\infty
\end{displaymath}
is the relative error of an approximation
$c_{\tau}^{(\diamond)}\approx\|\mathbf{x}_{\tau}^{}\|_F^{}$ computed
by a procedure $\diamond$.  The relative errors are shown in
\cref{f:SM6.1,f:SM6.2} for $\xi=0$ and $\xi=1008$, respectively, where
BLAS $\mathtt{DNRM2}$ refers to the reference Fortran implementation.

\begin{figure}[hbtp]
  \begin{center}
    \includegraphics{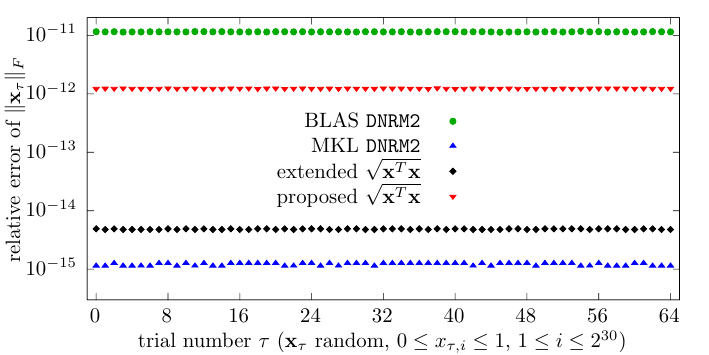}
  \end{center}
  \caption{Relative errors of $\|\mathbf{x}_{\tau}\|_F$ computed by
    several procedures, for $\xi=0$.}
  \label{f:SM6.1}
\end{figure}

\begin{figure}[hbtp]
  \begin{center}
    \includegraphics{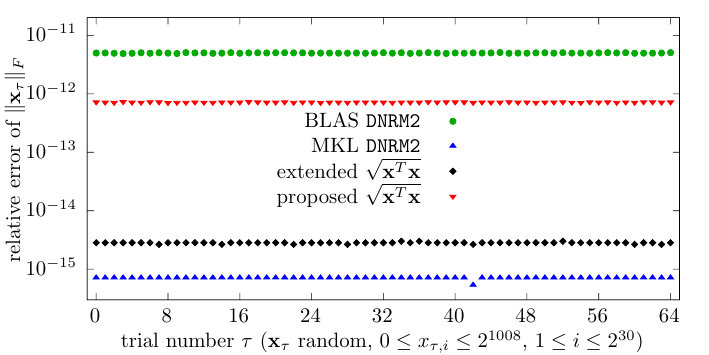}
  \end{center}
  \caption{Relative errors of $\|\mathbf{x}_{\tau}\|_F$ computed by
    several procedures, for $\xi=1008$.}
  \label{f:SM6.2}
\end{figure}

\Cref{f:SM6.1,f:SM6.2} suggest that the reference BLAS routine is
consistently about four orders of magnitude less accurate than the
MKL's $\mathtt{DNRM2}$.  Also, the MKL's routine is significantly more
accurate than computing in extended precision, what indicates that the
MKL does not use the reference algorithm in its published form.

Therefore, the MKL's routine is the most accurate one, but bound to
overflow at least when the actual result is not finitely representable
in double precision, while \cref{a:rnF} (or the ``proposed
$\sqrt{\mathbf{x}^T\mathbf{x}}$'') and the computation of
$\sqrt{\mathbf{x}^T\mathbf{x}}$ in extended precision cannot overflow
for any reasonably-sized $\mathbf{x}$.  However, from
\cref{f:SM6.3,f:SM6.4} it can be concluded that \cref{a:rnF} is
somewhat slower that the extended dot-product, which in turn is
significantly slower than the MKL's $\mathtt{DNRM2}$.  On the
MKL-targeted platforms, its $\mathtt{DNRM2}$ is thus the method of
choice for the Jacobi-type SVD, even if that entails a sporadic
downscaling of the input data.

\begin{figure}[hbtp]
  \begin{center}
    \includegraphics{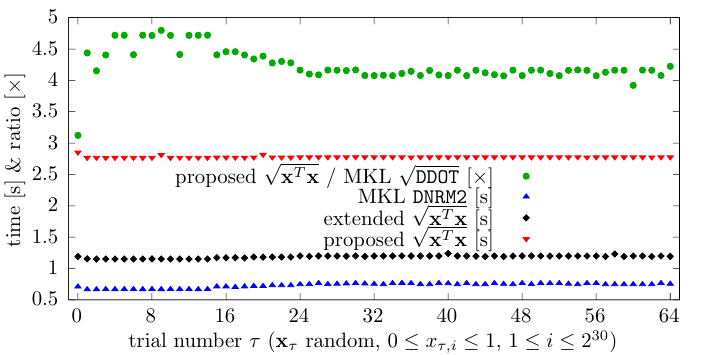}
  \end{center}
  \caption{Performance of computing $\|\mathbf{x}_{\tau}\|_F$ by
    several procedures, for $\xi=0$.}
  \label{f:SM6.3}
\end{figure}

\begin{figure}[hbtp]
  \begin{center}
    \includegraphics{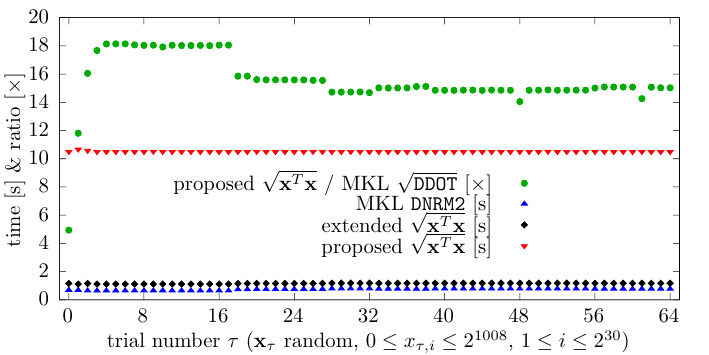}
  \end{center}
  \caption{Performance of computing $\|\mathbf{x}_{\tau}\|_F$ by
    several procedures, for $\xi=1008$.}
  \label{f:SM6.4}
\end{figure}

\paragraph{Vectorized quadruple precision}
An accurate alternative way of computing
$\sqrt{\mathbf{x}^T\mathbf{x}}$ without overflow is to use the SLEEF's
vectorized quad-precision math
library\footnote{\url{https://sleef.org/quad.xhtml}} by modifying
\cref{a:vdp} to convert the loaded double precision vectors to
quadruple precision ones and then call the vectorized quad-fma on
them.  Such an approach gives the results virtually indistinguishable
from a non-vectorized quadruple dot-product, but in the testing
environment, for $\xi=0$, it is $\approx 164\times$ and
$\approx 1.69\times$ slower than the MKL's $\mathtt{DDOT}$ and the
BLAS $\mathtt{DNRM2}$, respectively, maybe because SLEEF does not yet
provide the inlineable versions of its functions for the Intel's
compilers.
\subsection{Conclusion}\label{ss:SM6.5}
When the MKL's $\mathtt{DNRM2}$ and the hardware-supported extended
precision are not available, and the SLEEF's vectorized quad-precision
library happens to be too slow (or unavailable), the proposed
$\sqrt{\mathbf{x}^T\mathbf{x}}$ method can be employed as a fail-safe
option, either for computing $\|\mathbf{x}\|_F$ on its own, or to be
called on the second attempt, after a faster but at least comparably
accurate method overflows.

It is also easy to convert in a portable way the result of other
methods in double, extended, or quadruple precision to the $(e,f)$
representation, which thus can be used for encoding the Frobenius
norms of the columns of the iteration matrix in the Jacobi SVD,
regardless of the chosen method of computing them.  The remaining
parts of the SVD algorithm do not rely on a higher precision of
computation or of data, so a possible loss of it due to the rounding
of $f$ to double precision is acceptable.

As $m$ becomes larger, and/or with an unfavorable distribution of the
elements' magnitudes, such that the huge elements precede the small
ones, at a certain point the partial sums in the proposed algorithm
can become too large to be (significantly) affected by the subsequent
updates, resulting in a too small final norm and hence a big relative
error.  It is worth exploring if, e.g., three vector accumulators
instead of one---for the small, the medium, and the large magnitudes,
in the spirit of the Blue's algorithm~\cite{Blue-78}---could improve
the relative accuracy.  However, such an approach will inevitably
reduce the amount of the vector parallelism, since the present
instructions have to be trebled and masked to affect the appropriate
accumulator only.
\section{Addenda for \cref{s:5}}\label{s:SM7}
\Cref{s:5} from the main paper is here expanded.
\subsection{Addendum for \cref{ss:5.2}}\label{ss:SM7.1}
For the eigenvector matrix $\Phi$ of a Hermitian matrix of order two
holds $|\det\Phi|=c^2+|s|^2=1$, where $c=\cos\varphi$ and
$s=\mathrm{e}^{\mathrm{i}\alpha}\sin\varphi$.  Therefore,
$\delta_{\Phi}^{\text{\texttt{x}}}=||\det(\mathop{\mathrm{fl}}(\Phi))|-1|$
determines how much the real cosine and the real or complex sine,
computed by a routine designated by \texttt{x}, depart from the ones
that would make $\Phi$ unitary.  For a batch $\tau$, let
$\delta_{\tau}^{\text{\texttt{x}}}=\max_i^{}\delta_{\Phi_{\tau,i}}^{\text{\texttt{x}}}$,
$1\le i\le 2^{28}$.  Then, \cref{f:SM7.1} shows that, in the real
case, the batched and the LAPACK-based EVDs do not differ much in this
measure of accuracy, but in the complex case the LAPACK-based ones may
be several orders of magnitude less accurate.  Together with
\cref{f:SM7.2,f:SM7.3}, described below, this further explains
\cref{f:5.2}.

\begin{figure}[hbtp]
  \begin{center}
    \includegraphics{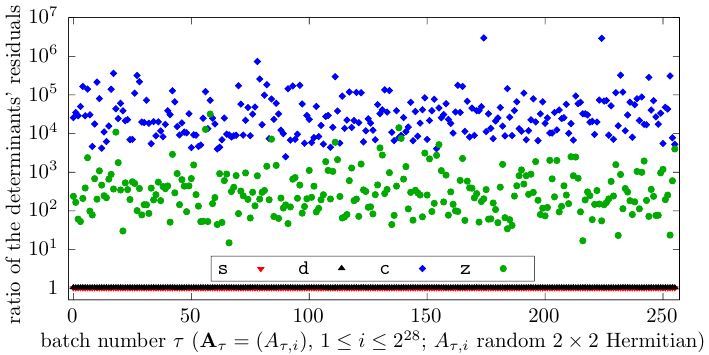}
  \end{center}
  \caption{Per-batch ratios
    $\delta_{\tau}^{\text{\texttt{S}}}/\delta_{\tau}^{\text{\texttt{s}}}$,
    $\delta_{\tau}^{\text{\texttt{D}}}/\delta_{\tau}^{\text{\texttt{d}}}$,
    $\delta_{\tau}^{\text{\texttt{C}}}/\delta_{\tau}^{\text{\texttt{c}}}$,
    and
    $\delta_{\tau}^{\text{\texttt{Z}}}/\delta_{\tau}^{\text{\texttt{z}}}$
    of the determinants' residuals.}
  \label{f:SM7.1}
\end{figure}

If a closer look is taken at the matrix $A_{\max}^{\delta}$ that
caused the greatest determinant's residual with \texttt{CLAEV2}, and
at the computed outputs $\Lambda_{\max}^{\delta}$ and
$U_{\max}^{\delta}$ (all in single precision, but printed out with
$21$ digits after the decimal point for reproducibility),
\begin{displaymath}
  A_{\max}^{\delta}=
  \begin{bmatrix}
    \text{\texttt{A}}&\text{\texttt{B}}\\
    \bar{\text{\texttt{B}}}&\text{\texttt{C}}
  \end{bmatrix},\qquad
  \Lambda_{\max}^{\delta}=
  \begin{bmatrix}
    \text{\texttt{RT1}}&0\\
    0&\text{\texttt{RT2}}
  \end{bmatrix},\qquad
  U_{\max}^{\delta}=
  \begin{bmatrix}
    \text{\texttt{CS1}}&-\overline{\text{\texttt{SN1}}}\\
    \text{\texttt{SN1}}&\hphantom{-}\text{\texttt{CS1}}
  \end{bmatrix},
\end{displaymath}
where
\begin{displaymath}
  \begin{gathered}
    \text{\texttt{A}}=-5.540058702080522136604\cdot 10^{-39},\\
    \text{\fbox{$\text{\texttt{B}}=-1.401298464324817070924\cdot 10^{-45}+1.401298464324817070924\cdot 10^{-45}\mathrm{i}$}},\\
    \text{\texttt{C}}=-5.832059874778063193026\cdot 10^{-39},
  \end{gathered}
\end{displaymath}
and
\begin{displaymath}
  \begin{gathered}
    \text{\texttt{RT1}}=\lambda_1=-5.832059874778063193026\cdot 10^{-39},\\
    \text{\texttt{RT2}}=\lambda_2=-5.540058702080522136604\cdot 10^{-39};\\
    \text{\texttt{CS1}}=\cos\varphi=-4.798947884410154074430\cdot 10^{-6},\\
    \text{\fbox{\fbox{$\text{\texttt{SN1}}=\mathrm{e}^{\mathrm{i}\alpha}\sin\varphi=-1-\mathrm{i}$}}},
  \end{gathered}
\end{displaymath}
a disastrous failure to represent $|\text{\texttt{B}}|$ in the
subnormal range by any value other than
$|\Re{\text{\texttt{B}}}|=|\Im{\text{\texttt{B}}}|$, as explained in
\cref{r:hypot}, becomes the obvious reason for the loss of accuracy.
But why was $\sin\varphi$ computed as unity?  The answer is that the
eigenvalues had to be swapped to make $|\lambda_1|\ge|\lambda_2|$.
Since \texttt{CLAEV2} calls \texttt{SLAEV2} with \texttt{A},
$|\text{\texttt{B}}|$, and \texttt{C}, the latter routine performed
the exchange $\cos\varphi\leftarrow-\sin\varphi$,
$\sin\varphi\leftarrow\cos\varphi$, as if the columns of its, real $U$
were swapped.  Then \texttt{CLAEV2} multiplied
$\bar{\text{\texttt{B}}}/|\text{\texttt{B}}|=\mathrm{e}^{\mathrm{i}\alpha}=-1-\mathrm{i}$
by the new value of $\sin\varphi$, i.e., by unity.  Note that if
\texttt{A} and \texttt{C} changed roles,
$\mathrm{e}^{\mathrm{i}\alpha}$ would have remained $-1-\mathrm{i}$,
but it would have been multiplied by
$\sin\varphi\approx 4.798947884\cdot 10^{-6}$.  Even tough
$\text{\texttt{CS1}}=-1$ then, $||\det{U}|-1|$ would have been small.
\Cref{a:z8jac2} ensures $|\sin\varphi|\le 1/\sqrt{2}$, up to the
rounding errors, so a bogus $\mathrm{e}^{\mathrm{i}\alpha}$ is always
scaled down.

As already known for the LAPACK's EVD, the smaller eigenvalue by
magnitude might not be computed with any relative accuracy, and this
is also the case with the batched EVD\@.  The normwise residuals
$\lambda_F^{\text{\texttt{x}},\tau}=\max_i^{}\|\mathop{\mathrm{fl}}(\Lambda_{\tau,i}^{})-\Lambda_{\tau,i}^{}\|_F^{}/\|\Lambda_{\tau,i}^{}\|_F^{}$,
however, seem stable, and their ratios are shown in \cref{f:SM7.2}.
The values below unity indicate when the batched EVD was less
accurate, while those above unity demonstrate that in a majority of
batches the LAPACK-based EVD was more inaccurate.

\begin{figure}[hbtp]
  \begin{center}
    \includegraphics{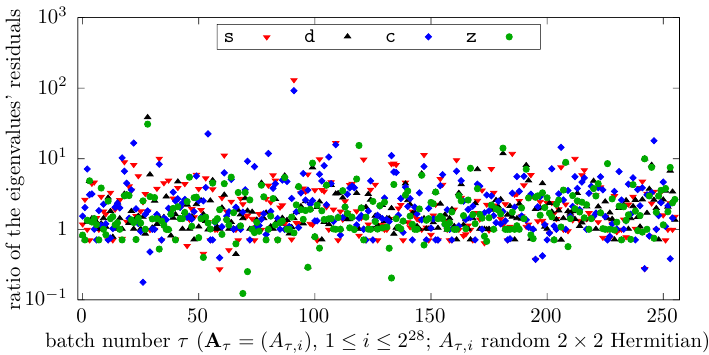}
  \end{center}
  \caption{Per-batch ratios
    $\lambda_F^{\text{\texttt{S}},\tau}/\lambda_F^{\text{\texttt{s}},\tau}$,
    $\lambda_F^{\text{\texttt{D}},\tau}/\lambda_F^{\text{\texttt{d}},\tau}$,
    $\lambda_F^{\text{\texttt{C}},\tau}/\lambda_F^{\text{\texttt{c}},\tau}$,
    and
    $\lambda_F^{\text{\texttt{Z}},\tau}/\lambda_F^{\text{\texttt{z}},\tau}$
    of the eigenvalues' normwise residuals.}
  \label{f:SM7.2}
\end{figure}

A similar conclusion holds for the eigenvalue larger by magnitude,
$\lambda_{\max}$, that is always computed relatively accurate, up to a
modest constant times $\varepsilon$.  \Cref{f:SM7.3} shows the ratios
of the $\max$-residuals,
$\lambda_{\max}^{\text{\texttt{x}},\tau}=\max_i^{}|\mathop{\mathrm{fl}}(\lambda_{\max;\tau,i}^{})-\lambda_{\max;\tau,i}^{}|/|\lambda_{\max;\tau,i}^{}|$.

\begin{figure}[hbtp]
  \begin{center}
    \includegraphics{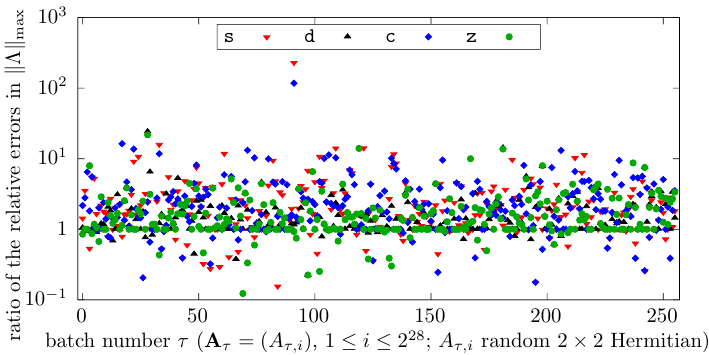}
  \end{center}
  \caption{Per-batch ratios
    $\lambda_{\max}^{\text{\texttt{S}},\tau}/\lambda_{\max}^{\text{\texttt{s}},\tau}$,
    $\lambda_{\max}^{\text{\texttt{D}},\tau}/\lambda_{\max}^{\text{\texttt{d}},\tau}$,
    $\lambda_{\max}^{\text{\texttt{C}},\tau}/\lambda_{\max}^{\text{\texttt{c}},\tau}$,
    and
    $\lambda_{\max}^{\text{\texttt{Z}},\tau}/\lambda_{\max}^{\text{\texttt{z}},\tau}$
    of the eigenvalues' $\max$-residuals.}
  \label{f:SM7.3}
\end{figure}
\subsection{Addendum for \cref{ss:5.3}}\label{ss:SM7.2}
\Cref{f:SM7.4,f:SM7.5} show the number of sweeps until convergence in
the real and the complex case on $\Xi_1^{\mathbb{F}}$ and
$\Xi_2^{\mathbb{F}}$, respectively, while \cref{f:SM7.6,f:SM7.7,f:SM7.8}
correspond to \cref{f:5.3,f:5.4,f:5.5}, respectively, in the real case
on $\Xi_1^{\mathbb{R}}$.

\begin{figure}[hbtp]
  \begin{center}
    \includegraphics{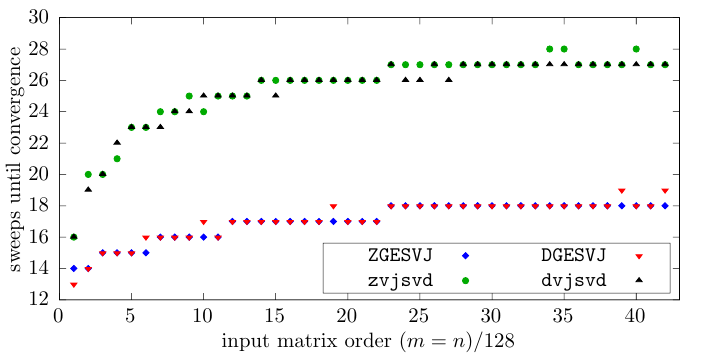}
  \end{center}
  \caption{Number of sweeps until convergence for the LAPACK's
    \texttt{xGESVJ} and the proposed routines (with the $\mathtt{MM}$
    parallel quasi-cyclic pivot strategy), on $\Xi_1^{\mathbb{F}}$.}
  \label{f:SM7.4}
\end{figure}

\begin{figure}[hbtp]
  \begin{center}
    \includegraphics{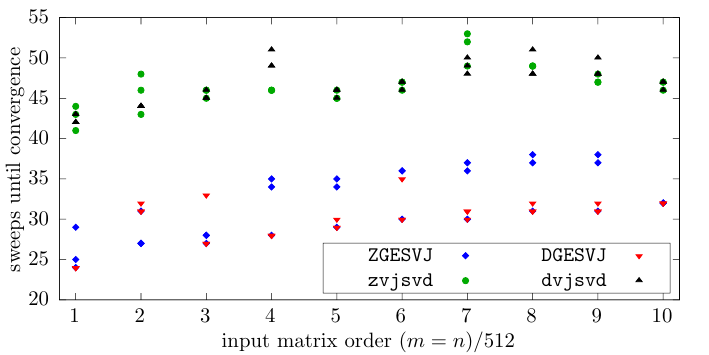}
  \end{center}
  \caption{Number of sweeps until convergence for the LAPACK's
    \texttt{xGESVJ} and the proposed routines (with the $\mathtt{ME}$
    parallel cyclic pivot strategy), on $\Xi_2^{\mathbb{F}}$.}
  \label{f:SM7.5}
\end{figure}

\begin{figure}[hbtp]
  \begin{center}
    \includegraphics{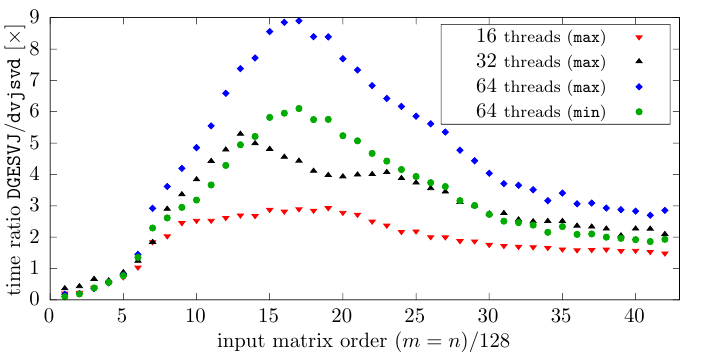}
  \end{center}
  \caption{Run-time ratios of \texttt{DGESVJ} and $\mathtt{dvjsvd}$
    with $\mathtt{MM}$ on $\Xi_1^{\mathbb{R}}$.}
  \label{f:SM7.6}
\end{figure}

\begin{figure}[hbtp]
  \begin{center}
    \includegraphics{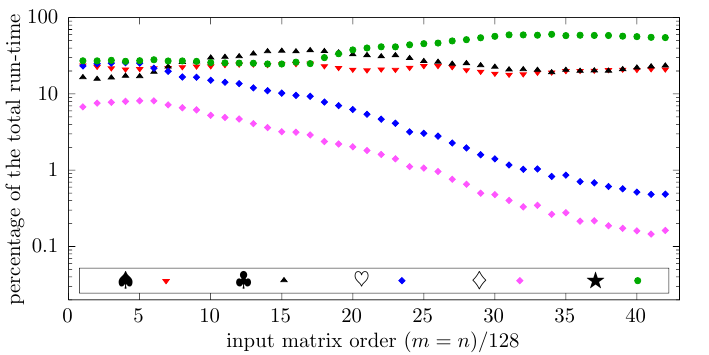}
  \end{center}
  \caption{Breakdown of the run-time of $\mathtt{dvjsvd}$ with
    $\mathtt{MM}$ on $\Xi_1^{\mathbb{R}}$ with $64$ threads.}
  \label{f:SM7.7}
\end{figure}

\begin{figure}[hbtp]
  \begin{center}
    \includegraphics{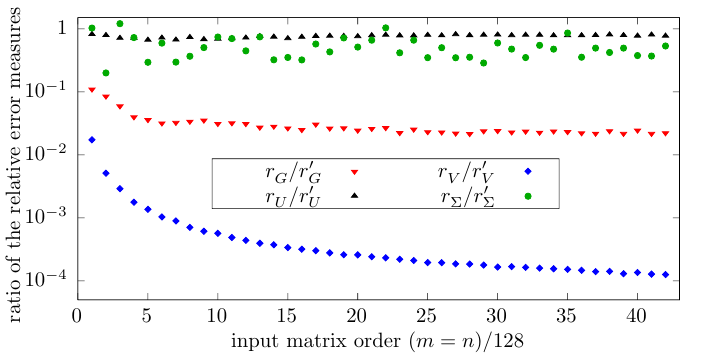}
  \end{center}
  \caption{Ratios of the relative error measures for the SVDs on
    $\Xi_1^{\mathbb{R}}$.}
  \label{f:SM7.8}
\end{figure}
\subsubsection{Performance of several double precision routines}\label{sss:SM7.2.1}
Performance of several sequential double precision MKL routines
(\texttt{DDOT}, \texttt{DROT}) and the similar ones proposed here
($\mathtt{ddpscl}'$, $\mathtt{djrotf}$) were compared in terms of giga
($10^9$) floating-point operations per second, i.e., GFLOP/s.  The
routines' FLOP counts are taken%
\footnote{See \url{https://github.com/venovako/VecJac/blob/master/src/dflops.c} for the testing code.}
as
\begin{equation}
  \mathop{\mathrm{FLOP}}(n)=\begin{cases}
  n, & \text{\texttt{DDOT} ($1$ \texttt{fma} per row of both arrays)},\\
  4n, & \text{\texttt{DROT} ($1$ \texttt{fma}, $1$ \texttt{mul} per row of each array)},\\
  18n+35, & \text{$\mathtt{ddpscl}'$ (\text{see \cref{a:ddpsclcs}})},\\
  4n, & \text{$\mathtt{djrotf}$ ($1$ \texttt{fma}, $1$ \texttt{mul} per row of each array)},
  \end{cases}
  \label{e:FLOP}
\end{equation}
where $n$ is the length of two input arrays of each routine,
$\mathtt{djrotf}$ is the ``fast'' variant of $\mathtt{djrot}$ without
the norm approximation, \texttt{fma} is the fused multiply-add
operation, \texttt{scalef} is the scaling by a power of two, and
\texttt{mul} is the floating-point multiplication.

\looseness=-1
\Cref{f:SM7.9} shows the testing results for $n=128i$, $1\le i\le 42$.
Reliability of timing was ensured by calling each routine with
different random inputs $100000$ times.  Observe that
$\mathtt{ddpscl}'$ is more performant than \texttt{DDOT} since the
former does more arithmetical processing with each input loaded from
the memory (by scaling the arrays' elements in both its
implementations, and by using the compensated summation in the
enhanced one, which was timed), while $\mathtt{djrotf}$ and
\texttt{DROT} are similar performance-wise.

\begin{figure}[hbtp]
  \begin{center}
    \includegraphics{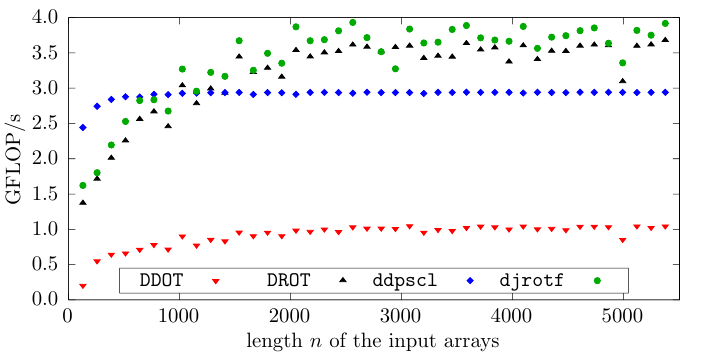}
  \end{center}
  \caption{Performance in GFLOP/s for various double precision routines.}
  \label{f:SM7.9}
\end{figure}

\looseness=-1
It is harder to figure out the exact implementation of complex
arithmetic in the MKL routines and deduce the FLOP counts from that.
The complex routines' execution could be profiled and the retired
floating-point instructions counted, what would be an effort worth
taking if the similar complex routines proposed here were known to be
optimized to the fullest extent possible.  This is therefore left for
future work.
\section{The single precision variants of the parallel Jacobi SVD method}\label{s:SM8}
The single precision implementation is conceptually identical to the
double precision one, with the necessary changes for the datatype and
the associated range of values.  For brevity, the pseudocode of the
single precision routines is omitted and the readers are referred to
the actual source code on GitHub, while bearing in mind that
$\mathtt{s}=16$.

\looseness=-1
Two single precision datasets, $\Xi_3^{\mathbb{R}}$ and
$\Xi_3^{\mathbb{C}}$, were generated similarly to $\Xi_1^{\mathbb{R}}$
and $\Xi_1^{\mathbb{C}}$, respectively (see \cref{ss:5.1}), but with
$\xi_3^{}=-12$.  \Cref{f:SM8.1} corresponds to \cref{f:SM7.4}, showing
that significantly more sweeps were also required for convergence
under parallel than under serial pivot strategies in the single
precision case.  \Cref{f:SM8.2,f:SM8.4}, compared to
\cref{f:5.3,f:SM7.6}, respectively, show similar speedup profiles,
albeit with a bit higher peaks and the larger matrix orders for which
they were attained.  \Cref{f:SM8.3,f:SM8.5} depict the percentages of
run-time of the key routines comparable to those on
\Cref{f:5.4,f:SM7.7}.  \Cref{f:SM8.6,f:SM8.7} show similar relative
error ratios as \cref{f:5.5,f:SM7.8}, respectively.

For the $\mathtt{sdpscl}$ routine in \cref{f:SM8.8} the FLOP count was
taken as $18n+67$, while the counts for the other routines remained
the same as in~\cref{e:FLOP}.  Compared to \cref{f:SM7.9},
\texttt{SROT} and $\mathtt{sdpscl}$ were noticeably less performant
than their double precision counterparts, what remains to be
explained.  Apart from that, the testing results of the single
precision real and complex variants of the SVD method are generally
consistent with those of the double precision real and complex ones,
respectively.

\begin{figure}[hbtp]
  \begin{center}
    \includegraphics{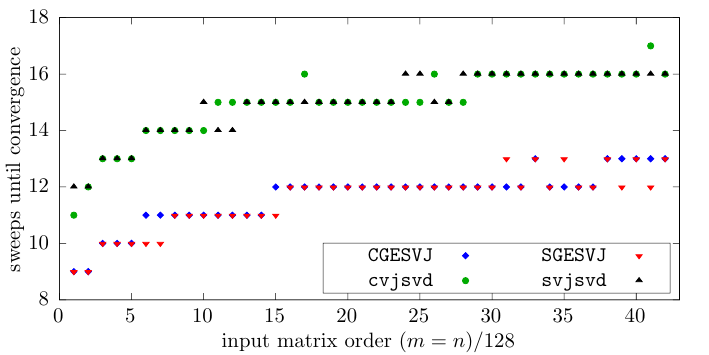}
  \end{center}
  \caption{Number of sweeps until convergence for the LAPACK's
    \texttt{xGESVJ} and the proposed routines (with the $\mathtt{MM}$
    parallel quasi-cyclic pivot strategy), on $\Xi_3^{\mathbb{F}}$.}
  \label{f:SM8.1}
\end{figure}

\begin{figure}[hbtp]
  \begin{center}
    \includegraphics{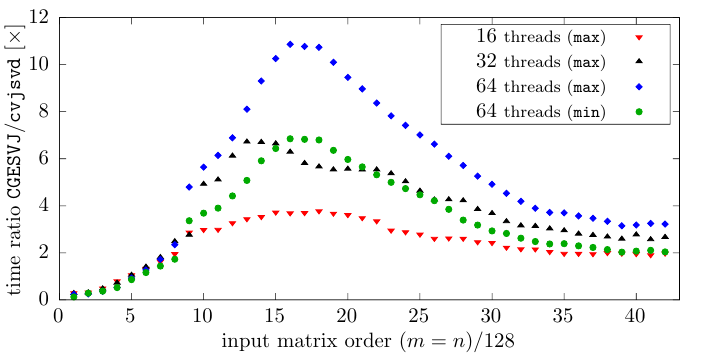}
  \end{center}
  \caption{Run-time ratios of \texttt{CGESVJ} and $\mathtt{cvjsvd}$
    with $\mathtt{MM}$ on $\Xi_3^{\mathbb{C}}$.}
  \label{f:SM8.2}
\end{figure}

\begin{figure}[hbtp]
  \begin{center}
    \includegraphics{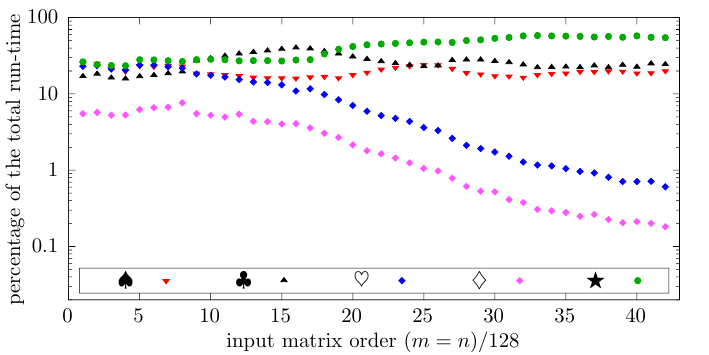}
  \end{center}
  \caption{Breakdown of the run-time of $\mathtt{cvjsvd}$ with
    $\mathtt{MM}$ on $\Xi_3^{\mathbb{C}}$ with $64$ threads.}
  \label{f:SM8.3}
\end{figure}

\begin{figure}[hbtp]
  \begin{center}
    \includegraphics{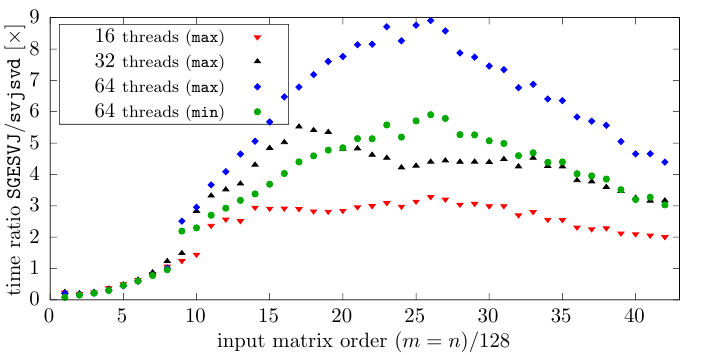}
  \end{center}
  \caption{Run-time ratios of \texttt{SGESVJ} and $\mathtt{svjsvd}$
    with $\mathtt{MM}$ on $\Xi_3^{\mathbb{R}}$.}
  \label{f:SM8.4}
\end{figure}

\begin{figure}[hbtp]
  \begin{center}
    \includegraphics{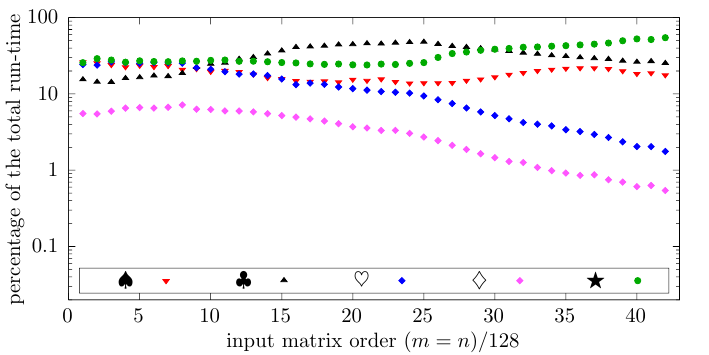}
  \end{center}
  \caption{Breakdown of the run-time of $\mathtt{svjsvd}$ with
    $\mathtt{MM}$ on $\Xi_3^{\mathbb{R}}$ with $64$ threads.}
  \label{f:SM8.5}
\end{figure}

\begin{figure}[hbtp]
  \begin{center}
    \includegraphics{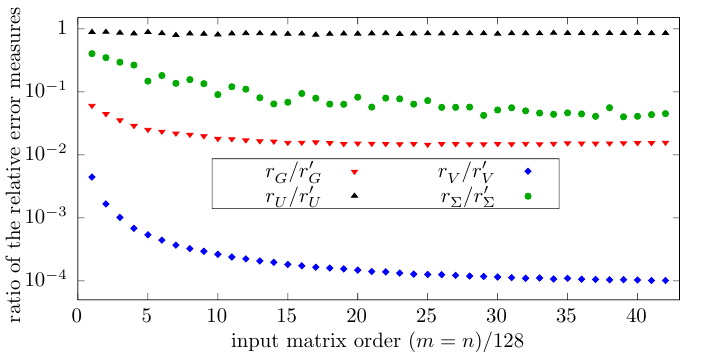}
  \end{center}
  \caption{Ratios of the relative error measures for the SVDs on
    $\Xi_3^{\mathbb{C}}$.}
  \label{f:SM8.6}
\end{figure}

\begin{figure}[hbtp]
  \begin{center}
    \includegraphics{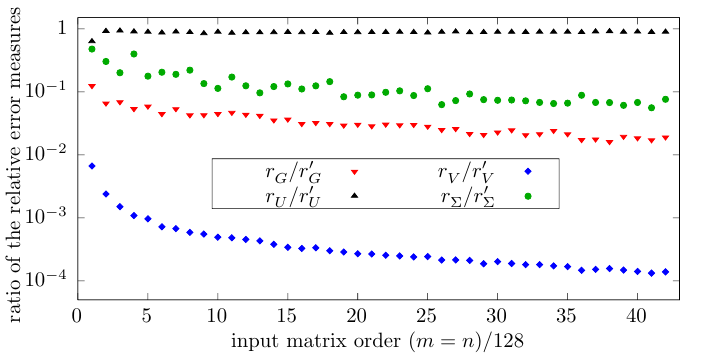}
  \end{center}
  \caption{Ratios of the relative error measures for the SVDs on
    $\Xi_3^{\mathbb{R}}$.}
  \label{f:SM8.7}
\end{figure}

\begin{figure}[hbtp]
  \begin{center}
    \includegraphics{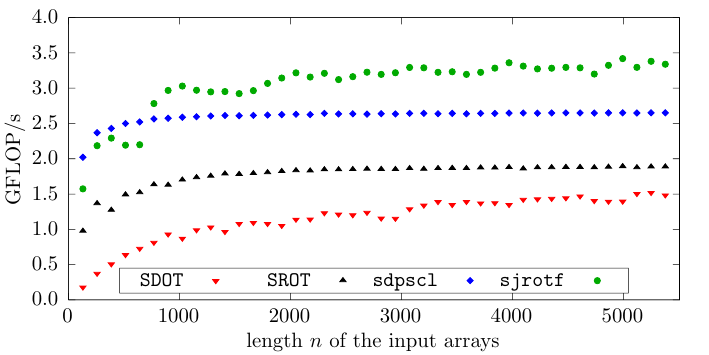}
  \end{center}
  \caption{Performance in GFLOP/s for various single precision routines.}
  \label{f:SM8.8}
\end{figure}
\section{Vectorization of the Gram--Schmidt orthogonalization}\label{s:SM9}
\looseness=-1
The complex Gram--Schmidt orthogonalization is vectorized in
\cref{a:ZGS}, and the real one in \cref{a:DGS}, both shown in double
precision.  The actual implementation of these algorithms also
includes the $\max$-norm estimation of $\mathop{\mathrm{fl}}(g_q')$,
in the same vein as in~\cref{a:zjrot}, but it has been omitted from
their presentation here for brevity.

\begin{algorithm}[hbtp]
  \caption{$\mathtt{zgsscl}$: a vectorized complex Gram--Schmidt orthogonalization.}
  \label{a:ZGS}
  \begin{algorithmic}[1]
    \REQUIRE{$a_{21}'=(\Re{a_{21}'},\Im{a_{21}'});\quad g_p^{}=(\Re{g_p^{}},\Im{g_p^{}}),\ g_q^{}=(\Re{g_q^{}},\Im{g_q^{}});$\\$0<\|g_p^{}\|_F^{}=(e_p^{},f_p^{}),\ 0<\|g_q^{}\|_F^{}=(e_q^{},f_q^{});$}
    \ENSURE{$\mathop{\mathrm{fl}}(g_q')$ from~\cref{e:GS}.}
    \STATE{$\mathsf{e}_q^{}=\mathop{\mathtt{set1}}(e_q^{});\quad-\mathsf{e}_p^{}=\mathop{\mathtt{set1}}(-e_p^{});\quad-\mathsf{e}_q^{}=\mathop{\mathtt{set1}}(-e_q^{});\quad f_{q/p}^{}=f_q^{}/f_p^{};$}
    \STATE{$-\Re{\bm{\psi}}=\mathop{\mathtt{set1}}(-\Re{a_{21}'}\cdot f_{q/p}^{});\quad-\Im{\bm{\psi}}=\mathop{\mathtt{set1}}(\Im{a_{21}'\cdot f_{q/p}^{}});$}
    \COMMENT{$-\bm{\psi}=-a_{21}^{\prime\ast}\cdot f_{q/p}^{};$}
    \FOR[sequentially]{$\mathtt{i}=0$ \TO $\tilde{m}-1$ \textbf{step} $\mathtt{s}$}
    \STATE{$\Re{\mathsf{x}}=\mathop{\mathtt{load}}(\Re{g_p}+\mathtt{i});\quad\Im{\mathsf{x}}=\mathop{\mathtt{load}}(\Im{g_p}+\mathtt{i});$}
    \COMMENT{load a chunk of $g_p$}
    \STATE{$\Re{\mathsf{x}'}=\mathop{\mathtt{scalef}}(\Re{\mathsf{x}},-\mathsf{e}_p);\quad\Im{\mathsf{x}'}=\mathop{\mathtt{scalef}}(\Im{\mathsf{x}},-\mathsf{e}_p);$}
    \COMMENT{$\mathsf{x}'=\mathsf{2}^{-\mathsf{e}_p}\mathsf{x}$}
    \STATE{$\Re{\mathsf{y}}=\mathop{\mathtt{load}}(\Re{g_q}+\mathtt{i});\quad\Im{\mathsf{y}}=\mathop{\mathtt{load}}(\Im{g_q}+\mathtt{i});$}
    \COMMENT{load a chunk of $g_q$}
    \STATE{$\Re{\mathsf{y}'}=\mathop{\mathtt{scalef}}(\Re{\mathsf{y}},-\mathsf{e}_q);\quad\Im{\mathsf{y}'}=\mathop{\mathtt{scalef}}(\Im{\mathsf{y}},-\mathsf{e}_q);$}
    \COMMENT{$\mathsf{y}'=\mathsf{2}^{-\mathsf{e}_q}\mathsf{y}$}
    \STATE{$\Re{\mathsf{y}'}=\mathop{\mathtt{scalef}}(\mathop{\mathtt{fmadd}}(-\Re{\bm{\psi}},\Re{\mathsf{x}'},\mathop{\mathtt{fnmadd}}(-\Im{\bm{\psi}},\Im{\mathsf{x}'},\Re{\mathsf{y}'})),\mathsf{e}_q);$}
    \COMMENT{$\Re{\cref{e:GS}}$}
    \STATE{$\Im{\mathsf{y}'}=\mathop{\mathtt{scalef}}(\mathop{\mathtt{fmadd}}(-\Re{\bm{\psi}},\Im{\mathsf{x}'},\mathop{\mathtt{fmadd}}(-\Im{\bm{\psi}},\Re{\mathsf{x}'},\Im{\mathsf{y}'})),\mathsf{e}_q);$}
    \COMMENT{$\Im{\cref{e:GS}}$}
    \STATE{$\mathop{\mathtt{store}}(\Re{g_q^{}}+\mathtt{i},\Re{\mathsf{y}'});\quad\mathop{\mathtt{store}}(\Im{g_q^{}}+\mathtt{i},\Im{\mathsf{y}'});$}
    \COMMENT{store $g_q'$}
    \ENDFOR\COMMENT{$g_q^{}\to g_q'$, $g_p^{}$ unchanged}
  \end{algorithmic}
\end{algorithm}

\begin{algorithm}[hbtp]
  \caption{$\mathtt{dgsscl}$: a vectorized real Gram--Schmidt orthogonalization.}
  \label{a:DGS}
  \begin{algorithmic}[1]
    \REQUIRE{$a_{21}';\quad g_p^{},\ g_q^{};\quad 0<\|g_q^{}\|_F^{}=(e_q^{},f_q^{}),\ 0<\|g_p^{}\|_F^{}=(e_p^{},f_p^{});$.}
    \ENSURE{$\mathop{\mathrm{fl}}(g_q')$ from~\cref{e:GS}.}
    \STATE{$\mathsf{e}_q^{}=\mathop{\mathtt{set1}}(e_q^{});\ -\mathsf{e}_p^{}=\mathop{\mathtt{set1}}(-e_p^{});\ -\mathsf{e}_q^{}=\mathop{\mathtt{set1}}(-e_q^{});\ -\bm{\psi}=\mathop{\mathtt{set1}}(-a_{21}'(f_q^{}/f_p^{}));$}
    \FOR[sequentially]{$\mathtt{i}=0$ \TO $\tilde{m}-1$ \textbf{step} $\mathtt{s}$}
    \STATE{$\mathsf{x}=\mathop{\mathtt{load}}(g_p+\mathtt{i});\qquad\mathsf{y}=\mathop{\mathtt{load}}(g_q+\mathtt{i});$}
    \COMMENT{scaling by a power of two is \ldots}
    \STATE{$\mathsf{x}=\mathop{\mathtt{scalef}}(\mathsf{x},-\mathsf{e}_p);\qquad\mathsf{y}=\mathop{\mathtt{scalef}}(\mathsf{y},-\mathsf{e}_q);$}
    \COMMENT{\ldots\ faster than full division}
    \STATE{$\mathsf{y}=\mathop{\mathtt{scalef}}(\mathop{\mathtt{fmadd}}(-\bm{\psi},\mathsf{x},\mathsf{y}),\mathsf{e}_q);\qquad\mathop{\mathtt{store}}(g_q+\mathtt{i},\mathsf{y});$}
    \COMMENT{\cref{e:GS}}
    \ENDFOR\COMMENT{$g_q^{}\to g_q'$, $g_p^{}$ unchanged}
  \end{algorithmic}
\end{algorithm}

A test $16\times 16$ matrix for the method with \cref{a:DGS} was
constructed as:
\begin{displaymath}
  G_{ij}=\begin{cases}
  2^{\hat{\eta}-(j-1)-|i-j|}, & 1\le j\le 8,\\
  2^{-51-(j-1)-|i-j|}, & 9\le j\le 16,
  \end{cases}
\end{displaymath}
bearing in mind that $\upsilon=2^{-53}\sqrt{16}=2^{-51}$ (see
\cref{sss:3.5.1}).  Of 128 transformations in the first three
$\mathtt{MM}$ sweeps, 72 were the Gram--Schmidt orthogonalization, as
well as 48 of 68 transformations in the last non-empty sweep, with
$r_G'\approx 1.095015\cdot 10^{-15}$, compared to
$r_G^{}\approx 3.220344\cdot 10^{-16}$ from \texttt{DGESVJ} on $G$
with an extra sweep.
\section{The batched eigendecomposition of Hermitian matrices of order two in CUDA}\label{s:SM10}
The serial code from Listing~\ref{l:1} can easily be converted into a
CUDA kernel\footnote{See
\url{https://github.com/venovako/VecJac/blob/master/cuda/device_code.h}
for the real and the complex kernels in single and double precision.}.
In the lines~\ref{ll:9} and \ref{ll:10} the CUDA integer device
function $\mathtt{min}$ can be used instead of the ternary operators,
and the double precision intrinsics elsewhere instead of the basic
arithmetic operators and functions, e.g., $\mathtt{\_\_fma\_rn}$
instead of $\mathtt{fma}$, $\mathtt{\_\_dsqrt\_rn}$ instead of
$\mathtt{sqrt}$, and $\mathop{\mathtt{\_\_drcp\_rn}}(x)$ instead of
computing $1/x$.

The layout of the input and the output data is the same as in the
vectorized algorithm from \cref{ss:2.4}, apart from the permutation
bit $p$ and $\zeta$, which are returned packed in an integer.  A GPU
thread reads the elements of its matrix $A$ and stores the results at
the offset \texttt{(size\_t)(blockIdx.x) * blockDim.x + threadIdx.x},
where each thread block in a onedimensional execution grid has $64$
threads (two warps).

Since each thread processes its own matrix independently of the other
threads, there are neither intra-block nor inter-block data
dependencies.  A kernel can also be converted into a device function,
to be incorporated into larger kernels, and the data can be read from
and/or be written to the shared instead of the global memory.

\looseness=-1
\Cref{f:SM10.1} shows the timing results (without the data transfers)
of the real and the complex kernels in single and double precision on
an NVIDIA GeForce GTX TITAN X (a Maxwell series GPU, \texttt{sm\_52}
architecture), with the CUDA Toolkit version 11.7.0.  The input
batches consisted of $2^{27}$ Hermitian matrices of order two each.
It is expected that on newer GPU architectures the run times should be
substantially lower.

\begin{figure}[hbtp]
  \begin{center}
    \includegraphics{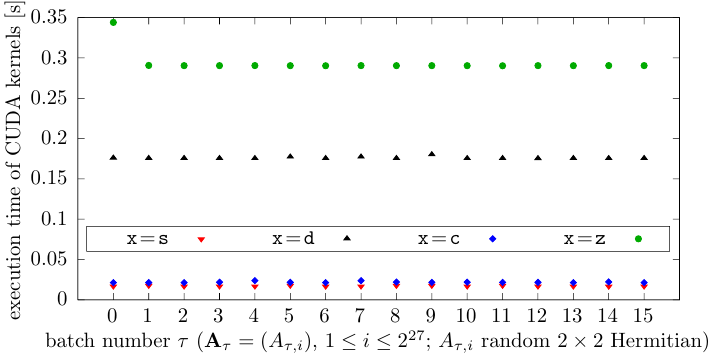}
  \end{center}
  \caption{Run-times of the CUDA batched Hermitian $2\times 2$ EVD kernels.}
  \label{f:SM10.1}
\end{figure}
%
%% \bibliographystyle{siamplain}
%% \bibliography{ms}

%
\end{document}